\documentclass[hidelinks, 12pt, oneside, reqno]{amsart}
\usepackage{dsliheader2}
\usepackage{geometry}
\usepackage{xy}
\usepackage{hyperref}
\title[Laplacian flow on complete asymptotically conical manifolds]{$G_2$-Laplacian flow on complete asymptotically conical manifolds}
\author[Khan]{Ilyas Khan}
\address{Department of Mathematics, Duke University, Durham, NC, USA}
\email{ikhan.maths@gmail.com}
\date{}
\begin{document}

\begin{abstract}
This paper proves the short-time existence of $G_2$-Laplacian flow for closed, asymptotically conical $G_2$-structures on a complete manifold, shows that this solution is unique among closed solutions, and gives a lower bound for the existence time depending on the initial geometry. 
\end{abstract}

\maketitle

\tableofcontents

\vspace{-.3in}

\section{Introduction}

The main result of this paper is the following theorem. 

\begin{thm}\label{theorem main theorem}
    Let $M$ be a complete 7-manifold, and let $\vp_0$ be a closed, asymptotically conical $G_2$-structure on $M$. There exists a solution $\vp(t)$ to the $G_2$-Laplacian flow on a time interval $[0,T)$ with $\vp(0) = \vp_0$. Furthermore, this solution is unique and the maximal existence time $T$ has an explicit lower bound depending only on the maximum of the curvature of the initial $G_2$-structure $\vp_0$.
\end{thm}

In particular, this theorem generalizes existence and uniqueness results for the $G_2$-Laplacian flow proved by Bryant-Xu in \cite{BryantXu} and Lotay-Wei in \cite{LotayWeiLaplacianFlow} for $G_2$-structures on compact manifolds to asymptotically conical initial data on complete manifolds. However, the non-compact setting necessitates the development of a new analytic method inspired by ideas from gluing in order to establish existence.

\subsection{Background and motivation}

The central problem in $G_2$-geometry is to find Riemannian metrics $g$ on $7$-manifolds $M^7$ with holonomy contained in the exceptional Lie group $G_2$. This is equivalent to finding a torsion-free $G_2$-structure (in the sense of the torsion of a $G$-structure) on $M$. $G_2$-structures are represented by positive 3-forms $\vp$ and induce a Riemannian $g_\vp$, a Hodge star $\star_\vp$, and a co-differential $\de_\vp$. The torsion-free condition is equivalent to the $G_2$-structure $\vp$ being closed and co-closed, i.e. $d \vp = 0$ and $\de \vp = 0$. 

The known examples of torsion-free $G_2$-structures have either been constructed via ansatz (e.g. the non-compact examples of Bryant-Salamon in \cite{BryantSalamonExceptionalHolonomy} and Foscolo-Haskins-Nordstrom in \cite{FoscoloHaskinsNordstromALC}) or by singular perturbation methods (e.g. the compact examples of Joyce in \cite{JoyceG2I} and \cite{JoyceG2II} and Kovalev in \cite{KovalevTwistedConnectedSums}). However, this stable of examples does not give much information about the existence of torsion-free $G_2$-structures outside of very specialized settings. 

The Laplacian flow of closed $G_2$-structures is a proposed parabolic approach to the problem of finding torsion-free $G_2$-structures. The Laplacian flow evolves any positive and closed $G_2$-structure $\vp$ in the direction of its Hodge Laplacian, taken with respect to the metric that $\vp$ itself induces. 

\begin{equation}\label{equation laplacian flow intro}
    \left\{
    \begin{array}{ll}
      \dd_t \vp &=   \Delta_{\vp} \vp\\
      \vp(0) &= \vp_0, \;\; d\vp_0 = 0.
    \end{array}
  \right.
\end{equation}

This flow arises as the upward gradient flow for the Hitchin volume functional,
\[
\cV(\vp) := \frac17 \int_M \vp \wedge \star \vp, \;\; \vp \in [\vp_0]_+
\]
which measures the volume of $(M,g_\vp)$ for $\vp$ cohomologous to $\vp_0$ and satisfying a positivity condition (i.e. $\vp$ induces a non-degenerate Riemannian metric $g_\vp$). The critical points of this functional are strict maxima and correspond to torsion-free $G_2$-structures. 

Ideally, the Laplacian flow would seek these maxima and converge to them as $t \rightarrow \infty$. However, this is very far from being known outside a few special cases, especially because of the near-complete ignorance of the phenomenology of finite-time singularity formation. Even so, many basic propoerties of the flow in the compact case have been established, most notably in the proof of the flow's short-time existence on compact manifolds by Bryant and Xu in \cite{BryantXu} and the papers of Lotay and Wei (see \cite{LotayWeiLaplacianFlow}, \cite{LotayWeiStability}, \cite{LotayWeiRealAnalyticity}). 

The theory of the flow is less developed in the non-compact case. Non-compact flows appear naturally in the study of singularities: see, for example, the non-compact solitons of Fowdar in \cite{FowdarS1LaplacianFlow}, Haskins-Nordstrom in \cite{HaskinsNordstromLaplacianSolitons}, and Haskins-Juneman-Nordstrom in \cite{HaskinsJunemanNordstromExpanders}, as well as the space of non-compact shrinkers studied by Li-Ma-Zheng in \cite{LiMaZhengG2Solitons}. In addition, many important examples of $G_2$-manifolds are non-compact, such as the asymptotically conical (AC) Bryant-Salamon manifolds \cite{BryantSalamonExceptionalHolonomy} and the asymptotically locally conical manifolds (ALC) of Foscolo-Haskins-Nordstrom \cite{FoscoloHaskinsNordstromALC}, and it makes sense to study the Laplacian flow against these and similar backgrounds, as well as in the singular setting.

The first step in the study of the flow on non-compact manifolds is to prove that the flow exists, perhaps under a certain restriction on the initial geometry. Theorem 1.1 is a partial answer to this question of existence, along with the useful facts of uniqueness and preservation of the AC condition.

\subsection{Strategy of the proof}

In order to motivate the strategy of the proof, it is illustrative to begin by outlining a ``na\"ive" approach to solving the problem, and to modify or develop this approach as it encounters obstacles. 

\subsubsection{A na\"ive approach}

The main theorem, Theorem \ref{theorem main theorem}, bears a resemblance to W.-X. Shi's existence result for Ricci flows on complete manifolds with bounded curvature tensor in \cite{ShiDeforming}. It thus stands to reason that one might attempt to obtain analogous results to Shi's for the $G_2$-Laplacian flow by similar means. 

In a nutshell, Shi's proof consists of the following steps:
\begin{enumerate}
    \item Decompose the non-compact manifold $M$ into an exhaustive, nested sequence of compact domains $M_i$, where $i \rightarrow  \infty$,
    \item Solve the Ricci flow equation with homogeneous boundary conditions on each compact manifold with boundary $(M_i, \dd M_i)$,
    \item Obtain a priori estimates on these solutions, $g_i(x,t)$, independent of the domain $M_i$. 
    \item Obtain a solution on the full manifold $M$ by taking a subsequential loc-limit of these solutions as $i \rightarrow \infty$.
\end{enumerate}

The first obstacle to generalizing this argument to the Laplacian flow is that the Laplacian-DeTurck flow is parabolic \textit{only} in the direction of closed forms. Thus, the equation can only be solved on domains within a restricted class of forms and thus the standard parabolic existence theory used by Shi cannot be directly brought to bear. To obtain the solutions on the domains, we will instead construct an approximate solution, and use the implicit function theorem to find a true solution nearby. 

As in \cite{BryantXu}, the crucial step is to invert the linearized operator acting on exact forms. Let $\sg(t)$ be a time-dependent family of non-degenerate $G_2$-structures and let $\theta(t) = d\beta(t)$ be a family of exact 3-forms. Then the linearization of the operator at $\sg(t)$ has the following form.
\begin{align*}
    \dd_t \theta + \Delta_{\sg} \theta - d\Phi(\theta) &= \dd_t (d \beta) + \Delta_{\sg} (d\beta) - d\Phi(d\beta) \\
    &= d \big( \dd_t  \beta + \Delta_{\sg} \beta - \Phi(d\beta) \big),
\end{align*}
where $\Phi$ is a linear operator which is algebraic in the torsion of $\sg$. Given a family of exact forms $d\xi(t)$ with potentials $\xi(t)$, solving the linearized equation
\begin{equation}\label{equation linearized exact demo}
\dd_t \theta + \Delta_{\sg} \theta - d\Phi(\theta) = d \xi
\end{equation}
follows from solving 
\begin{equation}
\dd_t  \beta + \Delta_{\sg} \beta - \Phi(d\beta) = \xi.
\end{equation}
This latter equation can be solved on a domain with prescribed boundary values by using standard methods in the theory of parabolic equations. However, the first derivatives $\na_i \beta$ will appear in the boundary values of $\theta$. These derivatives of $\beta$ in turn depend on $\xi$. Thus, any estimates on $\theta$ obtained as a consequence of \eqref{equation linearized exact demo} depend on the choice of the potential, $\xi$. 

By necessity, then, one must first carefully choose a potential and solve the linearized equation on the level of potentials. 

\subsubsection{Solving the potential equation}\label{subsection solving potential}

This sketch points to the core difficulty in our approach, which is to find a way to choose potentials $\xi_i$ for a sequence of exact forms $d \xi_i$, defined on an increasing sequence of domains $M_i$ which exhaust the manifold $M$. In particular, the Sobolev norms of the $\xi_i$'s should be controlled by the Sobolev norms of the $d\xi_i$'s in a \textit{controlled} way, i.e. independently of $i \rightarrow \infty$.

It is at this point that we must strengthen the geometric assumptions on the non-compact ends of the manifold $M$, since a pointwise bound on the curvature is not enough to guarantee the existence of a sequence of potentials $\{\xi_i\}_i$ with derivatives bounded independently of $i$. More concretely, in our desired norms, we may see the following behavior:
\[
\|\xi_i\| \le C_i \|d \xi_i\|, 
\]
where $C_i \rightarrow \infty$ as $i \rightarrow \infty$. Filtered through the inverse function theorem, this failure of uniformity would ultimately lead to a sequence of ``domain solutions" on time-intervals $[0, \eps_i]$, with $\eps_i \rightarrow 0$, resulting in a limit which does not solve the equation on any positive time-interval.

The asymptotically conical (AC) condition obtains for us the geometric control that allows us to choose this sequence of potentials with good constants, at the price of solving the equation only up to an element of an ``approximate kernel".

\subsubsection{A sequence of almost solutions and global existence}
After solving the linearized equation up to these error terms and applying the implicit function theorem, one obtains a solution of potentials representing ``almost solutions" on the exhaustive sequence of domains $(M_i, \dd M_i)$. It remains to show that if we take the limit of these ``almost solutions" in the $H^k_{loc}$-topology, the error terms in the domain solutions will vanish in the limit. This can be shown in our set-up to be true by a series of arguments using the Fredholm theory of the Dirac operator $d + \de$ and the Hodge Laplacian $\De$ in weighted Sobolev spaces. This yields the basic short-time existence of the flow on the complete manifold $M$, and the further properties of this flow are proved by more elementary arguments which are described in the next section.

\subsection{Outline of the paper}

\subsubsection{Outline of \S \ref{section preliminaries}}

The ``Preliminaries" section sets up the problem, beginning by recalling known results but moving on to outline the details of the framework specific to this problem. The first part of the section reviews well-known material about $G_2$-geometry, the asymptotically conical condition, Laplacian flow and the Laplacian-DeTurck flow discovered by Bryant and Xu in \cite{BryantXu}. 

The remainder of the section is devoted to defining the non-linear map whose inverse will yield the desired solution on the truncated domains $M_R$. In order to define this map, it is necessary to introduce the various spaces it passes through. First, the weighted H\"older and Sobolev spaces are introduced for both complete manifolds (\S \ref{section weighted sobolev holder}) and manifolds with boundary (\S \ref{section function spaces domains}), along with the Sobolev embedding theorems. 

Next, in order to deal with the necessity of solving the Laplacian flow equation in the space of exact forms, the Hodge-Morrey decomposition on manifolds with boundary is introduced in \S \ref{section hodge morrey}. 

The last spaces to be introduced are the spaces of functions with vanishing boundary conditions in \S \ref{section boundary conditions}. It is necessary to restrict our range to these spaces so that the correct boundary conditions for the parabolic existence theory are satisfied, and the right-inverse to the linearized map is well-defined. 

Now, given the approximate kernel defined in Appendix \ref{section estimates potential}, we can define a map from spaces of exact and co-exact potentials into the $L^2_\la$-orthogonal complement of the approximate kernel, 
\[
\cP : H^{r,s}_{0,\la+3} \Omega^2_D(N \times I) \oplus H_{0,\la+1}^{r-1,s-2}\Om^4 (N \times I)  \rightarrow H_{0,\la}^{r-1,s-3}\sU^3(N \times I).
\]
See \eqref{equation third component mapping} for the detailed formula for this non-linear mapping.

Finally, \S \ref{section IFT setup} states the quantitative implicit function theorem, calculates the linearization of $\cP$, and establishes the $R$-independence of the Lipschitz constant $\cP - D\cP_0$. All that remains to be shown is the existence of a right-inverse to $D\cP_0$ which is bounded independent of $R$, which will be established in \S \ref{section short time existence IVT argument}. 

\subsubsection{Outline of \S \ref{section approximate solution}}

This section contains the construction of the approximate solution around which the Laplacian flow equation is linearized. The basic idea is the standard technique of finding a formal series expansion which solves the equation to sufficiently high order at the initial time, $t = 0$. Then, after the expression for this approximate solution is obtained, it is shown that it lies in the appropriate weighted spaces and yields a \textit{bona fide} $G_2$-structure. Finally, Corollary \ref{corollary nearby vanishing short time} shows that given some Sobolev space neighborhood of the inhomogeneous term, this neighborhood contains a family of exact $3$-forms vanishing uniformly in an interval around $t=0$. Lemma \ref{lemma approx extension}, with a view to proving long-time existence in \S \ref{section long time behavior}, establishes all the aforementioned properties for approximate solutions extending true solutions on $[0,T_0]$ at $t = T_0$. 

\subsubsection{Outline of \S \ref{section short time existence IVT argument}}

In this section, the linearization of the map $\cP$ at the approximate solution constructed in \S \ref{section approximate solution} is computed and a right-inverse to this linearization is constructed, and then bounded in norm independently of the radial parameter $R$. The first part, \S \ref{subsection existence outline}, outlines the global existence argument in broad strokes, leaving the details to the later sections \S \ref{subsection parabolic existence}-\S \ref{subsection existence everywhere}. Assuming the results of these later sections and the appendices, Theorem \ref{theorem existence of inverse on domains} executes the implicit function theorem argument and establishes the existence of ``almost solutions" on domains, for which their image under $\cP$ vanishes on an interval $[0,\ve]$ with $\ve$ independent of $R$. 

A well-defined right-inverse is constructed in \S \ref{subsection parabolic existence}. The right inverse is defined to be the composition of the choice of potential made in Appendix \ref{section estimates potential} and the Lax-Milgram solution to a linear parabolic equation on potentials with homogeneous boundary conditions. The uniqueness of the Dirichlet potential established in Appendix \ref{section estimates potential} and the uniqueness from the Lax-Milgram argument tells us that this right-inverse is a well-defined, injective linear map.  

In order to turn the Schauder estimates for the linearized equation into estimates on the operator norm of the right-inverse, apriori estimates on the $L^2$- and $C^0$-norms of the solutions to this equation are obtained in \S \ref{subsection a priori estimates}. These estimates are proved using standard parabolic maximum principle arguments and the weighted Sobolev embedding theorem.

To show that the operator norm of the right-inverse is uniformly bounded along the entire sequence of domains $M_{R_i}$ (where $R_i \rightarrow \infty$), \S \ref{subsection uniform boundedness right inverse} puts together the arguments of \S \ref{subsection existence outline}-\S \ref{subsection a priori estimates} and Appendices \ref{section linear parabolic appdx}-\ref{section estimates potential}. 

In \S \ref{subsection existence outline}-\ref{subsection uniform boundedness right inverse}, we have applied the implicit function theorem to solve the Laplacian flow equation on the domains $M_{R_i}$ up to the ``approximate kernel" and a co-exact correction term, and have obtained a sequence $(\beta_i, \omega_i) \in \Omega^2(M_{R_i}) \times \Omega^4(M_{R_i})$ such that $\vp_i := d\beta_i + \vp_0$ and
    \[
    \partial_t \vp_i + \Delta_{\vp_i}\vp_i - \cL_{V(\vp_i)} \vp_i = \delta \omega_i + \psi_i, \qquad t \in [0, \epsilon]
    \]
where $\psi_i$ is an approximate kernel element and $\epsilon >0$ does not depend on $i$. In \S \ref{subsection existence everywhere}, we take the limits 
\[
    \beta_i \xrightarrow{H^k_{loc}} \beta,\; \omega_i \xrightarrow{H^k_{loc}} \omega, \; \psi_i \xrightarrow{H^k_{loc}} \psi,
\]
as $i \rightarrow \infty$. It can be shown that $\psi$ is $L_{\lambda}^2$-orthogonal to $\mathrm{Im}(d + \delta)$. By Hodge theory for weighted spaces on AC manifolds (see \cite{KarigiannisLotay}, \cite{Lockhart1987}), within $\mathrm{Im}(d + \delta)$, the exact and co-exact forms are $L^2$-orthogonal. Thus, if $\vp := d\beta + \vp_0$
    \[
    \partial_t \vp + \Delta_{\vp}\vp - \cL_{V(\vp)} \vp - \delta \omega - \psi = 0, \qquad t \in [0, \epsilon],
    \]
    then $\vp$ is a solution to Laplacian-DeTurck flow,
    \[
    \partial_t \vp + \Delta_{\vp}\vp - \cL_{V(\vp)} \vp = 0, \qquad t \in [0, \epsilon]
    \]
where the existence time $\epsilon > 0$ depends \textit{only} on the geometry of the initial data $\vp_0$. This solution can be transformed into a solution to Laplacian flow after modification by diffeomorphisms. 

\subsubsection{Outline of \S \ref{section uniqueness}}

Since the solution obtained in \S \ref{section preliminaries} - \S \ref{section short time existence IVT argument} is merely the limit of a subsequence, it is not necessarily unique. Its uniqueness among closed Laplacian flows will be established via the energy methods of Kotschwar in \cite{kotschwarEnergy}. In \cite{LotayWeiLaplacianFlow}, Lotay and Wei define an energy quantity for Laplacian flow in the style of \cite{kotschwarEnergy} which was sufficient to establish forward uniqueness in the compact case under a bounded curvature condition. Given two solutions $\vp(t), \tvp(t)$ and an auxiliary function $\eta$, a slight modification of their energy,
\[
\cE (t) = \int_M \big( |\vp - \tvp|^2 + |g-\tg|^2 + |\na - \tna|^2 + |T - \widetilde{T}|^2 + |\na T - \tna \widetilde{T}|^2 + |\Rm - \widetilde{\Rm}|^2 \big) e^{-\eta} d\mu,
\]
can be applied in the non-compact setting. 

The quantities in the integrand form a virtually parabolic PDE-ODE system which can be shown to imply that 
\[
\cE'(t) \le C \cE (t).
\]
Thus, if $\vp(0) = \tvp(0)$, then $\cE(t) \equiv \cE(0) = 0$; that is, the solution is unique. Note that this uniqueness result holds in the presence of bounded curvature only, and does not depend on the AC condition. 

\subsubsection{Outline of \S \ref{section long time behavior}} 

The argument of \S \ref{section preliminaries} - \S \ref{section short time existence IVT argument} implies that there is an $\eps > 0$, which depends in a complex way on the geometry of the initial condition $\vp_0$, such that a solution to Laplacian flow exists for $t \in [0, \eps]$. However, the flow may continue to exist beyond $t = \eps$, which raises two related questions: 
\begin{itemize}
    \item[(a)] Must the flow continue to satisfy the AC condition?
    \item[(b)] Can the existence time be bounded from below in a more explicit manner?
\end{itemize}
In this section, both of these questions are answered affirmatively. 

To answer (a), we assume a curvature bound $|\Rm| \le K_0$ and a maximal time $T$ up to which the AC condition is preserved. Then, for arbitrary $T' < T$, the short-time existence argument of \S \ref{section preliminaries} - \S \ref{section short time existence IVT argument} can be reiterated with initial condition $\vp(T')$ to continue the flow (as an AC flow) for a short time interval $[T', T' + \eps']$. It can be proved that $\eps'$ can be bounded below depending \textit{only} on $K_0$ and $T$. Setting $T' := T - \tfrac{\eps'}{2}$ contradicts the maximality of $T$, and we conclude that as long as the curvature remains bounded, the flow remains AC.

The answer to (b) follows from the preservation of the AC condition. Since the curvature achieves an interior maximum, as long as the AC condition is preserved, the non-compact maximum principle can be used as in \cite[\S 4]{LotayWeiLaplacianFlow} to prove the doubling time estimate, Corollary \ref{corollary doubling time estimate}. This estimate gives an explicit lower bound for the existence time in terms of the curvature. 

\subsubsection{Outline of Appendix \ref{section linear parabolic appdx}}

Appendix \ref{section linear parabolic appdx} collects a number of PDE results that are used crucially throughout the paper. This collection is a mix of restatements of standard results in elliptic and parabolic PDE and estimates adapted to our particular setting, which require more work to establish. 

Of particular note are the Schauder estimates up to the boundary for elliptic and parabolic equations in weighted spaces. These estimates are proved on the interior using a standard scaling argument originally used by Bartnik in \cite{BartnikMass} and up to the boundary using new arguments, depending on the boundary conditions. There are two theorems of this variety: a) Proposition \ref{proposition weighted schauder estimates} proves apriori estimates up to the boundary for solutions to parabolic equations with homogeneous boundary conditions; and b) Proposition \ref{proposition lopatinski shapiro condition} proves apriori estimates for solutions to certain elliptic equations which satisfy the Lopatinski-Shapiro condition at the boundary. In particular, the latter proposition and its corollary are used to establish the regularity of the left- and right-eigenbases corresponding to the singular value decomposition of the inverse of the Hodge Laplacian. 

\subsubsection{Outline of Appendix \ref{section estimates potential}}

This appendix, which addresses the issue described in \S \ref{subsection solving potential} around the choice of a potential 2-form, is the technical heart of the paper.  To choose an ``estimable" potential for an exact form, one solves a second-order elliptic boundary value problem for the \textit{Dirichlet potential} $\phi_D$ of a form $\eta$ defined on $N$, a compact domain with boundary. 
\begin{align}\label{equation intro dirichlet BVP}
    \begin{aligned}
    \Delta \phi_D &= \eta \\
    \bt \phi_D = 0 \;\;\;&\text{ and }\;\;\; \bt \de \phi_D = 0
    \end{aligned}
    &&
    \begin{aligned}
        &\text{ on }N\\
     &\text{ on }\dd N. 
    \end{aligned}
\end{align}
If $d\xi$ is the exact part of $\eta$ in the Hodge-Morrey decomposition, the potential $\xi$ may be obtained by co-differentiating, i.e. $\xi := \de \phi_D$. Thus, apriori estimates for the solutions of \eqref{equation intro dirichlet BVP} become estimates on the potential $\xi$. 

On each domain $M_{R_i}$, the boundary value problem \eqref{equation intro dirichlet BVP} is solvable up to a finite dimensional kernel of harmonic fields (closed and co-closed forms) subject to a certain Dirichlet boundary condition. However, it does not suffice to solve the equation perpendicular to this kernel--as $R_i \rightarrow \infty$, the norm $\|\De^{-1}\|$ may grow without bound. Intuitively, the forms responsible for the blowup of $\|\De^{-1}\|$ live in a finite dimensional space and correspond to ``small eigenvalues" of $\De$. Inspired by a similar idea used in gluing constructions (see, for example \cite{KapouleasCMC1} and \cite{BreinerKapouleasCMC1}), we only invert $\De$ up to an ``approximate kernel" that collects these forms corresponding to small eigenvalues, in addition to the Dirichlet harmonic fields. 

This simple picture is greatly complicated by the fact that in the weighted Sobolev spaces, the Hodge Laplacian $\De$ is \textit{not} self-adjoint. Thus, one is forced to consider the spectrum of the operator $\De^*\De$, the Hodge bi-Laplacian, where $\De^*$ is the adjoint of $\De$ in the weighted spaces. These eigenvalues and eigenforms give a singular value decomposition for a Green's operator corresponding to $\De$, with singular values $\{\sg_i\}$ and left- and right-eigenbases of forms $\{\phi_i\}$ and $\{\psi_i\}$ such that 
\[
\De \psi_i = \sg_i^{-1} \phi_i. 
\]
The details of this singular value decomposition can be found in \S \ref{section SVD}. 

Now, one must determine which singular values are ``small" in order to collect eigenforms into the approximate kernel. This cutoff limit is chosen by comparing to the bottom of the spectrum of $\De^*\De$ on the non-compact AC manifold $(M,g_0)$, whose Fredholm theory for weighted spaces is described in \cite{KarigiannisLotay} and \cite{Lockhart1987}. This is similar to the way that in gluing problems, as the small parameter goes to zero, the spectra of the Jacobi operator on the various pieces approach that of the larger spaces of which they are a part (e.g. the way in which the catenoidal necks in \cite{BreinerKapouleasCMC1} ``approach" the catenoid as the parameter $\tau \rightarrow 0$).

In contrast to the case of gluing, however, we have little control over the approximate kernel. It will always be finite-dimensional, but this dimension is uncontrolled as $R_i \rightarrow \infty$. Furthermore, we do not have a geometric characterization of the forms in the kernel which can help us control their behavior. Thus, we cannot employ a ``geometric principle" \textit{\`{a} la} Kapouleas in \cite{KapouleasCMC1} to counteract the influence of the approximate kernel. Instead, \S \ref{section est approx ker} examines the limiting behavior of bounded sequences of forms in the approximate kernel to show that they interact nicely with the Hodge decomposition on the full non-compact manifold $(M,g_0)$. In particular, Corollary \ref{corollary small eigenvalues become harmonic} shows that the $H^k_{loc}$-limit of a $L^2_\la$-bounded sequence of forms in the approximate kernel will lie in the $L^2_\la$-orthogonal complement to the Dirac operator, $d + \de$. 

Proposition \ref{proposition boundedness of restricted forms} establishes the uniform boundedness of the Green's operator for the Dirichlet potential, restricted to the $L^2_\la$-orthogonal complement of the approximate kernel.  This fact, in addition to the limiting properties of the approximate kernel elements mentioned above, shows that our choice of gauge has the requisite properties for the Hodge-theoretic arguments of Proposition \ref{proposition solution on whole space up to coexact} and Theorem \ref{theorem solution is exact solution} to go through.

\subsection{Open questions and further directions}

The results of this paper and the techniques employed to prove them raise a number of interesting questions and possibilities for further inquiry. 

Many elements of the method of proof are not specific to the $G_2$-setting, and could conceivably be applied to a number of geometric problems. The central issue is that on each domain, the demand that the solution be exact introduces a gauge indeterminacy, which in turn necessitates a judicious gauge selection. While the well-known Hodge theory illuminates the case of the exterior derivative $d$, it is conceivable that its role can be taken by any operator with a similar Fredholm theory on weighted spaces. Thus, while a more obvious next step might be to investigate other flows of exact or co-exact forms, the technique might be profitably adapted to geometric PDEs with similar gauge indeterminacies on an AC (or similar) background.

Similar problems for other flows (e.g. \cite{LMCFwithsings} and \cite{SuACLMCF} for the Lagrangian mean curvature flow) have been dispensed with by using heat kernel methods to obtain crucial estimates on whole non-compact background at once. To what extent can the result of Theorem \ref{theorem main theorem} be proved using similar methods? In particular, is there an underlying equivalence between the two approaches? What are the advantages and disadvantages of each?

While it is not proved explicitly in the paper, it is possible to see the continuity of the solutions obtained in Theorem \ref{theorem main theorem} under perturbation of the initial condition in the weighted spaces. However, it would be interesting to investigate the continuity properties of these solutions under larger perturbations of the initial conditions: for example, under the deformation of the link of the asymptotic cone to $\vp_0$. 

The dynamic stability of torsion-free $G_2$ structures on compact manifolds along the Laplacian flow was proved by Lotay and Wei in \cite{LotayWeiStability}. It is natural to investigate if the same property holds for torsion-free $G_2$ structures on non-compact backgrounds, and in particular for AC $G_2$-structures, which class includes the important examples of Bryant and Salamon in \cite{BryantSalamonExceptionalHolonomy}. 

Finally, a natural next step is to try to prove the analogues of Theorem \ref{theorem main theorem} for other asymptotic geometries (e.g. asymptotically locally conical, asymptotically cylindrical, and conically singular) by using largely the same strategy and techniques. The present proof for AC fundamentally relies on the existence of a well-developed Fredholm theory for the relevant operators, $d + \de$ and $\De$, on weighted spaces, and not necessarily on asymptotic conicality in particular. It would be particularly interesting to study the ``mirror image" of the AC setting and to see if the techniques developed in this paper could be used to define a Laplacian flow of $G_2$-structures with conical singularities.  

\subsection*{Acknowledgments}

The author was supported by the National Science Foundation through Mathematical Sciences Postdoctoral Research Fellowship DMS-2303187. The author would also like to thank Mark Stern, Jason Lotay, Sigurd Angenent, Mark Haskins and Robert Bryant for valuable discussions.  

\subsubsection*{AI Statement} Literature search was sped up by AI, particularly the AI summary given by Google Search. The creation BibTeX entries was partly automated. Style suggestions for this acknowledgment section were received from AI. In addition, an unsuccessful approach to a particular proposition was eliminated by using AI to estimate difficult asymptotic integrals. 

Outside of the minor uses stated above, AI was not involved in the writing of this paper on the level of ideas, computations, or even prose. For better or worse, it is a product of the human mind. 

\section{Preliminaries}\label{section preliminaries}

\subsection{$G_2$-structures}

Suppose $M$ is a smooth $7$-manifold, possibly with boundary. If $M$ is orientable and spinnable, then $M$ admits a $G_2$-structure, which is a reduction of the structure group of the frame bundle to the group $G_2\subset GL(7)$. This $G_2$-structure is equivalent to a choice of ``non-degenerate" $3$-form $\vp \in \Omega^3(M)$, called a $G_2$-form. Associated to a $G_2$-form, there must be a family of isomorphisms $\io: T_x M \to \bR^7$, smoothly varying in $x \in M$, such that $\io^* \vp_0 = \vp_x$, where $\vp_0$ is the ``standard" $3$-form on $\bR^7$. The standard $3$-form is defined as follows: associate $\bR^7$ to the imaginary octonions $\Im(\bO)$ and let
\[
\vp_0(u,v,w) := \langle uv, w\rangle,
\]
where $\langle \cdot, \cdot \rangle$ is the Euclidean inner product and $uv$ is multiplication of $u$ and $v$ in the octonions.

A $G_2$-structure $\vp$ on $M$ induces a metric $g_{\vp}$, which is defined by the formula
\[
g_{\vp}(u,v) \Vol_{\vp} = \tfrac{1}{6}\io_u \vp \wedge \io_v \vp \wedge \vp,
\]
which holds for any vector fields $u,v$ on $M$. Furthermore, a Hodge star $\star_{\vp}$ and thus a co-differential $\de := (-1)^k \star_\vp d \star_\vp$ can be defined via the metric $g_\vp$ and the orientation on $M$. The torsion of a $G_2$-structure is given by a two-tensor $T_{ij}$ satisfying
\begin{equation}\label{equation nabla phi and torsion}\na_i \vp_{jkl} = T_{i}^{\;m}(\star_{\vp} \vp)_{mjkl}.\end{equation}

Henceforth, every $G_2$-structure in this paper will be assumed to be closed, i.e. the $G_2$-form satisfies $d\vp = 0$. If the $G_2$-structure is closed, then $T_{ij}$ is a two-form which satisfies the relation
\begin{equation}\label{equation R = -|T|^2}
    R = -|T|^2,
\end{equation}
where $R$ is the scalar curvature of the metric $g_{\vp}$. 


The following lemma from \cite{HKP} records scaling and pull-back properties of geometric quantities associated to a $G_2$-structure $\vp$.

\begin{lem}\cite[Lemma 2.1]{HKP}\label{lemma scaling of metric}
If $\vp$ is a $G_2$-structure on $M$, $\Psi:M \rightarrow M$ is a diffeomorphism, and $\lambda > 0$ is a constant, consider the pulled-back and rescaled 3-form $\lambda \Psi^* \vp$, which is also a $G_2$-structure on $M$. Then, the corresponding metric $g_{\lambda \Psi^* \vp}$ is given by $\lambda^\frac{2}{3}\Psi^* g_{\vp}$. Moreover, 
the following identities are satisfied:
\begin{enumerate}
\item \textit{(Hodge Laplacian)}
\begin{equation}\label{equation scaling hodge laplacian}
    \Delta_{\la \Psi^* \vp}\la \Psi^* \vp = \la^{\frac{1}{3}}  \Psi^* (\Delta_{\vp} \vp)
\end{equation}
\item \textit{(Hodge Dual)}
\begin{equation}
    \star_{\la \Psi^* \vp} (\la \Psi^* \vp) = \la^{\frac{4}{3}} \Psi^* (\star_\vp \vp)
\end{equation}
\item \textit{(Torsion)}
\begin{equation}
    T_{\la \Psi^* \vp} = \la^{\frac{1}{3}} \Psi^* T_{\vp}
\end{equation}
where $T$ denotes the torsion tensor regarded as a $2$-tensor. 
\end{enumerate}
\end{lem}

\subsection{Laplacian flow}

A time-dependent family $\{\vp_t\}$ of closed $G_2$-structures parametrized by $t\in [0,T)$ is said to satisfy Laplacian flow if 
\begin{equation}\label{equation laplacian flow}
     \frac{\partial \vp}{\partial t} = \Delta_{\vp} \vp,
\end{equation}
where $\Delta_{\vp} \vp$ is the Hodge Laplacian of $\vp$.

\subsubsection{Evolution equations}

The evolution of the metric $g_t$ associated to closed 3-forms $\vp_t$ flowing by Laplacian flow is given by the following equation.
\begin{equation}\label{equation evolution metrics}
         \frac{\dd g_{ij}}{\dd t} = -2 \Ric_{ij} - \frac{2}{3}|T|^2 g_{ij}- 4 T_i^{\;  l}T_{lj} := \eta_{ij},
\end{equation}
where $T_{ij}$ is the torsion of the $G_2$-structure. The evolution of the volume element is thus given by
\begin{equation}\label{equation evolution volume}
         \frac{\dd }{\dd t}d\mu =  - \frac{2}{3}R d\mu. 
\end{equation}
We record the evolution of the Christoffel symbols and the curvature tensor, as expressed in \cite{LotayWeiLaplacianFlow}.
\begin{align}
    &\frac{\dd }{\dd t}\Ga^k_{ij} = \frac12 g^{kl}(\nabla_i \eta_{jl} + \nabla_j \eta_{il} - \nabla_l \eta_{ij}),
    \label{equation evolution christoffel symbols} \\
    &\frac{\dd }{\dd t} T = \Delta T + \Rm * T + \Rm * T * (\star \vp) + \nabla T * T * \vp + T * T*T,
   \label{equation evolution torsion}\\
   &\frac{\dd }{\dd t} \Rm = \Delta \Rm + \Rm * \Rm + \Rm * T * T + \nabla^2 T * T + \nabla T * \nabla T.
   \label{equation evolution curvature}
\end{align}
Note that in the last two equations, we write the right-hand side schematically, where $*$ indicates a metric contraction of tensors. In practice, the particular contraction does not particularly matter, and the schematic form holds all the information we will eventually need.

The following quick consequence of the metric evolution equation \eqref{equation evolution metrics} will be useful later on. 

\begin{cor}
    Given a flow $\vp(t)$ evolving by \eqref{equation laplacian flow} for $t \in [0,T]$ for which the associated flow of metrics $g(t)$ has curvature tensor bounded in norm by $K_0 > 0$, there exists a constant $c(K_0) > 0$ such that 
    \begin{equation}\label{equation compare metrics bdd K}
    e^{-c(K_0)t} g(0) \le g(t) \le e^{c(K_0)t} g(0).
    \end{equation}
\end{cor}

\begin{proof}
    First note that the identity $|T|^2 = -R$ implies that $\eta$ (defined in \eqref{equation evolution metrics}) can be bounded by
    \[
    |\eta|_{g(t)} \le C|\Rm|_{g(t)} \le CK_0.
    \]
    Then, the argument of \cite[p. 225]{ShiDeforming} using Gr\"onwall's inequality can be applied to prove \eqref{equation compare metrics bdd K}.
\end{proof}

\subsubsection{Laplacian-DeTurck flow}

Let $\vp$ and $\tvp$ be two closed, cohomologous $G_2$-structures. We will define the Laplacian-DeTurck vector field $V(\vp)$ against the background $\tvp$ as in \cite{BryantXu} and \cite{LotayWeiStability}. Let $\nabla$ and $\tilde \nabla$ be the Levi-Civita connections associated to $g_{\vp} (:= g)$ and $g_{\tvp} (:= \tg)$ respectively. Consider $S := \nabla - \tilde \nabla \in \Ga (TM \otimes S^2T^*M)$. Let
\begin{equation}\label{equation def deturck}
V(\vp) := - 5V_1(\vp) - V_2(\vp),    
\end{equation}
where 
\[
V_1(\vp) = \frac17 g^{pq}S^i_{pq}e_i, \;\;\; V_2(\vp) = 2A(g^{kj}S^i_{ik}e_j + 5V_1(\vp)),
\]
and $A$ is a representation theoretic constant calculated in \cite{BryantXu}. Note that when $\vp = \tvp$, $S$ vanishes and thus $V(\tvp) = 0$. By the work of \cite{BryantXu} and borrowing the notation of \cite{LotayWeiStability}, the linearization of the operator $\cA_{\tvp}(\vp) = \Delta_\vp \vp + \cL_{V(\vp)} \vp$ at $\hat \vp$ applied to a closed variation $\theta = \dd_s|_{s=0} \vp_s$, $\vp_0 = \hat \vp$ is
\[
D_\vp \cA_{\tvp}(\tvp) \theta = -\De_{\hat \vp} \theta + d\Phi (\theta),
\]
where $\Phi(\theta)$ is an algebraic linear operator on $\theta$ with coefficients depending on the torsion of $\hat \vp$. The sign on the Hodge Laplacian in this linearization reveals that the operator $\Delta_\vp \vp + \mathcal{L}_{V(\vp)}\vp$ is elliptic (in the direction of closed forms). Therefore, we can consider the Cauchy problem for the \textit{Laplacian-DeTurck flow}, a parabolic flow of 3-forms on $M$.
\begin{equation}\label{equation laplacian deturck boundary}
    \left\{
    \begin{array}{ll}
      \dd_t \vp &=   \Delta_{\vp} \vp + \mathcal{L}_{V(\vp)}\vp\\
      \vp(0) &= \vp_0
    \end{array}
  \right.
\end{equation}
where $t \in [0,T)$ for some $T>0$ and where $M$ may have non-empty boundary $\dd M$. It can be noted that Cartan's formula implies that $\mathcal{L}_{V(\vp)}\vp = d(V\lrcorner \vp)$. This fact, combined with the closedness of $\vp$ implies that the right-hand side of the first equation in \eqref{equation laplacian deturck boundary} is an exact form.

\subsection{Asymptotically Conical $G_2$-structures}

In this paper, we will consider flows of asymptotically conical (AC) $G_2$-structures. In this section, we recall basic facts about AC $G_2$-structures following the exposition of \cite{HKP}, where proofs of the following statements may be found.

Let $(\Sigma^6, g_\Sg)$ be a closed Riemannian $6$-manifold, and for $a \ge 0$, let $\cC^\Sg_a := (a, \infty) \times \Sg$ denote the product of smooth manifolds. A \textit{closed $G_2$-cone} over $\Sg$ is the $7$-manifold $(\cC^\Sg_0, \vp_C)$ equipped with a closed $G_2$-structure $\vp_C$, such that the associated metric is $g_{C} = dr^2 + r^2g_\Sg$ and such that $\rho_{\la}^* \vp_C = \la^3 \vp_C$, where $\rho_{\lambda}: \cC^\Sg_0 \rightarrow \cC^\Sg_0$ is the \textit{dilation map} given by $\rho_\lambda(r, \sigma) := (\lambda r, \sg)$ for $\lambda > 0$.
\begin{defn}\label{definition asymptotically conical}
Let $V$ be an end of $M$. We say that $(M, \vp)$ is \emph{asymptotic to the $G_2$-cone $(\cC^\Sg_0, \vp_C)$ along $V$} if, for some $r_0 >0$, there is a diffeomorphism $\Psi : \cC^\Sg_{r_0} \rightarrow V$ such that $\la^{-3} \rho_\la^* \Psi^* \vp \rightarrow \vp_C$ as $\la \rightarrow \infty$ in $C^3_{\text{loc}}(\cC^\Sg_{r_0}, g_C)$.
\end{defn}

\begin{lem}[\mbox{\cite[Lemma 3.1, Corollary 3.2, Corollary 3.3]{HKP}}]\label{lemma decay estimates on AC G2 forms}
Let $(M, g)$ be a Riemannian manifold, $V$ an end of $M$, and $\Psi : \cC^\Sg_a \rightarrow V$ a diffeomorphism for some $a > 0$. For any nonnegative integer $k$ we say that the property $(AC_k)$ holds (i.e. that $\vp$ converges to $\vp_C$ to order $k$) if
\begin{equation}\label{equation ac condition order k}
    \lim_{\la \rightarrow \infty} \la^{-3} \rho_\la^* \Psi^* \vp = \vp_C \;\; \text{ in } \;\;C^k_{loc}(\cC^\Sg_0, g_C). \tag{$AC_k$} 
\end{equation}
Then, for $b > a$,
\begin{itemize}
    \item[(a)]\eqref{equation ac condition order k} implies that
    \[\lim_{b \rightarrow \infty} b^l || \nabla^{l}_{g_C} (\Psi^*\vp - \vp_C)||_{C^0(\cC^\Sg_b, g_C)} = 0
    \]
    for each $l =0,1,2,\ldots, k$.
    \item[(b)] \eqref{equation ac condition order k} implies that 
    \[ \lim_{\la \rightarrow \infty} \la^{-2} \rho_{\la}^{*} \Psi^* g = g_{C} \;\; \text{ in } \;\;C^k_{loc}(\cC^\Sg_0, g_C),\]
    which holds if and only if
    \[\lim_{b \rightarrow \infty} b^l || \nabla^{l}_{g_C} (\Psi^*g - g_C)||_{C^0(\cC^\Sg_b, g_C)} = 0
    \]
    for each $l =0,1,2,\ldots, k$.
    \item[(c)] If ($AC_0$) holds, then the metrics $\Psi^*g$ and $g_C$ are uniformly equivalent on $\overline{\cC^\Sg_b}$ for any $b>a$, and, for each $\eps >0$, there exists $b > a$ such that, for $(r, \sg) \in \cC^\Sg_b$, 
    \begin{equation}
        (1- \eps)|r-b| \le \bar{r}_b(r, \sg) \le (1+ \eps) |r-b|,
    \end{equation}
    where $\bar{r}_b(x) := d_{\Psi^*g}(x, \partial \cC^\Sg_b)$.
    \item[(d)] If ($AC_2$) holds, then for any $b >a$, there exists a constant $K = K(b, g_{\Sigma}) > 0$ such that
    \begin{equation}
        \sup_{x \in \cC^\Sg_b} (\bar{r}_b^2(x) + 1) |\Rm(\Psi^*g)|_{\Psi^*g}(x) \le K.
    \end{equation}
\end{itemize}
\end{lem}

It will be necessary to apply the Fredholm theory on non-compact manifolds developed by Lockhart and McOwen in \cite{LockhartMcOwen1985} and \cite{Lockhart1987}. Thus, it must be shown that asymptotically conical $G_2$-structures defined in the manner of Definition \ref{definition asymptotically conical} give rise to metrics that are ``admissible" in the sense of Lockhart in \cite{Lockhart1987}. 

A slight complication is that all admissible metrics in \cite{Lockhart1987} are smooth and satisfy \eqref{equation ac condition order k} for all $k \in \bN$, and therefore initial data of lower regularity will not be admissible. However, given an AC $G_2$-structure that satisfies \eqref{equation ac condition order k} for $k > 2$, we can mollify the initial condition and obtain an admissible $G_2$-structure. As the mollifier approaches the delta function, the resulting flows will converge in the local H\"older/Sobolev topology to a flow starting at the original initial condition. 

Before admissibility, the asymptotic translation invariance of a metric must first be defined. 

\begin{defn}\label{definition asymptotically translation invariant}\cite[Definition 2.1]{Lockhart1987}
Suppose $h_\infty$ is a translation invariant metric with covariant derivative $\na_{h_\infty}$. A metric $h$ is asymptotic to $h_\infty$ if for each $k \in \bN$
\begin{equation}\label{equation asymp trans invariance}
    \lim_{z \rightarrow \infty} \sup_{\sg \in \Sg} \big\| \na_{h_\infty}^k h (\sg, z) -  \na_{h_\infty}^k h_\infty (\sg) \big\|_{h_\infty} = 0.
\end{equation}
\end{defn}

Note that metrics that are asymptotic to cylindrical metrics are asymptotically translation invariant. Cylindrical $G_2$-structures are of the form 
\[
dz \wedge \om + \Om, 
\]
where $r$ is the radial variable and $\om$ and $\Om$ are a 2-form and a 3-form on $\Sg$, respectively, and give rise to metrics of the form
\[
dz^2 + g_{\Sg}. 
\]
Conical $G_2$-structures are of the form 
\[
\vp_C :=  dr \wedge r^2 \om + r^3 \Om, 
\]
with corresponding metrics
\[
g_{C} = dr^2 + r^2 g_{\Sg}. 
\]
The conformally changed $G_2$-structure $r^{-3} \vp_C$ becomes the cylindrical structure
\begin{align*}
    r^{-3} \vp_C &= r^{-1} dr \wedge \om + \Om \\
    &= dz \wedge \om + \Om,
\end{align*}
after the change of variables $z = \log r$. 
\begin{lem}
    If $\vp$ is asymptotic to $\vp_C$ and satisfies \eqref{equation ac condition order k} for all $k \in \bN$, and we denote by $h$ and $h_C$ the metrics associated to $r^{-3}\vp$ and $r^{-3}\vp_C$, then $h$ is asymptotic to the translation invariant metric $h_C$ in the sense of Definition \ref{definition asymptotically translation invariant}. 
\end{lem}

\begin{proof}
    The equivalence of the condition \eqref{equation ac condition order k} and asymptotic translation invariant property \eqref{equation asymp trans invariance} follows from direct calculation.
\end{proof}

Take the conformally changed $G_2$-structure 
\[
r^{-3} \vp_C = dz \wedge \om + \Om,
\]
and consider the mollification
\[
(\psi_C)_\eps := dz \wedge (\phi_\eps * \om) + \phi_\eps * \Om,
\]
where $\phi (z,v) = \eta(z) \nu(v)$, $\eta, \sg$ are standard bumps on $\bR$ and $\bR^6$ respectively, and $\phi_\eps(z,v) = \tfrac{1}{\eps^7}\eta(\eps^{-1}z) \nu(\eps^{-1}\sg)$.

\begin{lem}\label{lemma infinite ACness of mollified}
    If $\vp$ is asymptotic to $\vp_C$ and to order 3 only and $\psi_\eps$ is the mollification $\phi_\eps * (r^{-3}\vp)$, let $h_\eps$ and $(h_C)_\eps$ denote the metrics associated to $\psi_\eps$ and $(\psi_C)_\eps$ respectively. Then, $h_\eps$ is asymptotic to the cylindrical metric $(h_C)_\eps$ to infinite order in the sense of \eqref{equation asymp trans invariance}.
\end{lem}

\begin{proof}
    Consider the covariant derivative
    \begin{align*}
        &(\na_{(h_C)_\eps})_{\dd_m} (h_\eps - (h_C)_\eps)_{ij} \\
        &\qquad = \dd_m (h_\eps - (h_C)_\eps)_{ij} - \Ga_{mi}^p(h_\eps - (h_C)_\eps)_{pj} - \Ga_{mj}^p(h_\eps - (h_C)_\eps)_{ip}
    \end{align*}
Note that the Christoffel symbols are those of the smoothed background cylinder, $(h_C)_\eps$. The last two summands can be approximated by 
\begin{align*}
    C\max_{\sg \in \Sg} \|\Ga\|_{(h_C)_\eps} \Big( \max_{y \in (z-\eps, z+ \eps)} \|h - h_C\|_{h_C} \Big), 
\end{align*}
which vanishes as $z \rightarrow \infty$ by the basic AC/ACyl condition.  The first summand can be estimated as follows.
\begin{align*}
        &\|\dd_m (h_\eps - (h_C)_\eps)_{ij}\|_{(h_C)_\eps} \\
        &\qquad = \bigg\| \int_M (h-h_C)_{ij}(y)(\dd_m)_x \phi_\eps(x-y) dy \bigg\|_{(h_C)_\eps} \\
        &\qquad \le C\|h - h_C\|_{h_C}, 
\end{align*}
which vanishes as $z \rightarrow \infty$ by the basic AC/ACyl condition. For higher derivatives, the exact same arguments may be applied, except that they will involve higher derivatives of the test function $\dd^k \phi_\eps$ and the Christoffel symbols $\dd^k \Ga$. 
\end{proof}

Finally, we check that these asymptotically cylindrical metrics satisfy Lockhart's definition of admissibility.

\begin{defn}\label{definition lockhart admissible}\cite[Definition 2.3]{Lockhart1987}
Let $g =e^{2\rho}h$, with $\rho \in C^{\infty}(M)$. The metric $g$ is admissible if there is a $C^\infty$, $\bR^+$-invariant 1-form $\theta$ on $(0,\infty) \times \Sg$ with the property that for all $k \in \bN$
\begin{equation}\label{equation admissible}
    \lim_{z\rightarrow \infty} \sup_{\sg \in \Sg} \|\na_{h}^{k+1} \rho - \na_{h}^k \theta \|_h = 0.
\end{equation}
\end{defn}

\begin{lem}
If $\vp$ is an AC $G_2$-structure in the sense of Definition \ref{definition asymptotically conical} and satisfies \eqref{equation ac condition order k} for all $k \in \bN$, then its associated metric $g_\vp$ is admissible.
\end{lem}

\begin{proof}
    If $\vp$ is AC and satisfies \eqref{equation ac condition order k} for all $k \in \bN$, then there is $\psi := r^{-3}\vp$, whose associated metric $h_\psi$ is asymptotically to a cylinder $dz^2 + g_\Sg$ (and therefore asymptotically translation invariant). With $z = \log r$, the following relation holds. 
    \[
    g_\vp = e^{2z}h_\psi. 
    \]
    Setting $\rho = z$ and $\theta = \dd_z$ yields $\theta$ and $\rho$ that satisfy the criteria \eqref{equation admissible} for admissibility.
\end{proof}

\begin{prop}\label{proposition uniform C2 convergence of approximations}
   If $\vp$ is an AC $G_2$-structure and has asymptotic cone $\vp_C$ and satisfies \eqref{equation ac condition order k} for any $k > 2$, then the curvature of $\vp_\eps$ is uniformly bounded above as $\eps \rightarrow 0$. 
\end{prop}

\begin{proof}
    Consider the approach of $\vp_\eps$ to $(\vp_C)_\eps$ in $C^2$. By integration by parts, we can bound $D^2(\vp * \phi_\eps)$ by the second derivatives of $\vp$.   Thus, on the AC end, the rate of convergence in $C^2$ of $\vp_\eps$ to $(\vp_C)_\eps$ as $r \rightarrow \infty$ is ultimately bounded by the rate of convergence of $\vp \rightarrow \vp_C$ in $C^2$. Fix a large $R_0 > 0$, and observe that the curvature of $\vp_\eps$ is uniformly close to that of $(\vp_C)_\eps$, independent of $\eps$. 
    
    The convergence of $(\vp_C)_\eps$ to $\vp_C$ reduces to convergence of these tensors on the compact link, which is uniform in $C^2$. Thus, for sufficiently small $\eps$, the curvature of $(\vp_C)_\eps$ is uniformly close to that of $\vp_C$. Combining the two estimates indicates that there are constants $R_0,\eps_0 >0$ such that for $\eps < \eps_0$ and $r > R_0$, the Riemannian curvature of the $G_2$-structure $\vp_\eps$ can be bounded by that of $\vp_C$. 

    On $M \setminus B_{R_0}$, the $C^2_{loc}$-convergence of $\vp_\eps$ to $\vp$ ensures that we can choose $\eps_0 >0$ such that for $\eps < \eps_0$, 
    \[
    \|\Rm_{\vp_\eps} - \Rm_{\vp}\|_{\vp_C} < \mu, 
    \]
    where $\mu > 0$ is some constant. Thus, the curvature $\Rm_{\vp_\eps}$ is uniformly bounded on $M$ for sufficiently small $\eps$. 
\end{proof}

The following corollary is key to proving Theorem \ref{theorem main theorem} for initial $G_2$-structures of lower regularity. Given the existence and uniqueness results for smooth AC $G_2$-structures which converge to infinite order, it shows that a flow from lower regularity initial data can be extracted from a sequence of smooth approximations.

\begin{cor}\label{corollary existence for lower regularity}
    Suppose Theorem \ref{theorem main theorem} holds for smooth AC initial data satisfying \eqref{equation ac condition order k} for all $k \in \bN$. If $\vp$ is an AC $G_2$-structure and has asymptotic cone $\vp_C$ and satisfies \eqref{equation ac condition order k} for some $k > 2$, then there is a unique Laplacian flow with initial condition $\vp$ whose existence time is bounded below depending on the upper bound of the curvature of $\vp$. 
\end{cor}

\begin{proof}
    Since we assume Theorem \ref{theorem main theorem} the curvature of $\vp_\eps$ converges uniformly to that of $\vp$, given a sequence $\eps_i \rightarrow 0$, the maximal existence time for the flows $\vp_{\eps_i}$ is bounded below. The compactness of these flows (\cite[Theorem 7.2]{LotayWeiLaplacianFlow}) implies that there is a subsequence converging locally in $C^{2,\al}$ to a flow starting at $\vp$ and satisfying \eqref{equation laplacian flow}. The arguments of \S 5-6 can be re-applied to this flow to prove the uniqueness and long-time existence. 
\end{proof}

\begin{rmk}\label{remark assumption of smoothness}
Henceforth, all initial AC $G_2$-structures $\vp_0$ will be assumed to be smooth and satisfy \eqref{equation ac condition order k} for all $k \in \bN$. If Theorem \ref{theorem main theorem} is proved for initial data of this kind, Corollary \ref{corollary existence for lower regularity} will immediately imply that it holds for $G_2$-structures that are simply AC in the sense of Definition \ref{definition asymptotically conical}. 
\end{rmk}

\begin{rmk}\label{remark lockhart vs KL}
Throughout the paper, certain results on Fredholm theory for operators on non-compact manifolds will be cited as formulated in Karigiannis-Lotay \cite{KarigiannisLotay}, although this paper ostensibly requires conditions on the rate of asymptotic convergence and the torsion of the conical $G_2$-structure which are not satisfied in our setting. However, the results we cite are ultimately consequences of the results of Lockhart \cite{Lockhart1987} which are applicable to any noncompact manifolds with admissible metrics. The reason for citing the versions in \cite{KarigiannisLotay} is for clarity of exposition. 
\end{rmk}

\subsection{Function spaces}

To establish the existence of a solution to the Laplacian-DeTurck flow \eqref{equation laplacian deturck boundary}, the appropriate function spaces must be chosen to serve as the setting of the problem. The strategy of the proof will be to find a sequence of solutions (or rather, ``almost-solutions") of \eqref{equation laplacian deturck boundary} on an exhaustive sequence of domains, $M_1 \subset M_2 \subset \cdots \Subset M$, that converge to a global solution on $M$. Therefore, we must find a family of function spaces in which the sequence of domain solutions $(M_i, \dd M_i, \vp_i)$ lie and a function space in which the global solution, their limit, exists.

The first step in choosing the correct spaces is to determine the best map that will ultimately provide a solution to \eqref{equation laplacian deturck boundary}. First, let $\vp_0$ be a fixed AC $G_2$-structure on the noncompact manifold $M$. Then, let $\beta_0$ be a time-dependent $2$-form such that $\beta_0(0) =0$, and consider the mapping
\begin{align}\label{equation the rough mapping}
    &\beta \mapsto d (\beta + \beta_0)  \mapsto \dd_t (d (\beta + \beta_0)) - \Delta_{d (\beta + \beta_0)+\vp_0} (d (\beta + \beta_0) + \vp_0) \nonumber \\
    & \qquad \qquad \qquad \qquad \qquad \qquad \qquad \qquad + \cL_{V(d (\beta + \beta_0) + \vp_0)} (d (\beta + \beta_0) + \vp_0) \\
    &\Ga(N \times [0,T], \La^2) \rightarrow \Ga(N \times [0,T], \La^3)
    \rightarrow \Ga(N \times [0,T], \La^3), \nonumber
\end{align}
where $N$ is a placeholder which can indicate either the complete manifold $M$ or any of the subdomains $M_i$. We are deliberately imprecise about the domains and codomains. In particular, we make no assumptions about regularity, but rather only wish to indicate that the map goes from 2-forms to 3-forms, and so on. 

Having chosen this map, the next step is to restrict its domain and codomain to spaces of forms that satisfy good properties with respect to the asymptotically conical structure. These are the weighted H\"older and Sobolev spaces. Since we have two types of Cauchy problem, that on a compact domain and that on the whole manifold, we will discuss these two cases separately, beginning with the latter. 

\subsubsection{Weighted H\"older and Sobolev spaces} \label{section weighted sobolev holder}
In this section, we assume that we are measuring all norms with respect to a given asymptotically conical metric (usually the metric $g_0$ associated to the initial condition $\vp_0$). Before defining the H\"older and Sobolev spaces, we define an auxiliary function $\rho : M \rightarrow \bR$, which is equal to $1$ on a core compact region of $M$, and eventually coincides with radius function on the cylinder $\Sg \times (b,\infty)$ underlying the non-compact end $V$. Where $\rho$ is not equal to the radial coordinate function, we may assume it has reasonable behavior (no zeros, bounded derivatives, and so on). 

The weighted H\"older and Sobolev spaces, roughly speaking, will contain functions that have $O(\rho^\la)$ decay/growth, where $\la \in \R$ and $\rho \rightarrow \infty$.

Consider the following weighted H\"older norms, defined as in e.g. \cite[\S 7]{LMCFwithsings}. 

\begin{defn}\label{definition weighted parabolic holder norms}
    Let $\rho : M \rightarrow [1, \infty) \subset \R$ be a smooth function that coincides with the radius function on the end of $M$, $V \cong \Sigma \times (b, \infty)$. Let $\la \in \R$. If $I\subset \R$ is a time interval and $\sg$ is a time-dependent tensor on $M$ defined on $I$, then
    \begin{align}\label{equation weighted holder definition}
        \|\sg\|_{C^{l,m,\al}_{\la}} &= \sum_{i,j} \bigg\{ \sup_{(t,x) \in I \times M} |\rho(x)^{-\la + 2i + j}\dd_t^i \nabla^j \sg(t,x)| + \sup_{t \in I} [\dd_t^i \nabla^j \sg(t, \cdot)]_{\al, \la -2i -j}\nonumber \\
        &\qquad \qquad + \sup_{x \in M} [\partial_t^i \nabla^j \sg(\cdot, x)]_{\al/2, \la - 2i -j}\bigg\},
    \end{align}
    where $i = 1,\ldots, l$, $j = 1, \ldots, m$ such that $2i + j \le m$. The weighted H\"older seminorms are defined for a rate $\la \in \R$ by
    \begin{align*}
        [\ups(t,\cdot)]_{\al, \la} &= \sup_{x \ne y, d(x,y) < \mathrm{inj}(g)} \bigg[ \min(\rho(x),\rho(y))^{-\la} \frac{|\ups(t,x) - \ups(t,y)|_g}{d(x,y)^\al} \bigg] \\
         [\ups(\cdot,x)]_{\al/2, \la} &= \sup_{t_1,t_2 \in I} \bigg[ \rho(x)^{-\la} \frac{|\ups(t_1,x) - \ups(t_2,x)|_g}{|t_1-t_2|^{\al/2}} \bigg]
    \end{align*}
\end{defn}

\begin{rmk}
    In the case that $\sg$ is time-independent the $C^{l,m,\al}_{\la}$-norms reduce to the standard $C^{m,\al}_{\la}$-norms given by
    \begin{align}\label{equation weighted static holder definition}
        \|\sg\|_{C^{m,\al/2}_{\la}} &= \sum_{j=1}^m \bigg\{ \sup_{x \in M} |\rho(x)^{-\la +  j} \nabla^j \sg(x)| + \sup_{x \in M} [\nabla^j \sg(x)]_{\al, \la -j}\bigg\}.
    \end{align}
\end{rmk}

\begin{defn}\label{definition weighted holder space of forms}
The weighted H\"older space of forms
\[
C^{l,m,\al}_{\la}\Omega^k(M \times I) 
\]
is the completion of the smooth, compactly supported, time-depedent $k$-forms $\Omega^k_c(M \times I)$ in the $C^{l,m,\al}_\la$-norm. 
\end{defn}

We will also need the corresponding weighted Sobolev norms, also defined as in \cite[\S 7]{LMCFwithsings}.

\begin{defn}\label{definition weighted parabolic sobolev norms}
    Let $\rho$ and $\la$ be as in Definition \ref{definition weighted holder space of forms}. If $I\subset \R$ is a time interval and $\sg$ is a time-dependent tensor on $M$ defined on $I$, then
    \begin{align}\label{equation weighted parabolic sobolev definition}
        \|\sg\|_{W^{r,s,p}_{\la}} &= \Bigg( \sum_{i,j} \int_I \int_M |\rho^{-\la + 2i + j}\dd_t^i \nabla^j \sg(t,\cdot)|^p \rho^{-7} dV_g dt \Bigg)^\frac{1}{p}.
    \end{align}
    where $i = 1, \ldots, r$, $j = 1, \ldots, s$ and $2i + j \le s$. 
\end{defn}

\begin{rmk}
    If we assume that $\sigma$ is time-independent and that $I$ is a unit-length interval, then this definition reduces to the following static weighted Sobolev norm.
\begin{align}\label{equation weighted static sobolev definition}
        \|\sg\|_{W^{s,p}_{\la}} &= \Bigg( \sum_{j=0}^s \int_M |\rho^{-\la + j} \nabla^j \sg|^p \rho^{-7} dV_g \Bigg)^\frac{1}{p}.
    \end{align}
\end{rmk}

\begin{defn}\label{definition weighted sobolev space of forms}
The weighted Sobolev space of forms
\[
W^{r,s,p}_{\la}\Omega^k(M \times I) 
\]
is the completion of the smooth, compactly supported, time-depedent $k$-forms $\Omega^k_c(M \times I)$ in the $W^{r,s,p}_\la$-norm. When $p = 2$, we will use the notation
\[
W^{r,s,2}_{\la}\Omega^k(M \times I) := H^{r,s}_{\la}\Omega^k(M \times I).
\] 
\end{defn}

\begin{rmk}
    We will often discuss weighted spaces of static (as opposed to time-dependent) forms. In this case, these spaces will be considered subspaces of corresponding spaces of time-dependent forms with the same weighted norms. They will be denoted with the natural and conventional notation $C^{m,\al}_\la\Omega^k(M), W^{s,p}_{\la}\Omega^k(M), H^{s}_{\la}\Omega^k(M)$, and so on.
\end{rmk}

\subsubsection{Function spaces on domains}\label{section function spaces domains}

Recall that we wish to solve \eqref{equation laplacian deturck boundary} on a sequence of domains. These domains will be the truncations 
\[
M_R := \rho^{-1}([1, R)), \;R \in [1,\infty).
\]  
Note that we will consider the endpoint $M_\infty$ to be equal to the original manifold $M$. In order to measure ``the same norms" along both the sequence of solutions on the domains $M_R$ and in the limit, we will consider the H\"older and Sobolev spaces of time-dependent $k$-forms on the $M_R$ measured with respect to the weighted H\"older and Sobolev norms \eqref{equation weighted holder definition} and \eqref{equation weighted parabolic sobolev definition}, respectively. To be more precise, we set 
\begin{align}\label{equation weighted holder for domains}
    C^{l,m,\al}_{\la}\Omega^k(M_R \times I) &:= \bigcap_{j=0}^l \Big( C^{j,\al/2}(I; C^{l-2j}_{\la - 2j} \Omega^k(M_R)) \cup C^{j}(I; C^{l-2j,\al}_{\la - 2j} \Omega^k(M_R))\Big),
\end{align}
where the $k$-forms on $M_R$ are the sections of the bundle $\La^k T^*M_R$. In particular, $T^*M_R$ is defined on the boundary $\dd M_R$ using a collar neighborhood. The space of weighted Sobolev $k$-forms is defined similarly. 
\begin{align}\label{equation weighted sobolev for domains}
    W^{r,s,p}_{\la}\Omega^k(M_R \times I) &:= \bigcap_{j=0}^r W^{j,p}(I; W^{s-2j}_{\la - 2j} \Omega^k(M_R)).
\end{align}

\begin{rmk}\label{remark weighted unweighted equivalence}
Since the domains $M_R$ are compact in $M$ and thus $\rho: M \rightarrow \R$ is bounded above and below away from 0 on $M_R$, the weighted spaces (e.g. $C^{l,m,\al}_{\la}\Omega^k(M_R \times I)$) are equivalent to their unweighted counterparts (e.g. $C^{l,m,\al}\Omega^k(M_R \times I)$), but not uniformly so. Thus, we will use the weighted notation when we wish to consider these functions against the weighted norm, and the unweighted notation when it is necessary to consider them in the unweighted spaces. 
\end{rmk}

There is a Sobolev embedding theorem for weighted spaces that will be useful.

\begin{thm}\label{theorem weighted sobolev embedding theorem}
    Let $(M,\vp)$ be a smooth, complete 7-manifold equipped with an AC $G_2$-structure as in Definition \ref{definition asymptotically conical} with radius function $\rho: M \rightarrow \R$ as in Definition \ref{definition weighted parabolic sobolev norms}. Let $s, l \in \bR$, $p,q \in [1,\infty)$, $\al \in (0,1)$, and $\la, \de \in \R$. Let $M_R := \rho^{-1}((0,R])$ for any $R \in (1, \infty]$.   If $s - \tfrac{7}{p} \ge l + \al$ and $\la \le \de$. If $\sigma \in W^{s,p}_{\la}$, then
    \begin{equation*}
        \|\sg\|_{C^{l,\al}_{\de}\Om^k(M_R)} \le c \|\sg\|_{W^{s,p}_{\la}\Om^k(M_R)}. 
    \end{equation*}
    for a constant $c >0$ independent of $R$.
\end{thm}

\begin{proof}
    The proof is standard modification of the scaling argument used in Bartnik \cite[Theorem 1.2]{BartnikMass} in Euclidean space. In the asymptotically conical case, the annulus $A_R$ is replaced by the set $\rho^{-1}([R, 2R])$. See also \cite[Theorem 2.9]{JoyceDesing1}.
\end{proof}

The above Sobolev embedding theorem can be used to prove a parabolic Sobolev embedding theorem for the time-dependent Sobolev and H\"older spaces. The following proposition corresponds to \cite[Proposition 7.2]{LMCFwithsings}. 

\begin{prop}\label{proposition parabolic weighted Sobolev embedding} Let $(M,\vp)$ be a smooth, complete 7-manifold equipped with an AC $G_2$-structure as in Definition \ref{definition asymptotically conical} with radius function $\rho: M \rightarrow \R$ as in Definition \ref{definition weighted parabolic sobolev norms}. Let $I \subset \R$ be an open and bounded interval, $s \in \bN$ with $s \ge 2$, $\la \in \R$, and $p \in (1, \infty)$. Let $\ve > 0$ and let $p > 2/\ve$ and $sp > 9$.  If  $\sg \in W^{1,s,p}_\la $, then
    \[
    \|\sg\|_{C^{0,0}_{\la + \ve}\Om^k(M_R \times I)} \le c\|\sg\|_{W^{1,s,p}_{\la}\Om^k(M_R \times I)}.
    \]
for a constant $c >0$ independent of $R$.
\end{prop}

\begin{rmk}\label{remark consequences of sobolev embedding}
    Note that since $p = 2$ in the rest of the paper, in practice $\ve$ will be slightly greater than 1. Furthermore, it will be necessary to have $s \ge 5$ weak derivatives, so that $sp > 9$. 
\end{rmk}

\begin{proof} 
    For each $M_R$, we use the extension operator $E_R$ defined in Construction \ref{construction extension to collar neighborhood} to define an embedding of 
    \[
    W^{1,s,p}_\la \Om^k(M_R \times I) \hookrightarrow W^{1,s,p}_\la \Om^k(M \times I),
    \]
    which is bounded independently of the parameter $R$, by Lemma \ref{lemma extension of functions}. 
    
    Now, the proof follows the same lines as that of \cite[Proposition 7.2]{LMCFwithsings}. By \cite[Proposition 2.4]{LMCFwithsings} and originally \cite[Ch. III, Thm. 4.10.2]{Amann}, 
    \[
    W^{1,s,p} \Om^k (M \times I) \hookrightarrow C^0(I; \Om^k(W^{s,p}_\la (M), W^{s-2,p}_{\la + 2} \Om^k(M))_{1/p,p}),
    \]
    is a continuous embedding. The space $(W^{s,p}_\la \Om^k (M), W^{s-2,p}_{\la + 2} \Om^k (M))_{1/p,p}$ is the real interpolation space (see \cite[Ch. 1, \S 2.4]{Amann}). This space can be continuously embedded into $W^{l,p}_{\la + \ve}\Om^k (M)$ for $l < s - 2/p$, by the argument of \cite[Lemma 5.4]{CoriascoDiffOps} applied to asymptotically conical rather than conically singular manifolds. 
    
    By the weighted Sobolev embedding theorem \cite[Theorem 2.9]{JoyceDesing1}, 
    \[
    C^0(I;W^{l,p}_{\la + \ve}\Om^k (M)) \hookrightarrow C^0(I;C^{0}_{\la + \ve}\Om^k (M))
    \]
    is a continuous embedding with uniform constant. However, the restriction
    \[
    C^0(I;C^{0}_{\la + \ve}\Om^k (M)) \hookrightarrow C^0(I;C^{0}_{\la + \ve}\Om^k (M_R))
    \]
    is a continuous embedding with operator norm independent of $R$. Composing these embeddings, we obtain for a time-dependent form $\phi \in W^{1,s,p}_\la \Om^k(M_R \times I)$
    \[
    \|\phi\|_{C^{0,0}_{\la + \ve}\Om^k(M_R \times I)} = \big\|E_R(\phi)\big|_{M_R}\big\|_{C^{0,0}_{\la + \ve}\Om^k(M_R \times I)} \le C \|\phi\|_{W^{1,s,p}_\la \Om^k(M_R \times I)},
    \]
    where $C>0$ does not depend on the parameter $R > R_0$.
\end{proof}

\subsubsection{Spaces of exact, co-exact, and harmonic forms on manifolds with boundary}\label{section hodge morrey}

Since we are interested in a flow of exact forms, we must first define the appropriate spaces of exact forms on a manifold with boundary for our family of boundary value problems.  To this end, we will use the Hodge-Morrey decomposition of square-integrable $k$-forms on a compact manifold with boundary $(N, \dd N)$. 

To obtain this decomposition, the tangential and normal parts of a $k$-form at the boundary $\dd N$ must first be defined.

\begin{defn}\label{definition tangential and normal components}
    Let $\omega$ be a $k$-form on a manifold $N$ with boundary $\dd N$. If $X$ is a vector field on $TN|_{\dd N}$, let $X = X^\perp + X^\parallel$ be the decomposition of $X$ into parts perpendicular and parallel to $T(\dd N)$. The tangential part of $\omega$ is defined by
    \[
    \bt \omega (X_1, \ldots, X_k) = \omega(X_1^\parallel, \ldots, X_k^\parallel) \;\;\forall X_i \in \Ga(TN|_{\dd N}).
    \]
    The normal part of $\omega \in \Om^k(N)$ is defined by
    \[
    \bn \omega = \omega|_{\dd N} - \bt \omega.
    \]
\end{defn}
Let $g$ be a metric on $N$ and $\de$ the co-differential corresponding to $g$. The tangential and normal projections have the following key commutative property.
\begin{prop}\cite[Proposition 1.2.6]{HodgeDecompSchwarz}
The exterior derivative commutes with tangential projection, and the co-differential with the normal projection. More precisely, for $\om \in \Om^k(N)$ and with the natural identifications,
\begin{equation}\label{equation commutativity of tangential with ext}
\bt(d\om) = d(\bt \om) \qquad \text{ and } \qquad \bn(\de \om) = \de (\bn \om).
\end{equation}
\end{prop}

For the convenience of the reader, we record Green's theorem for manifolds with boundary (see \cite[Proposition 2.1.2]{HodgeDecompSchwarz}). 
\begin{prop}\label{proposition greens theorem}
Let $(N,\dd N)$ be a compact manifold with boundary and let $\om \in H^1\Om^{k-1}(N)$ and $\eta \in H^1\Om^{k}(N)$. Then
\[
\langle d \om, \eta \rangle_{L^2(N)} = \langle \om, \de \eta \rangle_{L^2(N)} + \int_{\dd N} \bt \om \wedge \star \bn \eta.
\]
\end{prop}

With these concepts defined, we can now introduce the Hodge-Morrey decomposition. The components of this decomposition will be the following spaces of $L^2$-integrable exact forms, co-exact forms, and harmonic fields. 
\begin{align*}
    \sE^k(N) &:= \{d\al \;|\; \al \in H^1\Omega^{k-1}(N), \; \bt \al = 0\} \\
    \sC^k(N) &:= \{\de \beta \;|\; \beta \in H^1\Omega^{k+1}(N), \; \bn \beta = 0\} \\
    \sH^k(N) &:= \{\la \in H^1 \Omega^{k}(N) \; |\; d\la = 0 \textrm{ and } \de \la =0\}
\end{align*}
The Hodge-Morrey decomposition \cite[Theorem 2.4.2]{HodgeDecompSchwarz} states that for any $s \in \bN$ and $p \ge 2$, there is an $L^2$-orthogonal decomposition 
\begin{align}\label{equation hodge morrey decomp}
    W^{s,p}\Omega^k (N) = W^{s,p}\sE^k(N) \oplus W^{s,p}\sC^k(N) \oplus W^{s,p}\sH^k(N), 
\end{align}
where $W^{s,p}\sE^k(N) = \sE^k(N) \cap W^{s,p}\Omega^k (N)$, and so on. There is a further $L^2$-orthogonal decomposition of the harmonic fields given in \cite[Corollary 2.4.9]{HodgeDecompSchwarz}, which states that
\begin{align}\label{equation hodge morrey further decomp}
    W^{s,p}\Omega^k (N) = W^{s,p}\sE^k(N) \oplus W^{s,p}\sC^k(N) \oplus W^{s,p}\sH_{co}^k(N) \oplus W^{s,p}\sH_{D}^k(N), 
\end{align}
where $\sH_{co}$ are the co-exact harmonic fields and $\sH_{D}$ are the harmonic fields $\beta$ satisfying the \textit{Dirichlet boundary condition}, $\bt \beta = 0$. The decomposition \eqref{equation hodge morrey further decomp} implies the \textit{Helmholtz decomposition} which states that each $k$-form $\om$ can be split into co-exact and closed parts, $\om_{co}$ and $\om_{cl}$ such that
\begin{multline}\label{equation helmholtz decomp}
\om_{cl} \in W^{s,p}\sE^k(N) \oplus W^{s,p}\sH_{D}^k(N) \text{ and }\bt \om_{cl} = 0,\\ \om_{co} \in W^{s,p}\sC^k(N) \oplus W^{s,p}\sH_{co}^k(N).
\end{multline}

\begin{rmk}\label{remark mirror helmholtz decomp}
    There is an analogous Helmholtz decomposition of $k$-forms $\om$ into the sum $\om_{cc} + \om_{ex}$ of a co-closed part $\om_{cc}$ satisfying the \textit{Neumann boundary condition} $\bn \om_{cc} = 0$, and an exact part $\om_{ex}$.
\end{rmk}

Recall that the first map in the composition \eqref{equation the rough mapping} is the exterior derivative: So, in order for the image $d\beta$ to be exact in the sense of the the Hodge-Morrey decomposition, the $2$-form $\beta$ must be in $H^1\Omega^{2}(N)$ and must also satisfy the \textit{Dirichlet boundary condition}, $\bt \beta = 0$. These Dirichlet $2$-forms, following Schwarz's notation in \cite[\S 2.2]{HodgeDecompSchwarz}, will be denoted $H^1\Omega^2_D(N)$.

The time-dependent versions of these spaces are defined as in \eqref{equation weighted sobolev for domains}, with the co-differential and Hodge-Morrey decomposition taken on each time-slice with respect to a \textit{fixed} background metric on $N$.
\begin{align}
    W^{r,s,p} \Omega_D^2(N \times I) &:= \bigcap_{j=0}^r W^{j,p}(I; W^{s-2j,p} \Omega_D^2(N)), \label{equation time dependent dirichlet forms}\\
    W^{r,s,p} \sE^3(N \times I) &:= \bigcap_{j=0}^r W^{j,p}(I; W^{s-2j,p} \sE^3(N)), \label{equation time dependent exact forms}
\end{align}
where $r$ is the number of time derivatives and $s$ is the number of spatial derivatives. The weighted spaces are simply the same spaces of forms, but considered with the weighted norms. In this paper, we will restrict our view to the setting where $p=2$, for which we will use the standard notation $W^{r,s,2} = H^{r,s}$. 

\subsubsection{Incorporating boundary conditions}\label{section boundary conditions}

It will be necessary to impose boundary conditions on various time-dependent functions, and this section establishes the correct spaces in which these boundary conditions can be realized. 

\begin{lem}\label{lemma smoothness of time derivatives}
Assume that $r \ge 1$ and $s \ge 2r$. If $\eta \in H^{r,s} \Omega^k (N \times I)$ and $0 \le j \le r-1$, then
\begin{itemize}
    \item[(i)] $\dd_t^j \eta \in C(I; H^{(s-1)-2j}\Omega^k(N))$.
    \item[(ii)] The following estimate holds:
    \[
    \max_{t \in I} \|\dd_t^j \eta (t)\|_{H^{(s-1)-2j}} \le C \|\eta\|_{H^{r,s}},
    \]
    where $C > 0$ depends on $|I|$, $s$ and the geometry of $N$.
\end{itemize}
\end{lem}
\begin{proof}
Observe that for any $0 \le j \le r-1$, by definition $\dd_t^j\eta \in L^2(I; H^{s-2j} \Omega^k(N))$ and $\dd_t^{j+1} \eta \in L^2(I; H^{s-2j-2} \Omega^k(N))$. All arguments in the proof of the elementary theorem \cite[\S 5.9.2, Theorem 4]{EvansPDE} carry over to the case of a vector bundle over a compact manifold with smooth boundary, and therefore (i) and (ii) hold.
\end{proof}

The next definition establishes an important space of forms vanishing at the initial time to high order.

\begin{defn}\label{definition exact 3 forms vanishing on boundary}Let 
$\al \in H^{r,s}_{\la} \Om^3 (N \times I)$, where $r \ge 1$ and $s \ge 2r$. Consider the composition map
\[
\al \mapsto (\al(0),\dd_t \al(0), \dd_t^2 \al (0), \ldots, \dd_t^{r-1}\al(0) ). 
\]
\sloppy By Lemma \ref{lemma smoothness of time derivatives}, this is a well-defined and bounded map from $H^{r,s}_\la \Om^3 (N \times I)$ to $\bigoplus_{j=0}^{r-1} H^{(s-1)-2j}\Omega^3(N)$, and its kernel is thus a closed subspace of $H^{r,s}_\la \Om^3 (N \times I)$. This kernel is the subspace of 3-forms in $H^{r,s}_\la \Om^3 (N \times I)$ which vanish to order $r-1$ at $t=0$ and is denoted by 
\[
H_{0,\la}^{r,s} \Om^3 (N \times I)
\]
to be the kernel of this map, i.e. the subspace of 3-forms in $H^{r,s}_\la \Om^3 (N \times I)$ whose potentials vanish to order $r-1$ at $t=0$. Let $H^{r,s}_\la \sU^3 (N \times I)$ be the closed subspace of $H^{r,s}_\la \Om^3 (N \times I)$ defined in Proposition \ref{proposition boundedness of restricted forms}. Let
\[
H_{0,\la}^{r,s} \sU^3 (N \times I)
\]
be the intersection $H^{r,s}_{0,\la} \Om^3 (N \times I) \cap H^{r,s}_\la \sU^3 (N \times I)$. Note that this definition coincides with that given in Proposition \ref{proposition boundedness of restricted forms}. In particular, this is the subspace of 3-forms in $H^{r,s}_\la \sU^3 (N \times I)$ which vanish to order $r-1$ at $t=0$. 
\end{defn}

\begin{rmk}
Note that in particular the elements of $H_{0,\la}^{r,s} \Om^3 (N \times I)$ automatically satisfy the compatibility conditions $g_1,\ldots, g_r \in H^1_0$ at $t=0$ (defined in Theorem \ref{theorem higher regularity linear par pde}) for the parabolic operator $\dd_t - \sL$ .
\end{rmk}

We also restrict the space of $2$-forms $\beta$ in $H^{r,s}\Omega_D^2 (N \times I)$ under consideration. 
\begin{defn}\label{defn dirichlet 2 forms with vanishing boundary condition}
Let $r \ge 2$, $s \ge 2r$, and let
    \[
    H_{0,\la}^{r,s}\Omega_D^2 (N \times I)
    \]
    be the space of $H^{r,s}$-regular Dirichlet $2$-forms $\beta$ measured against decay rate $\la$ which
    \begin{itemize}
        \item vanish on $N \times \{0\}$ to order $r-1$ in $t$;
        \item satisfy $\bt \beta = 0$ on $\dd N \times I$ as $H^1$-forms defined on every time slice $N  \times \{t\}$.
    \end{itemize}
\end{defn}

\subsubsection{Readjusting the mapping}\label{subsubsection readjustment}

In this section, we revisit the mapping \eqref{equation the rough mapping} and make some final modifications. First, add in the assumption that the first $r$ terms of the Taylor expansion in $t$ around $t=0$ of the approximate solution $\beta_0 \in C^{r,s}(M \times I)$ match the the first $r$ terms of a formal series solution at $t = 0$ of the homogeneous equation
\[
\dd_t (d\beta_0) - \Delta_{d\beta_0 + \vp_0} (d\beta_0 + \vp_0) + \cL_{V(d\beta_0 + \vp_0)} (d\beta_0 + \vp_0) = 0.
\]
Under this assumption, the image of $\beta_0$ under the (nonlinear) Laplacian-DeTurck operator,
\begin{equation}\label{equation image approx solution}
\eta_0 := \dd_t (d\beta_0) - \Delta_{d\beta_0 + \vp_0} (d\beta_0 + \vp_0) + \cL_{V(d\beta_0 + \vp_0)} (d\beta_0 + \vp_0),
\end{equation}
vanishes to order $r$ at $t = 0$. The existence of such an approximate solution will be established in Section \ref{section approximate solution}. 

Let $N$ stand for a subdomain $M_R$ of $M$, where $R \in [1,\infty)$. The first component of our new mapping will be 
\begin{equation}\label{equation first component mapping}
\beta \in H^{r,s}_{0,\la+3} \Omega^2_D(N \times I) \mapsto d(\beta + \beta_0) \in H^{r,s-1}_{\la + 2}\Omega^3(N \times I) .
\end{equation}
Here we may assume that $s$ is selected so that $s-3 \ge 2r$. Note that although $d\beta$ is exact in the sense of the Hodge-Morrey decomposition \eqref{equation hodge morrey decomp}, the sum $d(\beta + \beta_0)$ is not and must be considered simply an exact form, perhaps with a harmonic component. The next stage in the mapping is 
\begin{align}
    &d(\beta + \beta_0) \in H_{\la +2}^{r,s-1}\Omega^3(N \times I) \nonumber\\ 
    & \qquad  \mapsto \dd_t (d(\beta + \beta_0)) - \Delta_{d(\beta + \beta_0) + \vp_0} (d(\beta + \beta_0) + \vp_0) \label{equation second component mapping} \\
    & \hspace{5cm}+ \cL_{V(d(\beta + \beta_0) + \vp_0)} (d(\beta + \beta_0) + \vp_0) \nonumber \\
    &\hspace{7cm} \in H_{0,\la}^{r-1,s-3}\Omega^3(N \times I). \nonumber
\end{align}
Note that the vanishing of $\beta$ to order $r-1$ at $t=0$ implies that the image of the map approaches $\eta_0$ to order $r-2$ in $t$. Since $\eta_0$ vanishes to order $r$, the image of the map is in $H_{0,\la}^{r-1,s-3}\Omega^3(N \times I)$.

The final steps in the mapping are \textbf{(a)} the addition of a co-exact (\textit{not} in the sense of Hodge-Morrey) correction term $\de \omega$ and \textbf{(b)} the $L^2_{\la}$-orthogonal projection $\Pi_{\sU}$ of the resulting expression onto a finite-codimensional subspace $H_{0,\la}^{r-1,s-3}\sU^3(N \times I)$ of the space $H_{0,\la}^{r-1,s-3}\Om^3(N \times I)$. The subspace $\sU$ is defined precisely in Proposition \ref{proposition boundedness of restricted forms} and can be considered the complement of an ``approximate kernel". The adjustments \textbf{(a)} and \textbf{(b)} are performed by the following mapping.
\begin{align}
    &\Big( \dd_t (d(\beta + \beta_0)) - \Delta_{d(\beta + \beta_0) + \vp_0} (d(\beta + \beta_0) + \vp_0) + \cL_{V(d(\beta + \beta_0) + \vp_0)} (d(\beta + \beta_0) + \vp_0), \om\Big) \nonumber \\
    & \qquad \in H_{0,\la}^{r-1,s-3}\Omega^3(N \times I) \oplus H_{0,\la + 1}^{r-1,s-2}\Om^4 (N \times I)  \nonumber \\
    &\mapsto \Pi_{\sU} \Big(\dd_t (d(\beta + \beta_0)) - \Delta_{d(\beta + \beta_0) + \vp_0} (d(\beta + \beta_0) + \vp_0) \label{equation third component mapping} \\ 
    &\hspace{5cm} + \cL_{V(d(\beta + \beta_0) + \vp_0)} (d(\beta + \beta_0) + \vp_0) + \de \om \Big) \nonumber \\
    & \hspace{6.9cm} \in H_{0,\la}^{r-1,s-3}\sU^3(N \times I) \subset H_{0,\la}^{r-1,s-3}\Om^3(N \times I). \nonumber 
\end{align}

\begin{rmk}\label{remark uniform boundedness of projection to U}
By Proposition \ref{proposition uniform boundedness of projection} and Lemma \ref{lemma projection to harmonic dirichlet fields}, the $L^2_{\la}$-orthogonal projection $\Pi_{\sU}$ of $H^{r-1,s-3}_{0,\la}\Om^3(N \times I)$ to the subspace $H^{r-1,s-3}_{0,\la}\sU^3(N \times I)$ is a well-defined, continuous linear map, that is bounded uniformly with respect to $R$, when $N$ is set to be $M_R$. 
\end{rmk}

The final composition of all of the above mappings \eqref{equation first component mapping}-\eqref{equation third component mapping} is a non-linear map which we will call
\begin{equation}\label{equation fix notation mapping}
\cP : H^{r,s}_{0,\la+3} \Omega^2_D(N \times I) \oplus H_{0,\la+1}^{r-1,s-2}\Om^4 (N \times I)  \rightarrow H_{0,\la}^{r-1,s-3}\sU^3(N \times I).
\end{equation}

\subsubsection{Setting up the inverse function theorem argument.}\label{section IFT setup}

The goal is to show that $\cP$ is locally surjective in a neighborhood of $\Pi_{\sU} (\eta_0)$ (see \eqref{equation image approx solution} for the definition of $\eta_0$) in the subspace $H_{0,\la}^{r-1,s-3}\sU^3(N \times I)$, and that the size of this neighborhood in the weighted norms is uniformly bounded below for the entire exhaustion sequence, $N = M_{R_i} \in \{M_{R_i}\}_{i=1}^\infty$. This will be achieved by the quantitative inverse function theorem for Banach spaces. 

\begin{thm}\label{theorem quantitative inverse function theorem}
Let $E_1, E_2$ be Banach spaces and $F: E_1 \rightarrow E_2$ a differentiable map. Suppose that $DF_0$ is surjective, with right inverse $P$ and let $\ka_0 > 0$ be such that the map $F - DF_0$ is Lipschitz with constant $(2\| P \|)^{-1}$ on the ball $B_{\ka_0}(0) \subset E_1$. Let $\ka := \ka_0(2\|P\|)^{-1}$. 

Then, for all $y \in B_{\ka}(F(0)) \subset E_2$, there is an $x \in B_{\ka_0}(0) \subset E_1$ such that $F(x) = y$.
\end{thm}

\subsubsection{Properties of the mapping}

In order to use Theorem \ref{theorem quantitative inverse function theorem}, we calculate the linearization of the map $\cP$. 

\begin{lem}
    Consider a neighborhood $B$ of $0 \in H^{r,s}_{0,\la+3} \Omega^2_D(N \times I)$ such that if $\ga \in B$, then $d(\ga + \beta_0) + \vp_0$ is an AC $G_2$-structure. The linearization of $\cP$ at $(\ga, \nu)$, for any $\ga \in B$ and $\nu \in H_{0,\la+1}^{r-1,s-2}\Om^4 (N \times I)$, is
    \[
    D\cP_{(\ga,\nu)} : H^{r,s}_{0,\la+3} \Omega^2_D(N \times I) \oplus H_{0,\la+1}^{r-1,s-2}\Om^4 (N \times I) \rightarrow H_{0,\la}^{r-1,s-3}\sU^3(N \times I),
    \]
    which is given by the formula
    \[
    D\cP_{(\ga,\nu)}(\chi,\om) = \Pi_{\sU} \big(\dd_t (d\chi) + \Delta_{d(\ga + \beta_0) + \vp_0} (d\chi) - d\Phi (d\chi) + \de \om \big),
    \]
    where we evaluate in the direction of $(\chi, \om) \in H^{r,s}_{0,\la+3} \Omega^2_D(N \times I) \oplus H_{0,\la+1}^{r-1,s-2}\Om^4 (N \times I) $.
\end{lem}

\begin{proof}
    The first map in the composition of $\cP$ is the affine map, $\beta \mapsto d(\beta + \beta_0)$, whose Fr\'echet derivative is simply the linear map $\chi \mapsto d\chi$. The third map is the linear projection $\Pi_{\sU}$ onto a direct summand of $H_{0,\la}^{r-1,s-3}\Omega^3(N \times I)$, and its Fr\'echet derivative is simply itself, $\Pi_{\sU}$. By \cite[Lemma 2.4]{BryantXu}, the linearization of the Laplacian-DeTurck operator on a closed 3-form $\theta$ is
    \[
    \dd_t \theta + \Delta_{d(\ga + \beta_0) + \vp_0} \theta - d\Phi (\theta),
    \]
    where, on each time-slice, $\Phi$ depends on the torsion of the $G_2$-structure $d(\ga + \beta_0) + \vp_0$. Finally, the map $\nu \rightarrow \de \nu$ is linear and its derivative is $\om \mapsto \de \om$. Putting these all together gives the lemma. 
\end{proof}

To satisfy the hypotheses of the quantitative inverse function theorem, the Lipschitz norm of the map $\cP - D\cP_0$ must be controlled. The following lemma establishes the continuity of $D\cP$ at $0$, whose modulus is uniformly bounded with respect to the radial cut-off parameter $R$. 

\begin{lem}\label{lemma lipschitz bounds}
    Let $M_R := \rho^{-1}([0, R])$, where $R \in [1, \infty]$ and $r,s > C$, some fixed parameter. For any $\eps >0$, there exists for all $R$, a $\kappa >0$ such that if a pair $(\ga, \nu)$ is in the ball $B_\kappa(0)$ around $0$ in $H^{r,s}_{0, \la+3} \Omega^2_D (M_R \times I) \oplus H_{0,\la+1}^{r-1,s-2}\Om^4 (N \times I)$, then the operator norm
    \[
    \|D\cP_{(\ga, \nu)} - D\cP_0\| < \eps.
    \]
    In particular, the Lipschitz constant $L$ of $\cP - D\cP_0$ is less than $\eps$ in $B$.
\end{lem}

\begin{proof}
Let $g := g_{d\beta_0 + \vp_0}$ and $\hat g := g_{d(\ga + \beta_0) + \vp_0}$, and $\cR_g, \cR_{\hat g}$ are the corresponding curvature terms from the Weitzenb\"ock identity. Consider the difference of the two linearizations. 
\begin{align*}
    D\cP_{(\ga, \nu)}(\beta,\om) -  D\cP_0(\beta,\om)&= \Pi_{\sU} \big( \dd_t (d\beta) + \Delta_{d(\ga + \beta_0) + \vp_0} (d\beta) - d\Phi_\ga (d\beta) + \de \om  \\
    &\qquad \qquad - (\dd_t (d\beta) + \Delta_{d\beta_0 + \vp_0} (d\beta) - d\Phi_0 (d\beta) + \de \om )  \big) \\
    &=  \Pi_{\sU} \big( (\Delta_{\hat g} - \Delta_{g}) (d\beta) + (\cR_{\hat g} - \cR_{g})(d\beta) - d(\Phi_\ga - \Phi_0) (d\beta) \big).  
\end{align*}
Note that we are fixing the co-differential $\de$ to be with respect to the initial metric $g_0(0)$. (N.B. This fixing of $\de$ is for the purpose of fixing a Hodge-Morrey decomposition which does not depend on the evolution of the metric in time.) If $r,s$ are sufficiently large, the Sobolev embedding theorem (Theorem \ref{theorem weighted sobolev embedding theorem}) and the formulas for metric, curvature, and torsion in terms of a given $G_2$-structure (see, for instance, \cite[\S 2]{LotayWeiLaplacianFlow}) imply that given $\eps > 0$, there exists a $\ka >0$ such that if 
\[
(\beta,\om) \in B_\ka(0) \subset H^{r,s}_{0, \la+3} \Omega^2_D (M_R \times I) \oplus H_{0,\la+1}^{r-1,s-2}\Om^4 (N \times I), 
\]
then 
\[
\|(\Delta_{\hat g} - \Delta_{g}) (d\beta) + (\cR_{\hat g} - \cR_{g})(d\beta) - d(\Phi_\ga - \Phi_0) (d\beta) \|_{H^{r-1,s-3}_\la} < \eps \|\beta\|_{H^{r,s}_{\la+3}}.
\]
The uniform boundedness of $\Pi_{\sU}$ proved in Proposition \ref{proposition uniform boundedness of projection} and Lemma \ref{lemma projection to harmonic dirichlet fields} implies that we may adjust $\ka$ such that 
\[
\|(D\cP_{(\ga,\nu)} - D\cP_0) (\beta) \|_{H^{r-1,s-3}_\la} < \eps \big(\|\beta\|_{H^{r,s}_{\la+3}} + \|\om\|_{H^{r-1,s-2}_{\la+1}}\big).
\]
Observing that the choice of $\ka$ was independent of the choice of $R$ completes the proof of the lemma.
\end{proof}

\begin{rmk}
In order to apply the quantitative inverse function theorem (Theorem \ref{theorem quantitative inverse function theorem}), it remains to prove the surjectivity of the linearizations $D \cP_0$ and the uniform boundedness of the right inverses for $D \cP_0$ on the domains $M_R$. These properties will be established in Section \ref{section short time existence IVT argument}.
\end{rmk}

\section{Constructing an approximate solution to Laplacian-DeTurck flow}\label{section approximate solution}

In this section, an approximate solution $\theta_0 = d\beta_0$ to the Laplacian-DeTurck flow is constructed which satisfies all the necessary properties outlined in \S \ref{subsubsection readjustment}.

\begin{lem}\label{lemma approximate solution}
    Given $r \in \bN$, $s \ge 2r$, and an initial closed, AC 3-form $\vp_0$ of sufficiently high regularity, there is a smooth, exact 3-form $\theta$ defined on a time interval $I = [0,T]$ such that for any rate $\la > -2$,
    \[
    \theta \in H^{r+1,s+2}_{\la}\Omega^3(M \times I) 
    \]
    and, for $\vp := \theta + \vp_0$, the expression
    \[
    \dd_t \theta - \Delta_{\vp} \vp - \cL_{V(\vp)}\vp,
    \]
    vanishes to order $r$ in $t$ at the initial time-slice $\{0\} \times M$.  
\end{lem}

\begin{proof}
    We begin by considering the formal series of coefficients of the following Taylor expansion in $t$.
    \begin{align*}
        \theta^{(0)}(x) &= 0 \\
        \theta^{(1)}(x) &= \Delta_{\vp_0} \vp_0 \\
        \theta^{(2)}(x) &= \dd_t (\Delta_{\vp} \vp + \cL_{V(\vp)})|_{t=0} = F_2(\nabla^4 \vp_0, \nabla^3 \vp_0, \nabla^2 \vp_0, \nabla \vp_0, \vp_0) \\
        \theta^{(3)}(x) &= \hspace{1cm} \cdots \\
        &\vdots
    \end{align*}
    Note that $\partial_t$ commutes with $d$, and thus each $\theta^{(j)}$ is an exact form. In order to estimate the $\theta^{(j)}$'s, we expand the expression $\Delta_\vp \vp + \cL_{V(\vp)}\vp$. First, expand 
    \begin{align*}
        \cL_{V(\vp)} \vp (X,Y,Z) &= (\nabla_{V(\vp)}\vp)(X,Y,Z) - \vp(\nabla_X V(\vp), Y, Z) \\
        &\qquad - \vp(X, \nabla_Y V(\vp), Z) - \vp(X, Y, \nabla_Z V(\vp))
    \end{align*}
    with arbitrary unit vector fields $X,Y,Z$. Notice that $V(\vp)$ is given by a contraction of the Christoffel symbols (up to a constant). Appealing to \eqref{equation evolution christoffel symbols}, observe that
    \begin{align*}
        \dd_t \big(\cL_{V(\vp)} \vp (X,Y,Z)\big)|_{t=0} &\sim \dd_t \Ga * \nabla \vp_0 +  \nabla \dd_t \Ga \\
        &\sim \nabla \eta * T + \nabla^2 \eta,
    \end{align*}
    where $\eta$ is given by \eqref{equation evolution metrics}. All geometric terms (e.g. $\eta$, $\Rm$, $T$, and $\Ga$) are taken with respect to the initial form $\vp_0$. By Lemma \ref{lemma decay estimates on AC G2 forms}, this term is of order $O(\rho^{-4})$ on each end. Further derivatives can be found by applying the evolution equations \eqref{equation evolution metrics}, \eqref{equation evolution christoffel symbols}, \eqref{equation evolution torsion}, and \eqref{equation evolution curvature}, followed by Lemma \ref{lemma decay estimates on AC G2 forms}--this yields
    \begin{align*}
        \dd^k_t \big(\cL_{V(\vp)} \vp (X,Y,Z)\big)|_{t=0} &= O(\rho^{-2 - 2k}),
    \end{align*}
    on each end. To compute the time derivatives of the Hodge Laplacian, use the Weitzenb\"ock formula to expand
    \begin{align*}
    \dd_t\big(\Delta_{\vp} \vp\big) |_{t=0} &= \dd_t\big(\Delta_g \vp + \cR_g \vp \big) |_{t=0} \\
    &\sim \eta * \nabla^2 \vp_0 + g_0 * \nabla^4 \vp_0 + \dd_t \Rm * \vp_0 + \Rm * \nabla^2 \vp_0 + \Rm * \Rm * \vp_0.
    \end{align*}
    The evolution equations and Lemma \ref{lemma decay estimates on AC G2 forms} tell us that this term is $O(\rho^{-4})$ on each end. By similar logic, we have 
    \begin{align*}
        \dd^k_t \big(\Delta_{\vp}\vp\big)|_{t=0} &= O(\rho^{-2 - 2k}),
    \end{align*}
    on each end.

    Now we construct the approximate solution.  Let $r \in \bN$ be as in the hypotheses of the lemma.
    Let 
    \begin{equation}\label{equation approximate solution first pass}
    \theta (x,t) := \sum_{k=1}^{r+1} \frac{t^k}{k!}\theta^{(k)}(x).
    \end{equation}
    The construction of the $\theta^{(k)}$'s implies that $\dd_t \theta - \Delta_{\vp} \vp - \cL_{V(\vp)}\vp$ vanishes to order $r+1$ at the initial time-slice $\{0\} \times M$. Here, we assume that the regularity of $\vp_0$ is high enough that each $\theta^{(k)}$ is $(s+2)$-times differentiable in space. Thus, for any $\la > -2$, 
    \[
    \theta \in H^{r+1, s+2}_{\la} \Om^3(M \times I). 
    \]
    The commutativity of spatial and time derivatives and the exactness of the $\theta^{(k)}$'s ensure that $ \theta(x,t)$ is exact. 
\end{proof}

\begin{cor}\label{corollary equivalence of approx soln metrics}
    A time $T' > 0$ can be selected so that the $G_2$-structure $\vp(t) = \theta(t) + \vp_0$ is non-degenerate. This time $T'$ depends only on the norm (measured with respect to $g_0$) of the curvature of $\vp_0$ and that of its covariant derivatives up to order $s+2$.
\end{cor}

\begin{proof}
    Let $R_0 > 0$ be a large radius. Cover the compact ``core" $M_{R_0}$ of $M$ by a finite number of coordinate neighborhoods $U^a$ with orthonormal frames $\{e_i^a\}_i$ (with respect to $g_0$). Cover the link $\Sg$ of the asymptotic cone $C$ by a finite number of coordinate neighborhoods $V^b$, and cover the end $M \setminus M_{R_0}$ by coordinate neighborhoods of the form $V^b \times \R$ with orthonormal frames $\{e_i^b\}_i$ (with respect to $g_0$). 

    We assume a curvature bound of the form
    \begin{equation}\label{equation curvature bound initial condition}
        |\na^k \Rm_{g_0}|_{g_0} \le C_k \qquad \text{for }k = 0,1,\ldots, s+2. 
    \end{equation}
    
    By the construction in Lemma \ref{lemma approximate solution} and the identities in \cite[\S 2]{LotayWeiLaplacianFlow}, 
    \begin{align*}
    \big|\theta^{(k)}\big|_{g_0} &\le F\big(|\na^{k+2} \Rm|_{g_0}, \ldots,|\Rm|_{g_0}\big) \le C_{r,s}\sum_{k=0}^{r+3} C_k^2  \qquad \text{ for }k = 1, \ldots, r+1. 
    \end{align*}
    In order to verify the condition of non-degeneracy, we consider the tensor (c.f. \cite[\S 2.1]{KarigiannisFlows})
    \[
    B_{ij} = \big((e_i \lrcorner \vp) \wedge (e_j \lrcorner \vp) \wedge \vp \big)\big(e_1,e_2, \ldots, e_7),
    \]
    for an orthonormal frame $\{e_i\}_i$ (with respect to $g_0$) in a coordinate neighborhood $U$. If $\vp = \vp_0$, then 
    \[
    B_{ij} = -6\de_{ij}.
    \]
    If $\vp(t) = \vp_0 + \theta(t)$, then
    \begin{align*}
        B_{ij} &= -6 \de_{ij} + \bigg(t^k \sum_{k = 1}^{r+1} \vp_0 * \vp_0 * \theta^{(k)} + t^{k+l} \sum_{k,l = 1}^{r+1} \vp_0 * \theta^{(l)} * \theta^{(k)} \\ &\qquad  \qquad  \qquad +  t^{k+l+m} \sum_{k,l,m = 1}^{r+1} \theta^{(m)} * \theta^{(l)} * \theta^{(k)}\bigg) \Big( e_i \odot e_j, e_1 \wedge \cdots \wedge e_7 \Big).
    \end{align*}
    If $t$ is made sufficiently small compared to the $C_k$'s, and since $g_{ij} = 6^{-2/9}\det(B)^{1/9}B_{ij}$, 
    \[
    \tfrac12 (g_0)_{ij}\le g_{ij}(t) \le 2 (g_0)_{ij}.
    \]
    By taking the minimum of such $t$'s on each coordinate neighborhood $U^a$ and $V^b \times \R$, we obtain $T' > 0$ depending only on the upper bounds \eqref{equation curvature bound initial condition} such that $\vp$ is nondegenerate on $[0,T']$. 
\end{proof}

\begin{rmk}
    To reduce notation, we continue to use the constant $T$ in our domain. However, we assume that either $T$ has been a) reduced to $T'$ (from Corollary \ref{corollary equivalence of approx soln metrics}) or b) $\theta$ has been smoothly cut off before $T'$ so that the metrics induced by the $G_2$-structure $\theta + \vp_0$ are non-degenerate.
\end{rmk}

\begin{cor}\label{corollary asymptotics of eta}
    Given $\theta$ as constructed in Lemma \ref{lemma approximate solution}, let 
    \[
    \eta_0 := \dd_t \theta - \Delta_{\vp} \vp - \cL_{V(\vp)}\vp,
    \]
    where $\vp = \theta + \vp_0$. Then, 
    \[
    \eta_0 \in H^{r,s}_{0,\la-2}\Om^3(M \times [0,T]),
    \]
    for any $\la > -2$. 
\end{cor}

\begin{proof}Define $P(\theta) = \dd_t \theta - \Delta_{\vp} \vp - \cL_{V(\vp)}\vp$. Compute
    \begin{align*}
        \eta_0(x,t) &= \dd_t \bigg(\sum_{m=1}^{r+1} \frac{t^m}{m!}\theta^{(m)}(x)\bigg) - \Delta_{\vp} \vp - \cL_{V(\vp)}\vp \\
        &= \dd_t \bigg(\sum_{m=2}^{r+1} \frac{t^m}{m!}\theta^{(m)}(x)\bigg) - (\Delta_{\vp} \vp - \Delta_{\vp_0}\vp_0) - \cL_{V(\vp)}\theta - d((V(\vp) - V(\vp_0))\lrcorner \vp_0) \\
        &= O\big(\theta^{(2)}(x)\big)  + O(\theta * \Rm) + O(\nabla \theta * \Ga) + O(\nabla^2\theta) \\
        &= O(\rho^{-4}),
    \end{align*}
    since $\theta \in H^{1,2}_{\la}(M \times I)$. By the same kind of computation, observe that for $l \le r$,
    \begin{align*}
        \dd_t^l \eta_0 (x,t) &= \dd_t^l \bigg(\sum_{m=1}^{r+1} \frac{t^m}{m!}\theta^{(m)}(x)\bigg) - \dd_t^l(\Delta_{\vp} \vp - \cL_{V(\vp)}\vp )\\
        &= O\big(\theta^{(1+l)}(x_0)\big)  + O\bigg(\sum_{j+k=l} \dd_t^j \theta * \dd_t^k \Rm\bigg) \\
        &\qquad + O\bigg(\sum_{j+k=l} \nabla \dd_t^j \theta * \dd_t^k\Ga\bigg) + O\big(\nabla^2\dd_t^l\theta\big) \\
        &= O(\rho^{-2 - 2l}).
    \end{align*}
    Finally, if $2l + k \le s$, 
    \begin{align*}
        \nabla^k \dd_t^l \eta_0(x,t)
        &= O\big(\nabla^k \theta^{(1+l)}(x_0)\big) + O\bigg(\sum_{p+q =k}\sum_{i+j=l} \nabla^p \dd_t^i \theta * \nabla^q \dd_t^j \Rm\bigg) \\
        &\qquad + O\bigg(\sum_{p+q =k}\sum_{i+j=l} \nabla^p \dd_t^i \theta * \nabla^q \dd_t^j\Ga\bigg) + O\big(\nabla^{2+k}\dd_t^l\theta\big) \\
        &= O(\rho^{-2 - 2l-k}).
    \end{align*}
    This completes the proof of the corollary.
\end{proof}

\begin{cor}\label{corollary nearby vanishing short time}
    Given $\eps >0$ and $\theta$ as constructed in Lemma \ref{lemma approximate solution} and $\eta_0$ as defined in Corollary \ref{corollary asymptotics of eta}, there is a small $\ka > 0$ and an exact form $\eta_\ka$ that is $\eps$-close to $\eta_0$ in $H^{r-1,s-2}_{0,\la-2}\Om^3(M \times [0,T])$, for $\la > -2$, and that vanishes uniformly on the interval $[0,\ka] \subset [0,T] = I$.
\end{cor}

\begin{proof}
    Since $\eta_0$ vanishes to order $r$ at the initial time slice, $\{0\} \times M$, we can extend $\eta$ by zero on $(-1,0)$ to be defined in $H^{r,s-2}_\la \Om^3(M \times (-1,T])$. We consider the shifted forms 
    \begin{align*}
        \eta_{\ka}(x,t) := \eta_0(x, t-\ka),\;\;\; \eta_{\ka} \in H^{r,s-2}_\la \Omega^3 (M \times [0,T]).
    \end{align*}
    Consider the $(k+2l)$-order derivatives $\nabla^k\dd_t^l \eta_0$, where $k + 2l \le s-2$. By \cite[Theorem 2, \S 5.9.2]{EvansPDE}, for every $t_0 \in [0,T]$,
    \[
    \dd_t^l \nabla^k \eta_0(t_0 - \ka, \cdot) - \dd_t^l \nabla^k \eta_0(t_0, \cdot) = \int_{t_0 - \ka}^{t_0}  \dd_t^{l+1} \nabla^k \eta_0(\tau, \cdot) d\tau,
    \]
    in $L^2$. Taking norms, we can estimate
    \begin{align*}
        \big\| \dd_t^l \nabla^k \eta_0(t_0 - \ka) - \dd_t^l \nabla^k \eta_0(t_0)\big\|_{L^2_{\la-k-2l}} &\le \int_{t_0 - \ka}^{t_0} \big \| \dd_t^{l+1} \nabla^k \eta_0(\tau)\big\|_{L^2_{\la -k -2l}} d\tau.
    \end{align*}
    Since $\|\dd_t^{l+1} \nabla^k \eta_0(\tau)\|_{L^2_{\la-k-2l}}$ is finite for each $\tau$ (i.e. less than $\|\dd_t^{l+1} \nabla^k \eta_0(\tau)\|_{L^2_{\la-k-2(l+1)}}$ which depends only on $\vp_0$) we may, by the definition of the weighted parabolic Sobolev norms in \eqref{equation weighted parabolic sobolev definition}, choose $\ka >0$ small enough that the right-hand side is less than $\eps$.

    Considering all $k + 2l \le s-2$ and $l \le r-1$, repeating this process, and taking the minimum of the resulting collection of $\ka$'s yields a $\ka > 0$ such that $\eta_\ka$ satisfies all the desired properties in the statement of the corollary.
\end{proof}

The next lemma constructs approximate extensions to flows beginning at $\vp_0$ and solving \eqref{equation laplacian deturck boundary} on some time interval $[0,T]$. These approximate extensions will be used in the proof of long-time existence in \S \ref{section long time behavior}. Note that all norms in all Sobolev spaces in the following are taken with respect to the metric $g_0$.

\begin{lem}\label{lemma approx extension}
     Let $\vp(t)$ solve \eqref{equation laplacian deturck boundary} with initial data $\vp_0$ for $t \in [0,T]$ such that $\vp(t) - \vp_0 \in H^{r+1,s+2}_{\la}\Omega^3(M \times [0,T])$ for $\la > -2$, $r \in \bN$, $s \ge 2r$. Suppose further that for each $t \in [0,T]$, the curvatures associated to $\vp(t)$ are uniformly bounded in the norm induced by $g_0$, i.e.
    \begin{equation}\label{equation long time curvature bounds}
    |(\na^{g(t)})^k \vp(t)|_{g_0} \le C_k \rho^{-k},  
    \end{equation}
    for $k \in \bN_{\ge 0}$. Then there exists $T' > T$ and $\theta(t) \in H^{r+1,s+2}_{\la}\Omega^3(M \times [0,T'])$ such that $\theta(t) = \vp(t) - \vp_0$ on $[0,T]$, the expression 
    \[
    \eta_0 := \dd_t \theta - \De_\vp \vp - \cL_{V(\vp)}\vp
    \]
    vanishes to order $r$ for $t \in [0,T]$ and is contained in the space $H^{r,s}_{0,\la-2}\Om^3(M \times [0,T'])$, where $|T' - T|$ is bounded below depending  only on the constants $C_k$ in \eqref{equation long time curvature bounds}. In addition, the geometric data up to order $l$ associated to the $G_2$-structure $\theta(t) + \vp_0$ is uniformly bounded in terms of the $C_k$'s, for $k \le 3l + r + s + 3$. 
    
    Finally, given some $\eps > 0$, there is a small $0 < \ka < 1$, depending only on the constants $C_k$ in the curvature bounds \eqref{equation long time curvature bounds}, and an exact form $\eta_\ka$ that is $\eps$-close to $\eta_0$ in $H^{r-1,s-2}_{\la-2}\Om^3(M \times [0,T'])$ for $\la > -2$, and that vanishes uniformly on the interval $[0,T + \ka] \subset [0,T']$.
\end{lem}

\begin{proof}
    The construction of the extended solution begins precisely as in Lemma \ref{lemma approximate solution}, except we take the formal Taylor series expansion at $\vp(T)$ instead of $\vp_0$. Set $\theta(t) = \vp(t) - \vp_0$ on $[0,T]$ and define $\theta(t)$ as in \eqref{equation approximate solution first pass} on $[T,T + 1]$. Note that the two definitions of $\theta(t)$ match at $t=T$ up to order $r$. 
    
    Since $\theta(t)$ may grow rather large from $T$ to $T+1$, the domain must be restricted from $[0,T+1]$ to $[0, T']$, where $T' \in (T, T+1)$ is uniformly bounded below depending only on the constants $C_k$ from \eqref{equation long time curvature bounds} up to a certain order $k$. First, $T'$ can be chosen as in Corollary \ref{corollary equivalence of approx soln metrics} so that $\vp_0 + \theta(t)$ only stretches the metric $g(T)$ by a bounded amount.  
    
    We would furthermore like the Christoffel symbols, curvatures, and their higher-order derivatives (up to some high order) to be uniformly bounded for the approximate solution. Thus, we can take $l$ space derivatives and $p$ time derivatives of \eqref{equation approximate solution first pass}, and see that we can reduce $t$ to make $|\dd_t^p \na^l \theta(t)|_{g_0}$ arbitrarily small, depending only on the $C_k$'s up to large but finite order (i.e. order $p + (s + 2) + 2l + (r +1)$). 

    The vanishing form $\eta_\ka$ can be obtained in precisely the same way as in Corollary \ref{corollary nearby vanishing short time}, where $\ka$ will roughly be the largest $\ka$ such that 
    \[
    \sum_{k = 1}^q \ka C_k \le \eps. 
    \]
    Thus, we have $\eta_\ka$ in $H^{r-1,s-2}_{\la-2}\Om^3(M \times [0,T'])$ and there is a constant $K(C_k,r,s) > 0$ such that 
    \[
    0 < T + K(C_k,r,s) < T + \ka < T'. 
    \]
    This completes the proof of the lemma.
\end{proof}

\section{Short-time existence of Laplacian flow}\label{section short time existence IVT argument}

\subsection{Outline of the proof of existence}\label{subsection existence outline}

The first step to finding a short-time solution to the Laplacian-DeTurck flow is to use the inverse function theorem to find ``almost" solutions (solutions up to ``small eigenforms" of the Hodge Laplacian) to \eqref{equation laplacian deturck boundary}, the Laplacian-DeTurck flow on the domains $M_R$. This is achieved by first using the existence theory for linear parabolic equations to prove the existence of a right inverse $\cT_R$ for $D\cP_0$, and then using the maximum principle and $L^2$-Schauder estimates for the linearized equation to bound the norm $\|\cT_R\|$ of the right inverse independently of $R$. 

The linear operator we will consider while calculating the right inverse $\cT_R$ will be the following operator which acts on 2-forms. 
\begin{align}\label{equation linearization operator on domains}
    \sL_R \beta &= \dd_t \beta + \Delta_{d\beta_0 + \vp_0} \beta - \Phi(d\beta), \qquad \beta \in \Omega^2(M_R). 
\end{align}
Since the output of the Laplacian-DeTurck operator is an exact form, we can follow the bottom path of the following commutative diagram.
\[
\xymatrix@+1pc{
H_{0,\la+2}^{r,s-1}\sE^3(M_R \times I) \ar[r(1.4)]^-{\dd_t + \De - d(\Phi(\cdot))}  && H_{0,\la}^{r-1,s-3}\Om^3(M_R \times I)  \\
H^{r,s}_{0,\la+3} \Omega^2_D(M_R \times I) \ar[u]^d \ar[r(1.4)]^-{\sL_R} && H^{r-1,s-2}_{0,\la+1} \Omega^2 (M_R \times I) \ar[u]^d
}
\]
Recall the nonlinear operator
\[
\cP : H^{r,s}_{0,\la+3} \Omega^2_D(M_R \times I) \oplus H_{0,\la+1}^{r-1,s-2}\Om^4 (M_R \times I)  \rightarrow H_{0,\la}^{r-1,s-3}\sU^3(M_R \times I),
\]
defined in \eqref{equation fix notation mapping}. The linearization of $\cP$ at 0 is $D\cP_0$, which can be expressed in terms of $\sL_R$ as follows.
\[
H^{r,s}_{0,\la+3} \Omega^2_D(M_R \times I) \oplus H^{r-1,s-2}_{0,\la+1} \Omega^4(M_R \times I) \rightarrow H_{0,\la}^{r-1,s-3}\sU^3(N \times I)
\]
\begin{equation}\label{equation linearization}
  (\eta, \om) \mapsto \Pi_{\sU}\big( d (\sL_R \eta) + \de \om \big)  
\end{equation}

Thus, to find the right inverse, first notice that by Proposition \ref{proposition boundedness of restricted forms}, for a given $\nu \in H_{0,\la}^{r-1,s-3}\sU^3(M_R \times I)$, we can find a primitive 2-form $\xi \in H_{0,\la+1}^{r-1,s-2}\Omega_D^2(M_R \times I)$ and a primitive 4-form $\om \in H_{0,\la+1}^{r-1,s-2}\Omega^4(M_R \times I)$ such that 
\[
\nu = d\xi + \de \om, \; d\xi \in H_{0,\la}^{r-1,s-3}\sE^3(M_R \times I), \;  \de \om \in H_{0,\la}^{r-1,s-3}(\sC^3 \oplus \sH^3_{co})(M_R \times I),
\]
and which satisfy the bounds
\[
\|\xi\|_{H^{r-1,s-2}_{\la+1}}, \; \|\om\|_{H^{r-1,s-2}_{\la+1}} \le C \|\nu\|_{H^{r-1,s-3}_{\la}},
\]
where $C > 0$ is independent of $R$. Given a potential $\xi$ for the Hodge-Morrey exact part with these properties, existence theory for linear parabolic equations (\S \ref{subsection existence parabolic pde}) yields a 2-form $\beta$ with homogeneous boundary conditions such that $\sL_R \beta = \xi$. The homogeneous boundary conditions on $\beta$ and the compatibility conditions on $\xi$ allow us to conclude that the 2-form $\beta$ has the desired regularity. The details of this existence argument are contained in \S \ref{subsection existence of solution to linearization}.

The right inverse $\cT_R$ of $D\cP_0$ is taken to be
\[
\cT_R(\nu) = (\beta, \om).
\]
To bound $\cT_R$ uniformly in $R$, the weighted Schauder estimates (see \S \ref{subsection weighted schauder}) are combined with a priori estimates (see \S \ref{subsection a priori estimates}) from a barrier argument to prove a bound of the form 
\[
\|\beta\|_{H^{r,s}_{\la+3}} \le C \|\xi\|_{H^{r-1,s-2}_{\la+1}},
\]
where $C>0$ is independent of $R$. This estimate and the bounds on $\om$ from Proposition \ref{proposition boundedness of restricted forms} is sufficient to conclude that we can apply Theorem \ref{theorem quantitative inverse function theorem} to $\cP$ defined on each $M_R$ in a uniform neighborhood of the restricted form $\eta_0|_{M_R}$ (see Corollary \ref{corollary asymptotics of eta} for the definition of $\eta_0$) projected to $H^{r-1, s-3}_{0,\la}\sU^3 (M_R \times I)$. The details of the argument establishing the uniform boundedness of the right inverse are contained in \S \ref{subsection uniform boundedness right inverse}. 

\begin{thm}\label{theorem existence of inverse on domains}
    Let $\beta_0$ be a potential for the approximate solution defined in Lemma \ref{lemma approximate solution} and $\eta_\ka$ be the family of 3-forms defined in Corollary \ref{corollary nearby vanishing short time}. Given $R > R_0$, let $\beta_{0,R}$, $\eta_{0,R}$, and $\eta_{\ka,R}$ denote the restriction of $\beta_0$, $\eta_0$, and $\eta_\ka$ to $M_R \times I$, respectively.
    
    There exists $\ve_0 > 0$ such that given $0 < \ve < \ve_0$, there exists a $\ka > 0$ such that for all $R > R_0$, there is a pair 
    \[
    (\beta_{R}, \om_R) \in B_{\ve}\big((0,0)\big) \subset H^{r,s}_{0,\la+3}\Om^2_D(M_R \times I) \oplus H^{r-1,s-2}_{0,\la+1}\Om^4(M_R \times I)
    \]
    solving
    \[
    \cP(\beta_{R},\om_R) = \Pi_{\sU}(\eta_{\ka,R}),
    \]
    where $\Pi_{\sU}(\eta_{\ka,R})$ vanishes for times $t \in [0,\ka]$.
\end{thm}

\begin{proof}
The Lipschitz bounds from Lemma \ref{lemma lipschitz bounds} and the uniform boundedness of the right inverse $\cT_R$ from Proposition \ref{proposition uniform boundedness of right inverse} allow us to apply the quantitative inverse function theorem (Theorem \ref{theorem quantitative inverse function theorem}) on each domain $M_R$ centered at $\Pi_{\sU}(\eta_{0,R}) = \cP(\beta_{0,R},0)$ with identical constants for all $R > R_0$. 

Given $\ve >0$, the quantitative inverse function theorem says that on $B_{\ve}(\beta_{0,R})$, $\cP$ is surjective onto a small neighborhood of $\Pi_{\sU}(\eta_{0,R})$, with diameter depending only on $\ve$, not $R$. By the uniform boundedness of $\Pi_{\sU}$ given in Proposition \ref{proposition uniform boundedness of projection}, $\ka>0$ can be chosen small enough so that $\Pi_{\sU}(\eta_{\ka,R})$ is in a sufficiently small $H^{r-1,s-3}_{\la}$-neighborhood of $\Pi_{\sU}(\eta_{0,R})$ for all $R>0$.

Let $\cK$ be the local right inverse of $\cP$ given by Theorem \ref{theorem quantitative inverse function theorem}. Setting $(\beta_{R},\om_{R}) := \cK\big(\Pi_{\sU}(\eta_{\ka,R})\big)$ proves the theorem.
\end{proof}

Let $R_i \rightarrow \infty$ be a sequence of radii and let $M_{R_i}$ be a corresponding sequence of domains. Taking $\beta_{R_i}$ from the inverse images $(\beta_{R_i},\om_{R_i})$ given by Theorem \ref{theorem existence of inverse on domains} yields a locally bounded sequence of flows $\vp_i = d(\beta_{R_i} + \beta_0|_{R_i}) + \vp_0|_{R_i}$ on the exhaustive sequence of domains $(M_{R_i}, \dd M_{R_i})$, where $R_i \rightarrow \infty$ and such that for some $\ka > 0$
\[
\Pi_{\sU}\big(\dd_t \vp_i - \De_{\vp_i}\vp_i - \cL_{V(\vp_i)}\vp_i\big|_t  + \ka_{\vp_0} \om_{R_i}\big) = 0 \text{ for } t \in [0, \ka] \text{ and } \vp_i(0) = \vp_0|_{M_{R_i}}.
\]
Let $\vp := \lim_{i\rightarrow \infty} \vp_i$, where the nature of the limit is to be specified in \S \ref{subsection existence everywhere}. Theorem \ref{theorem solution is exact solution} tells us that
\[
\dd_t \vp - \De_{\vp}\vp - \cL_{V(\vp )}\vp = 0, \;\;t \in [0, \ka],
\]
i.e. that $\vp(t)$ is a solution to the closed Laplacian flow on $M \times [0,\ka]$. The details of the proof of global existence are contained in \S \ref{subsection existence everywhere}.  

\subsection{Existence of a solution to the linearization}\label{subsection existence of solution to linearization}\label{subsection parabolic existence}

Given a domain $M_R$ with $R< \infty$ and a 3-form $\nu = d\xi + \de \om$ in $H^{r-1, s-3}_{0,\la}\sU^3(M_R \times I)$, we aim to find a 2-form in $H^{r,s} \Omega^2(M_R \times I)$ mapping to $d \xi$ which satisfies the Dirichlet boundary condition. Proposition \ref{proposition boundedness of restricted forms} tells us that $d\xi$, the exact part of $\nu$ in $H^{s-3}_{0,\la}\sE^3(M_R)$, has a primitive $\xi \in H_{0,\la+1}^{r-1,s-2}\Omega_D^2(M_R \times I)$ satisfying the following estimate,
\begin{align}\label{equation time dependent potential sobolev estimate}
\|\xi\|_{H^{r-1, s-2}_{\la+1}} \le C \|d\xi\|_{H^{r-1, s-3}_{\la}},
\end{align}
where $C$ is independent of the domain $M_R$. In particular, $\xi$ satisfies the compatibility conditions in Theorem \ref{theorem higher regularity linear par pde}.

\begin{prop}\label{proposition unique solution sufficient regularity}
Given the primitive $\xi$ in \eqref{equation time dependent potential sobolev estimate}, there exists a unique 2-form $\beta$ which solves the linearized equation
\begin{align}
    \dd_t \beta + \Delta_{d\beta_0 + \vp_0} \beta - \Phi(d\beta) &= \xi.  \label{equation linearized LF potential}
\end{align}
and such that $\beta \in \overset{\bullet}{H^{1,1}}\Om^2 (M_R \times I) \cap H^{r,s}(M_R \times I)$.
\end{prop}

\begin{proof}
The Weitzenb\"ock identity gives the following equivalent equation,
\begin{align*}
    \dd_t \beta - \Delta_g \beta + \cR(\beta) - \Phi(d\beta) &= \xi,  
\end{align*}
where $\De_g$ denotes the Laplace-Beltrami operator. In local coordinates, this equation is of the form \eqref{equation linear parabolic pde}. Thus, the $L^2$-regularity of $\xi$ and Theorem \ref{theorem existence of linear weak solution} implies the existence of a unique weak solution in $\overset{\bullet}{H^{1,1}}\Om^2(M_R \times I)$, the space of $H^{1,1}$-regular forms with homogeneous conditions on the parabolic boundary.

Observe that by definition the functions $g_1, \ldots, g_{r-1}$ corresponding to $\xi$ in the hypotheses of Theorem \ref{theorem higher regularity linear par pde} vanish uniformly. The regularity assumptions on $d\xi$ and the estimate \eqref{equation time dependent potential sobolev estimate} ensure that $\xi$ satisfies the regularity condition in the first hypothesis of Theorem \ref{theorem higher regularity linear par pde}. Thus, the theorem implies that
\[
\frac{d^k\beta}{dt^k} \in L^2(I, H^{s-2k}\Om^2(M_R)) \qquad \text{for } k = 0, \ldots, r,
\]
i.e. $\beta$ is contained in $H^{r,s}\Om^2(M_R \times I)$ and has regularity up to the boundary (where it vanishes).
\end{proof}

\subsection{A priori estimates on the linearized equation} \label{subsection a priori estimates}
In this section, we prove a priori estimates on the weighted $L^2$- and $C^{0}$-norms of solutions to the linearized equation
\begin{align}\label{equation repeat of potential equation}
    \dd_t \beta + \Delta_{d\beta_0 + \vp_0} \beta - \Phi(d\beta) &= \xi, 
\end{align}
where $\beta$ and $\xi$ are time-dependent 2-forms of sufficient regularity. 

\begin{lem}\label{lemma diff ineq for evol of norm}
    If $\beta, \xi \in \Omega^2(M)$ satisfy \eqref{equation repeat of potential equation}, then $|\beta|^2$ satisfies 
    \begin{equation}\label{equation evolution norm beta}
        \dd_t |\beta|^2 \le \Delta_{g_0(t)} |\beta|^2 + C(1 + |\Rm|)|\beta|^2 + |\xi|^2,
    \end{equation}
    where $g_0(t)$ is the metric associated to $d\beta_0 + \vp_0$, $\De_{g_0(t)}$ is the Laplace-Beltrami operator and $C>0$ is a universal constant. Note that the coefficients $|\beta|^2$ in \eqref{equation evolution norm beta} are uniformly bounded, depending only on the approximate solution $d\beta_0 + \vp_0$ constructed in \S \ref{section approximate solution}.
\end{lem}

\begin{proof} On each time-slice, take all inner products and covariant derivatives with respect to the metric $g_0(t)$ induced by $d\beta_0 + \vp_0$ and its Levi-Civita connection. In the rest of the proof, $\Delta$ will denote the Laplace-Beltrami operator associated to $g_0(t)$ (as opposed to the Hodge Laplacian) and $\cR$ will denote the curvature term from the Weitzenb\"ock formula.

Compute
    \begin{align*}
        \dd_t |\beta|^2 &= 2\langle \beta, \dd_t \beta \rangle + \dd_t g_0 (\beta, \beta)\\
        &= 2\langle \beta, \Delta \beta + \cR \beta + \Phi (d\beta) + \xi \rangle +  \dd_t g_0 (\beta, \beta)\\
        &= \Delta |\beta|^2 - 2|\nabla \beta|^2 + 2 \langle \beta, \cR \beta \rangle + 2\langle \beta, \Phi (d \beta)\rangle  + 2 \langle \beta, \xi \rangle + \dd_t g_0 (\beta, \beta)\\
        &= \Delta |\beta|^2 - 2|\nabla \beta|^2 + 2 \langle \beta, \cR \beta \rangle  + 2\langle \beta, \Phi * \nabla \beta \rangle + 2 \langle \beta, \xi \rangle+ \dd_t g_0 (\beta, \beta),
    \end{align*}
    where $\Phi$ is a linear operator on the appropriate space with coefficients bounded by the torsion of $\vp_0$ and its derivatives. Since $g_0(t)$ is uniformly non-degenerate, there is a constant $C>0$ depending only on the fixed family $d\beta_0(t) + \vp_0$ such that $|\dd_t g (\beta, \beta)| \le C|\beta|_g^2$.
    Apply Cauchy-Schwarz and Young's inequalities to obtain
    \begin{align*}
        \dd_t |\beta|^2  &\le \Delta |\beta|^2 - 2|\nabla \beta|^2 + 2 |\beta| |\cR \beta| + 2|\beta| |\Phi * \nabla \beta| + 2|\beta| |\xi| + C|\beta|^2 \\
        &\le \Delta |\beta|^2 - 2|\nabla \beta|^2 + 2 |\cR| |\beta|^2 + 2|\beta| |\Phi| |\nabla \beta| + 2|\beta| |\xi| + C|\beta|^2 \\
        &\le \Delta |\beta|^2 - 2|\nabla \beta|^2 +  C|\Rm||\beta|^2 + C|\Phi|^2|\beta|^2 + |\nabla \beta|^2 + C|\beta|^2 + |\xi|^2\\
        &\le \Delta |\beta|^2 + C(1 +|\Rm|)|\beta|^2 +|\xi|^2.
    \end{align*}
    Note that $|\Phi|^2 = O(|T|^2) = O(|\Rm|)$.
\end{proof}

\begin{lem}\label{lemma barrier for norm}
    Define $u: M \times [0,T] \rightarrow \R$ by
    \[
    u(x,t) := e^{Nt}\rho(x)^{-k}
    \]
    for $k \in \bR_{\ge 0}$ and $N>0$ sufficiently large, and where the smooth function $\rho: M \rightarrow \R$ coincides with the radius function on the diffeomorphism-fixed ends $V_i$ and is such that $\rho \ge 1$. Then
    \begin{equation}\label{equation barrier inequality}
    \dd_t u - \Delta u - C(1 + |\Rm|)u \ge \tfrac{N}{2}u.
    \end{equation}
\end{lem}

\begin{proof}
    Let 
    \[
    L :=\Delta + C(1  + |\Rm|) 
    \]
    such that \eqref{equation barrier inequality} can be written as
    \[
    (\dd_t - L)u \ge 0.
    \]
    Fix a sufficiently large radius $R_0 > 0$ and separate $M$ into two domains, $M_{R_0} := \rho^{-1}((0,R_0))$ and $V := M \setminus M_{R_0}$. First, take $N > 0$ to be such that
    \begin{align*}
        \frac{N}{2} > \max_{M_{R_0} \times [0,T]} |L(\rho^{-k})| \rho^k
    \end{align*}
    so that on $M_{R_0}$, for any $t \in [0,T]$
    \begin{align*}
        \dd_t u - Lu &= Ne^{Nt}\rho^{-k} - L(e^{Nt}\rho^{-k})\\
        &\ge Ne^{Nt}\rho^{-k} - \big(\max_{M_{R_0} \times [0,T]} |L(\rho^{-k})| \rho^{k}\big) e^{Nt} \rho^{-k} \\
        & \ge \frac{N}{2}e^{Nt}\rho^{-k},
    \end{align*}
for $N$ sufficiently large. Now we must show that the inequality holds on the ends. We may select $R_0 > 0$ large enough such that the metric $g$ on the end $V$ differs from its asymptotic cone metric $g_{C}$ by a small amount. We let $ S := \Delta_{g} - \Delta_{g_{C}}$ be a second order operator with small coefficients. Then,
\begin{align*}
    \dd_t u - Lu &= \dd_t u - \Delta_{g_{C}} u - S u -  C(1 + |T| + |\Rm|) u \\
    &\ge Ne^{Nt}\rho^{-k} - e^{Nt}\Big(\dd^2_\rho \rho^{-k} + 6\rho^{-1}\dd_\rho \rho^{-k} + \rho^{-2} \Delta_{\Sigma} \rho^{-k}\Big) - C\rho^{-k}e^{N_t}, 
\end{align*}
where $C > 0$ is a large constant bounding the contributions from the zero-order term and the small coefficients of $S$. Continuing the computation, we find
\begin{align*}
    & \ge Ne^{Nt}\rho^{-k} - Ce^{Nt} \rho^{-k-2} - C\rho^{-k}e^{Nt} \\
    & \ge \frac{N}{2}e^{Nt}\rho^{-k},
\end{align*}
for $N > 0$ sufficiently large. Taking the maximum of the $N$'s computed on the domains $M_{R_0}$ and $V$ completes the proof. 
\end{proof}

\begin{cor}\label{corollary apriori estimate on L2 norm of solution}
    If $\beta, \xi \in \Omega^2([0,T] \times M_R)$ satisfy \eqref{equation evolution norm beta} and $\beta$ vanishes on the parabolic boundary $(M \times \{0\}) \cup (\dd M_R \times [0,T])$, there exists a constant $C>0$ depending on $T$ and $(M,\vp_0)$ but \emph{not} depending on $R$, such that for any rate $\la \in \R$ and $\eps \in (0,1)$ we have
    \begin{align*}
    \|\beta\|_{L^2_{\la}} &\le C\|\xi\|_{C^{0,0}_{\la - \eps}}.
    \end{align*}
\end{cor}

\begin{proof}Set $F := \|\xi\|^2_{C^{0,0}_{\la-\eps}}$. By Lemma \ref{lemma diff ineq for evol of norm} and Lemma \ref{lemma barrier for norm}, if we set $k = - 2\la + 2\eps$ on each end, then
\begin{align*}
    (\dd_t - L)(Fu - |\beta|^2) & \ge F\big( \tfrac{N}{2}e^{Nt}\rho^{2\la - 2\eps} \big) - |\xi|^2 \\
    &\ge \|\xi\|^2_{C^{0,0}_{\la -\eps}} \rho^{2\la - 2\eps} - |\xi|^2 \\
    &\ge \big(\rho^{-2\la + 2\eps}|\xi|^2 \big)\rho^{2\la - 2\eps} - |\xi|^2 \\
    & \ge 0,
\end{align*}
on $M_R$ by the definition of the $C^{l,k,\al}_{\la}$ weighted norms. Furthermore, restricted to the parabolic boundary
\begin{align*}
    Fu - |\beta|^2 &\ge Fu \\
    &\ge 0,
\end{align*}
since $|\beta|^2 = 0$ on $ (M_R \times \{0\}) \cup (\dd M_R \times [0,T])$. Thus, the weak maximum principle (e.g. \cite[Theorem 8 (ii), \S 7.1.4]{EvansPDE}) implies that
\[
Fu - |\beta|^2 \ge 0
\]
on $M_R \times [0,T]$. Rearranging this, we find that for each $(x,t) \in M_R \times [0,T]$,
\begin{align*}
    |\beta|^2(x,t) &\le e^{Nt}\rho^{2\la - 2\eps}\|\xi\|^2_{C^{0,0}_{\la-\eps}} \\
    \implies \rho^{-2\la}|\beta|^2(x,t) &\le e^{Nt}\rho^{-2\eps}\|\xi\|^2_{C^{0,0}_{\la-\eps}}.
\end{align*}
Integrate both sides of the inequality in space-time against the weighted volume $\rho^{-7}dV_{g_0(t)}$.
\begin{align*}
    \|\beta\|^2_{L^2_{\la}(M_R \times [0,T])} = \int_0^T \int_{M_R} \rho^{-2\la}|\beta|^2(x,t)\rho^{-7}dV_{g_0}dt &\le \int_0^T\int_{M_R} e^{Nt}\|\xi\|^2_{C^{0,0}_{\la-\eps}} \rho^{-7-2\eps}dV_{g_0}dt \\
    &\le e^{NT}C_{M,\eps}\|\xi\|^2_{C^{0,0}_{\la-\eps}},
\end{align*}
where $C_{M,\eps}>0$ is a fixed constant depending on $\eps > 0$ and on the approximate solution $d\beta_0(t) + \vp_0$ (which ultimately depends on the geometric data of the initial condition $\vp_0$). 
\end{proof}

\begin{rmk}
Note that we use the $C^{0,0}_{\la - \eps}$-norms, although $\xi$ is typically given as bounded in the $H^{r,s}_{\la-2}$-norms. However, the parabolic Sobolev embedding, Proposition \ref{proposition parabolic weighted Sobolev embedding}, takes a penalty of $\rho^{1+\ka}$ growth, so these Sobolev norms only bound the $C^{0,0}_{\la - \eps}$-norms. Fortunately, this is still enough to prove the desired a priori estimates.
\end{rmk}

\subsection{Uniform boundedness of the right inverse}\label{subsection uniform boundedness right inverse}

Next, the weighted Schauder estimates proved in Appendix \ref{subsection weighted schauder} will be deployed to obtain a uniform bound on $\|\cT_R\|$ for all $R > R_0$, the operator norm of the right inverse $\cT_R$ to the linearization $D\cP_0$, where $\cP$ is defined on the appropriate spaces corresponding to the domain $M_R$. Since we already have the estimate \eqref{equation time dependent potential sobolev estimate}, we focus our attention on the linear PDE for the primitive 2-forms, 
\begin{equation}\label{equation riemannian potential linear equation}
\dd_t \beta - \Delta_g \beta + \cR(\beta) - \Phi(d\beta) = \xi.
\end{equation}
In order to apply the Schauder estimates, the coefficients of this equation must be shown to satisfy the hypotheses of Proposition \ref{proposition weighted schauder estimates}.

\begin{lem}\label{lemma coefficients for weighted space}
    Label the order two, one and zero coefficients of \eqref{equation riemannian potential linear equation} by tensors $A$, $B$, and $C$ respectively, as in \eqref{equation linear parabolic pde}.  Then $A$, $B$, and $C$ satisfy \eqref{equation conditions on coeffs}.
\end{lem}

\begin{proof}
The term $\Phi(d\beta)$ is of the form $\Phi * \na \beta$, where $*$ denotes a tensor contraction, and thus contributes to the term $B * \na U$ (in the notation of Appendix \ref{section linear parabolic appdx}). By Bryant-Xu \cite{BryantXu}, $\Phi$ is algebraic in the torsion, and thus is $O(\rho^{-1})$ on an AC manifold with radial parameter $\rho$. 

The term $\cR(\beta)$ from the Weitzenb\"ock identity contributes to the zero-order coefficient $C$. It is a curvature-type term and thus has asymptotic decay $O(\rho^{-2})$.  

The Laplace-Beltrami operator is taken with respect to the time-dependent metric $g_0(t)$, corresponding to the approximate solution $\vp_0 + d\beta_0(t)$. However, each of the $d\beta_0(t)$'s is in a weighted Sobolev space and thus has zero-order decay of order $o(\rho^{-\frac32})$, first order decay of order $o(\rho^{-\frac52})$, and so on. Thus, $g_0(t)$ differs from $g_0(0)$ by a term of order $o(\rho^{-\frac32})$, and the difference of the corresponding Christoffel symbols decays of order $o(\rho^{-\frac52})$. 

Therefore, the difference of the Laplace-Beltrami operators $\De_{g_0(t)}-\De_{g_0(0)}$ yields coefficients contributing $A$, $B$, and $C$, which have sufficient decay to satisfy \eqref{equation conditions on coeffs}.  This completes the proof of the lemma.
\end{proof}

\begin{cor}\label{corollary coefficients for weighted space time derivs}
    If $\beta$ and $\xi$ have regularity in time of order $k$ and $k-1$ respectively the $L^2$-time derivative $\dd^k_t \beta$ satisfies an equation of the form \eqref{equation linear parabolic pde} with coefficients $A$, $B$ and $C$ satisfying \eqref{equation conditions on coeffs}, and with inhomogeneous term satisfying the hypotheses of Proposition \ref{proposition weighted schauder estimates}. 
\end{cor}

\begin{proof}
The corollary is proved by induction on $k$. Differentiate \eqref{equation riemannian potential linear equation} with respect to $t$ and analyze the resulting equation. 
\begin{align*}
    \dd_t \xi &= \dd_t\big(\dd_t \beta - \Delta_g \beta + \cR(\beta) - \Phi(d\beta)\big)  \\
    &= \dd_t (\dd_t \beta) - \dd_t (\Delta_g \beta) + (\dd_t\cR)(\beta) - \cR(\dd_t \beta) - (\dd_t \Phi)(\beta) - \Phi(\dd_t \beta)
\end{align*}
Notice that $ \Delta_g \beta $ is schematically of the form 
\[
g^{-1} * (\dd^2 \beta + \Ga * \dd \beta + (\dd \Ga + \Ga * \Ga) \beta). 
\]
and $\dd_t (\De_g \beta)$ is of the form
\begin{align*}
    &(\dd_t g^{-1}) * (\dd^2 \beta + \Ga * \dd \beta + (\dd \Ga + \Ga * \Ga) \beta) \\
    &\;\;+ g^{-1} * (\dd^2 (\dd_t \beta) + \dd_t \Ga * \dd \beta + \Ga * \dd (\dd_t \beta) +  \dd_t(\dd \Ga + \Ga * \Ga) \beta + (\dd \Ga + \Ga * \Ga) (\dd_t \beta))
\end{align*}
By the evolution equations \eqref{equation evolution metrics}, \eqref{equation evolution christoffel symbols}, and \eqref{equation evolution curvature}, the terms 
\begin{align*}
    &(\dd_t g^{-1}) * (\dd^2 \beta + \Ga * \dd \beta + (\dd \Ga + \Ga * \Ga) \beta) + g^{-1} * (\dd_t \Ga * \dd \beta +  \dd_t(\dd \Ga + \Ga * \Ga) \beta )
\end{align*}
are in the weighted Sobolev space with two fewer derivatives and weight $\la -4$, which is sufficient to be grouped into the inhomogeneous $F$ term in the context of Proposition \ref{proposition weighted schauder estimates}. The remaining term is the Laplace-Beltrami operator with respect to the time-dependent $g_0$, which has been discussed in the proof of Lemma \ref{lemma coefficients for weighted space}. This proves the base case $k =1$. 

Assume that the corollary is true for time derivatives up to order $k-1$, and in particular that all coefficients are either in weighted Sobolev spaces or are tensors in terms of $g,\Ga,\Rm$ or their derivatives in space and time. Repeat the calculation for the base case, for the Laplace-Beltrami operator as follows.
\begin{align*}
    &(\dd_t g^{-1}) * (\dd^2 (\dd_t^{k-1}\beta) + \Ga * \dd (\dd_t^{k-1}\beta) + (\dd \Ga + \Ga * \Ga) (\dd_t^{k-1}\beta)) \\
    &\;\;+ g^{-1} * \Big(\dd^2 (\dd_t^k \beta) + \dd_t \Ga * \dd (\dd_t^{k-1}\beta)+ \Ga * \dd (\dd_t^k \beta) \\
    &\qquad +  \dd_t(\dd \Ga + \Ga * \Ga) (\dd_t^{k-1}\beta) + (\dd \Ga + \Ga * \Ga) (\dd_t^k \beta)\Big)
\end{align*}
Exactly the same reasoning as before tells us that this can be split up into the Laplace-Beltrami operator applied to $\dd_t^k \beta$ plus a piece that can be absorbed into the inhomogeneous term. Similar reasoning can be applied to every other term to either into an inhomogeneous term with decay at least $\rho^{\la - 2k}$ or a linear coefficient of $\dd_t^k \beta$ in a PDE of the form \eqref{equation linear parabolic pde} which satisfies \eqref{equation conditions on coeffs}. This proves the induction step and concludes the proof of the corollary.
\end{proof}

\begin{prop}\label{proposition uniform boundedness of right inverse}
    The norm of the right inverse $\cT_R$ defined by the choices made in Proposition \ref{proposition unique solution sufficient regularity} and Proposition \ref{proposition boundedness of restricted forms} is independent of $R$.
\end{prop}

\begin{proof}
    The Schauder estimates in Proposition \ref{proposition weighted schauder estimates} applied to $\beta$ and $\dd_t^k\beta$ via Lemma \ref{lemma coefficients for weighted space} and Corollary \ref{corollary coefficients for weighted space time derivs}, the a priori estimates on the weighted $L^2$-norm of $\beta$ in Corollary \ref{corollary apriori estimate on L2 norm of solution}, and Proposition \ref{proposition boundedness of restricted forms} give an estimate for the norm of the right inverse $\cT_R$ that is independent of $R$. Let $\nu = d\xi + \de \om$ be a form in $H_{0,\la}^{r-1,s-3} \sU^3(M_R \times [0,T])$. By Proposition \ref{proposition boundedness of restricted forms}, we may uniquely select $\xi$ satisfying the bound \eqref{equation time dependent potential sobolev estimate} and which satisfies the compatibility conditions of Theorem \ref{theorem higher regularity linear par pde}.  Solving the linearized PDE as in Proposition \ref{proposition unique solution sufficient regularity}, we obtain a unique solution $\beta$ in $H^{r,s}_{0,\la +3} \Om^2_D (M_R \times I)$ whose $H^{r,s}_{\la + 3}$-norm is uniformly bounded by $\|\nu\|_{H^{r-1,s-3}_\la(M_R)}$, independent of $R$. Furthermore, taking the exterior derivative of the Dirichlet potential of $\nu$ as in Proposition \ref{proposition boundedness of restricted forms} to obtain a unique $\om$ with $H^{r-1,s-2}_{\la + 1}$-norm uniformly bounded by uniformly bounded by $\|\nu\|_{H^{r-1,s-3}_\la(M_R)}$, independent of $R$. Thus, we have uniquely defined the right inverse $\cT_R$,
    \[
    \cT_R(\nu) := (\beta, \om) \in H^{r,s}_{0,\la+3}\Om^2_D(M_R \times I) \oplus H^{r-1,s-2}_{0,\la+1}\Om^4(M_R \times I), 
    \]
    such that 
    \[
    D\cP_0(\beta, \om) = \nu,
    \]
    and
    \[
    \|\beta\|_{H^{r,s}_{\la+3}(M_R \times I)} + \|\om\|_{H^{r-1,s-2}_{\la+1} (M_R \times I)} \le C \|\nu\|_{H^{r-1,s-3}_\la(M_R \times I)},
    \]
    where $C > 0$ does not depend on $R$.
\end{proof}

\subsection{Existence on the whole space}\label{subsection existence everywhere}

The preceding sections and the appendices give all the results needed to prove Theorem \ref{theorem existence of inverse on domains}, which establishes the existence of sequences of time-dependent 2-forms $\beta_i(t)$ and 4-forms $\om_i(t)$ on $M_{R_i}$, where $R_i \rightarrow \infty$ such that 
\begin{enumerate}
    \item[(a)] the families $\vp_i(t) := d(\beta_i(t) +\beta_0(t)) + \vp_0$ are families of $G_2$-structures on the domains $(M_{R_i}, \dd M_{R_i})$;
    \item[(b)] there exists a uniform constant $\eps > 0$ such that 
            \[
                \cP(\beta_i(t),\om_i(t)) \equiv 0, \qquad \text{for } t \in [0, \eps].
            \]
\end{enumerate}

Without projecting to $H^{r-1,s-3}_{\la}\sU^3(M_{R_i})$, define
\begin{align*}
    \cQ(\beta_i(t)) &:= \dd_t(d(\beta_i(t) + \beta_0(t)) + \vp_0) + \De_{\vp_i}\vp_i + d(V(\vp_i)\lrcorner \vp_i) + \de_{\vp_0} \om_i\\
    &= d(\dd_t (\beta_i + \beta_0) + \de_{\vp_i}\vp_i  + V(\vp_i) \lrcorner \vp_i) + \de_{\vp_0} \om_i.
\end{align*}
In particular, $\cQ(\beta_i)$ is the sum of an exact form and a co-exact form, with potentials bounded in terms of $\beta_i$ and $\om_i$, the initial data $\vp_0$, and their derivatives. The constants in the implicit function theorem argument are independent of $i$ and thus on every compact set $K$, the $\beta_i$'s and $\om_i$'s restricted to $K$ comprise a sequence uniformly bounded in $H^{r,s}\Om^2(K \times [0, \eps])$ and $H^{r-1,s-2}\Om^4(K \times [0, \eps])$, respectively. By Rellich-Kondrachov, a subsequence of the $(\beta_i,\om_i)$'s has its limit in $H^{r-1,s-2}\Om^2(K \times [0, \eps]) \oplus H^{r-2,s-4}\Om^4(K \times [0, \eps])$. This implies that (up to a subsequence)
\[
d \big(\dd_t \beta_i + \de_{\vp_i}\vp_i  + V(\vp_i) \lrcorner \vp_i\big) + \de_{\vp_0} \om_i
\]
has a $H^{r-3,l-6}$-regular limit. Taking a diagonal subsequence, obtain $\beta$ as the $H^{r-1,s-2}_{loc}$-limit of the $\beta_i$'s and $\om$ as the $H^{r-2,s-4}_{loc}$-limit of the $\om_i$'s. 
\begin{prop}\label{proposition solution on whole space up to coexact}
Let $\vp = d(\beta + \beta_0) + \vp_0$ and let
\[
\cQ(\beta,\om) = d(\dd_t (\beta + \beta_0) + \de_\vp \vp  + V(\vp) \lrcorner \vp) + \de_{\vp_0} \om.
\]
Then,
\[
\cQ(\beta(t),\om(t)) = 0,
\]
for $t \in [0, \eps]$. 
\end{prop}

\begin{proof}
Notice that in the $H^{r-3,s-6}_{loc}$ topology,
\[
\cQ(\beta_i,\om_i) \rightarrow \cQ(\beta,\om) = d(\dd_t (\beta + \beta_0) + \de\vp  + V(\vp) \lrcorner \vp) + \de_{\vp_0} \om,
\]
where $\vp = d\beta + d\beta_0 + \vp_0$. For sufficiently large $r,s$, the $\om_i$'s and the $\beta_i$'s are uniformly bounded above in $C([0,\eps], H^1_{\la}\Om^\bullet(M_{R_i}))$. By the extension Lemma \ref{lemma extension of functions} and the weak lower semi-continuity of the $H^1_{\la}(M)$-norm, the loc-limits 
\[
\om, \; \dd_t (\beta + \beta_0) + \de\vp  + V(\vp) \lrcorner \vp
\]
are in $H^1_\la \Om^\bullet(M)$. Thus, 
\[
\cQ(\beta, \om) \in \im(d + \de_{\vp_0}) \subset L^2_{\la}\Om^3(M).
\]
Since the time-slices of the $\cQ(\beta_i,\om_i)$ are contained in the $L^2_\la$-orthogonal complement of the ``upper spectrum" $U_R$ (since the projection $\Pi_{\sU}$ is simply $L^2_{\la}$-orthogonal projection to $U_R$ on each time-slice), Corollary \ref{corollary small eigenvalues become harmonic} implies that $\cQ(\beta(t),\om(t))$ vanishes on each time-slice for $t \in [0,\eps]$. 
\end{proof}

\begin{thm}\label{theorem solution is exact solution}
    The time-dependent family of forms $\vp(t) = d(\beta(t) + \beta_0(t)) + \vp_0$ is a solution to Laplacian flow on $M \times [0,\eps]$ with initial condition $\vp_0$.
\end{thm}

\begin{proof}
    In Proposition \ref{proposition solution on whole space up to coexact} it was shown that for $t \in [0,\eps]$,
    \begin{align*}
        \cQ(\beta,\om) &= d(\dd_t (\beta + \beta_0) + \de_\vp \vp  + V(\vp) \lrcorner \vp) + \de_{\vp_0} \om \\
        &= 0.
    \end{align*}
    Thus, 
    \[
    d(\dd_t (\beta + \beta_0) + \de_\vp \vp  + V(\vp) \lrcorner \vp) = \de_{\vp_0} \om,
    \]
    and on each time-slice
    \[
    d(\dd_t (\beta + \beta_0) + \de_\vp \vp  + V(\vp) \lrcorner \vp) \in d(H^1_{\la}\Om^2(M)), \text{ and } \de_{\vp_0} \om \in \de_{\vp_0}(H^1_{\la}\Om^4(M)).
    \]
    However, by \cite[Proposition 4.31]{KarigiannisLotay}, $d(H^1_{\la}\Om^2(M))$ and $\de_{\vp_0}(H^1_{\la}\Om^4(M))$ are $L^2$-orthogonal closed subspaces of $L^2_\la \Om^3(M)$ and thus intersect only at 0. This implies that
    \[
    d(\dd_t (\beta + \beta_0) + \de_\vp \vp  + V(\vp) \lrcorner \vp) = 0
    \]
    for each time $t \in [0,\eps]$ and thus that $\vp$ is a solution of Laplacian flow on $M$ with initial condition $\vp_0$. 
\end{proof}

\section{Uniqueness of the Laplacian flow on a complete manifold}\label{section uniqueness}

The uniqueness of the flow constructed in the previous sections on an AC background will follow from a general uniqueness theorem for complete closed Laplacian flows with bounded curvature tensor. This theorem will be proved by Kotschwar's energy method in \cite{kotschwarEnergy}, with modifications for the Laplacian flow case drawn from Lotay-Wei's forward uniqueness result \cite[\S 6]{LotayWeiLaplacianFlow}.  
\begin{thm}\label{theorem uniqueness complete bdd curvature}
    Let $(M,\vp_0)$ be a closed $G_2$-structure on a noncompact manifold $M$ such that $(M,g_{\vp_0})$ is a complete Riemannian manifold with bounded curvature. If $\vp(t)$ and $\tvp(t)$ are smooth solutions to \eqref{equation laplacian flow} on $M \times [0,T]$, for which
    \[
    |\Rm (x,t)|_{g(t)} + |\widetilde \Rm (x,t)|_{\tg(t)} \le K_0
    \]
    on $M \times [0,T]$ for some constant $K_0 \ge 0$, then $\vp(t) = \tvp(t)$ for all $t \in [0,T]$. 
\end{thm}

\subsection{Preliminaries}

Before proving this theorem, recall the following results from Lotay-Wei \cite{LotayWeiLaplacianFlow}. The first are the local Shi-type estimates.

\begin{thm}\cite[Theorem 4.4]{LotayWeiLaplacianFlow}\label{theorem local shi estimates}
Let $K>0$ and $r> 0$. Let $M$ be a 7-manifold, $p \in M$, and $\vp(t)$, $t \in [0,1/K]$ be a solution to Laplacian flow \eqref{equation laplacian flow} for closed $G_2$-structures on an open neighborhood $U$ of $p$ containing $B_{g(0)}(p,r)$ as a compact subset, such that $K$ bounds the curvature $|\Rm|_{g(t)}$ for all $t$.

For any $k \in \bN$, there exists a constant $C = C(K,r,k)$ such that if $\La(x,t) \le K$ for all $x \in U$ and $t \in [0,1/K]$, then for all $y \in B_{g(0)}(p, r/2)$ and $t \in [0,1/K]$, we have
\[
|\na^k \Rm | + |\na^{k+1}T| \le C(K,r,k)t^{-\frac{k}{2}}.
\]
\end{thm}

\begin{rmk}
We will always take $r = 1$ in order to remove the $r$-dependence of the bound.
\end{rmk}

The energy quantity to be analyzed consists of certain quantities arising from tensors associated to the flows $\vp(t)$ and $\tvp(t)$. Following Lotay-Wei \cite[\S 6.1]{LotayWeiLaplacianFlow}, define
\[
\phi = \vp - \tvp, \; h = g - \tg, \; A = \na - \tna 
\]
\[
U = T- \widetilde T, \; V = \na T - \tna \widetilde T, \; S = \Rm - \widetilde \Rm.
\]
The evolution equations for these quantities were calculated in Lotay-Wei and are reproduced in the following lemma.
\begin{lem}\label{lemma evolution equations differences}\cite[Lemma 6.2]{LotayWeiLaplacianFlow} We have the following inequalities:
\begin{align}
    \bigg| \frac{\dd }{\dd t} \phi(t) \bigg|_{g(t)} &\le C (|V(t)|_{g(t)} + |A(t)|_{g(t)}); \label{equation ev diff g2 structures} \\
    \bigg| \frac{\dd }{\dd t} h(t) \bigg|_{g(t)} &\le C (|S(t)|_{g(t)} +|h(t)|_{g(t)} + |U(t)|_{g(t)}); \label{equation ev diff metrics} \\
    \bigg| \frac{\dd }{\dd t} A(t) \bigg|_{g(t)} &\le C(|A(t)|_{g(t)} +|\tna \widetilde \Rm|_{g(t)}|h(t)|_{g(t)} \nonumber \\
    &\qquad + |U(t)|_{g(t)} + |V(t)|_{g(t)} + |\na S (t)|_{g(t)}); \label{equation ev diff christoffels}\\
    \bigg| \frac{\dd }{\dd t} U(t) \bigg|_{g(t)} &\le C\ (|\phi(t)|_{g(t)} + |A(t)|_{g(t)} +|U(t)|_{g(t)} + |S(t)|_{g(t)} \nonumber \\
    &\qquad + |V(t)|_{g(t)} + |\na V (t)|_{g(t)}); \label{equation ev diff torsions}
\end{align}
\begin{align}
    &\bigg| \frac{\dd }{\dd t} V(t)  - \De V(t) - \Div \cV(t) \bigg|_{g(t)} \nonumber \\
    &\qquad \le C (|V(t)|_{g(t)} +(1+|\tna^2\widetilde T|_{g(t)})|A(t)|_{g(t)} +(1+ |\na \Rm|_{g(t)})|U(t)|_{g(t)}  \nonumber \\ 
    &\qquad \qquad +|S(t)|_{g(t)} + (1+|\tna^2\widetilde T|_{g(t)})|h(t)|_{g(t)} +(1+ |\na \Rm|_{g(t)})|\phi(t)|_{g(t)} \nonumber \\
    &\qquad \qquad \qquad + |\na S (t)|_{g(t)} + |\na V (t)|_{g(t)}), \label{equation ev diff deriv tors} 
\end{align}
where $\cV$ satisfies 
\[
|\cV(t)|_{g(t)} \le C(|\tna^2 \widetilde T|_{g(t)}|h(t)|_{g(t)} + |A(t)|_{g(t)});
\]
and
\begin{align}
    &\bigg| \frac{\dd }{\dd t} S(t)  - \De S(t) - \Div \cS(t) \bigg|_{g(t)} \nonumber \\
    &\qquad \le C (|V(t)|_{g(t)} + |\tna \widetilde \Rm|_{g(t)}|A(t)|_{g(t)} \nonumber \\
    &\qquad \qquad \qquad +|\na^2 T|_{g(t)}|U(t)|_{g(t)} +|S(t)|_{g(t)} + |\na V (t)|_{g(t)}), \label{equation ev diff curvatures} 
\end{align}
where $\cV$ satisfies 
\[
|\cS(t)|_{g(t)} \le C(|\tna \widetilde \Rm|_{g(t)}|h(t)|_{g(t)} + |A(t)|_{g(t)}).
\]
In the above inequalities, $\na$, $\De$, and $\mathrm{div}$ are the Levi-Civita connection, Laplacian and divergence on $M$ with respect to $g(t)$ and $C$ denotes uniform constants depending on the curvature upper bound $K_0$.  
\end{lem}

\begin{proof}
    The small modifications to the statement of \cite[Lemma 6.2]{LotayWeiLaplacianFlow} to make the dependence of the estimates on the derivatives $\na \Rm$, $\tna \widetilde \Rm$, $\na^2 T$, and $\tna^2 \widetilde T$ follow immediately from examining the computations in the proof of Lemma 6.2 in \cite{LotayWeiLaplacianFlow}. 
\end{proof}

Next, analogously to \cite[Lemma 6]{kotschwarEnergy} we will prove estimates for the decay as $t \rightarrow 0$ for the ``ODE" quantities $\phi$, $h$, $A$ and $U$.

\begin{lem}
Let $T^* = \max\{T,1\}$. Under the assumptions of Theorem \ref{theorem uniqueness complete bdd curvature}, we have $|\phi(p,t)|, |h(p,t)| < Nt$ and $|A(p,t)|, |U(p,t)| \le N \sqrt {t}$ on $M \times [0, T]$ for some constant $N = (K_0, T^*)$. 
\end{lem}

\begin{proof}
    The bounded curvature condition implies that the metrics $g(t), \tg(t), g_0$ are all uniformly equivalent on $[0,T]$ by a standard application of Gr\"onwall's inequality (c.f. \cite{ShiDeforming}). From this and the evolution equations \eqref{equation evolution christoffel symbols}, the Christoffel symbols are also uniformly bounded. For any $p \in M$,
    \begin{align*}
        |\phi(p,t)| \le N|\phi(p,t)|_{g_0} &\le N \int_0^t |V(p,s)|_{g_0} + |A(p,s)|_{g_0} ds \\
        &\le N \int_0^t |\Rm (p,s)|_{g_0} + |\widetilde \Rm (p,s)|_{g_0} \\
        &\qquad \qquad \qquad + | \Ga (p,s)|_{g_0} + | \widetilde \Ga (p,s)|_{g_0} ds\\
        &\le Nt.
    \end{align*}
    The bound $|\na T| \le C |\Rm|$ used implicitly in the second line follows from \cite[Proposition 2.4]{LotayWeiLaplacianFlow}. One can argue similarly for the bound on $|h|$, noting that $|T|$ is uniformly bounded since $|T|^2 = -R$. 

    The bounds on $|A|$ and $|U|$ are demonstrated with a similar argument, except that the integrand will be dominated for small $t$ by the first derivatives of the curvature. The Shi-type estimates in Theorem \ref{theorem local shi estimates} tell us that $|\na \Rm|$ and $|\tna \widetilde \Rm|$ are $O\big(t^{-1/2}\big)$. Thus, for any $0 < \eps < t$, we have
    \begin{align*}
        |A(p,t) - A(p, \eps)| &\le N|A(p,t) - A(p, \eps)|_{g_0} \\
        &\le N \int_{\eps}^t \Big( |\na \Rm(p,s)|_{g_0} + |\tna \widetilde \Rm(p,s)|_{g_0} + 1 \Big) ds \\
        &\le N \int_{\eps}^t s^{-1/2} ds\\
        &\le N(\sqrt{t} - \sqrt{\eps}). 
    \end{align*}
Send $\eps \rightarrow 0$ to obtain the result for $|A|$. The result for $|U|$ is proved along exactly the same lines.
\end{proof}

We will need to introduce an auxiliary function in order to define a weight function that decays rapidly enough to make our energy quantity well-defined. The following lemma in \cite{kotschwarEnergy} establishes the existence of such a function. 

\begin{lem}\label{lemma decay function}\cite[Lemma 5]{kotschwarEnergy}
    Suppose $\bar g(t)$ is a smooth family of complete metrics on $M \times [0,T]$ satisfying $\ga^{-1}\bar g(0) \le g(t)$. Let $x_0 \in M$ and define $\bar r(x) = \mathrm{dist}_{\bar g(0)}(x,x_0)$. Then, for any positive constants $L_1$ and $L_2$, there exists a positive constant $T' = T'(\ga, L_1,L_2,T)$ and a function $\eta: M \times [0,T'] \rightarrow \R$ that is smooth in $t$, Lipschitz everywhere and smooth $d\mu_g-a.e.$ on each $M \times \{t\}$, and that simultaneously satisfies the conditions
    \[
    -\frac{\dd \eta}{\dd t} + L_1 |\na \eta|^2_{\bar g(t)} \le 0,  \;\; \text{ and } \; e^{-\eta} \le e^{-L_2\bar r^2(x)},
    \]
    on $M \times [0,\tau]$ whenever $0 < \tau  \le T'$.
\end{lem}
We will choose $L_2 = B/8T$ for some $B>0$ such that $L_1 = 3$ and 
\[
\frac{\dd \eta}{\dd t} - 3|\na \eta|^2 \ge 0. 
\]
We can now define the four integral quantities whose sum will comprise our energy functional. Let $t \in (0,T]$ and $\beta \in  (0,1)$. Define 
\[
\cG (t) := \int_M (|V|^2 + |S|^2)e^{-\eta} d\mu, \;\; \cH(t) := t^{-1}\int_M (|h|^2 + |\phi|^2)e^{-\eta}d\mu,
\]
\[
\cI (t) := t^{-\beta}\int_M (|A|^2 + |U|^2)e^{-\eta} d\mu, \;\; \cJ(t) := \int_M (|\na S|^2 + |\na V|^2)e^{-\eta}d\mu,
\]
and let the energy quantity be defined as
\[
\cE(t) := \cG(t) + \cH(t) + \cI(t). 
\]

\begin{lem}\label{lemma well defined integral quantities}
The quantities $\cG(t), \cH(t), \cI(t),$ and $\cJ(t)$ are well-defined and $\cE(t)$ is differentiable on $(0, T]$ and $\lim_{t\searrow 0} \cE(t) = 0$.
\end{lem}

\begin{proof}
    The proof follows similar the analogous arguments in \cite[\S 2.2]{kotschwarEnergy}. The uniform comparability of the metrics $g(t),\tg(t),$ and $g_0$, the curvature bound $K_0$, and the Bishop-Gromov comparison theorem imply the following bounds on volume growth.
    \[
    \mathrm{vol}_{g(t)}(B_{g_0}(x_0, r)) \le Ne^Nr, 
    \]
    for some $N$. Since for each $0 < t < T$, the quantities $|V|^2$, $|S|^2$, $|h|^2$, $|\phi|^2$, etc. are bounded in space, the integrals $\cG(t), \cH(t), \cI(t),$ and $\cJ(t)$ are defined away from $0$. Since 
    \[
    |V|^2, |S|^2, t^{-1}(|h|^2 + |\phi|^2) \;\; \text{and } t^{-\beta} (|A|^2 + |U|^2) \searrow 0
    \]
    uniformly, the quantities $\cG(t), \cH(t), \cI(t)$ are defined at $0$. 
    
    For any $\eps > 0$, the Shi-type estimates (Theorem \ref{theorem local shi estimates}) and the evolution equations (Lemma \ref{lemma evolution equations differences}) imply that the time derivatives of $\phi, h, A, U, V$ and $S$ are uniformly bounded on $M \times [\eps, T]$. The derivatives $\frac{\dd \eta}{\dd t}$ and $\na \eta$ are of quadratic growth in the radial variable and are thus integrable. Finally, $\frac{\dd}{\dd t} d\mu = -\frac23 R d\mu$ by \eqref{equation evolution volume}. Thus, $\cE(t)$ is differentiable on $(0,T]$. 

    Finally, 
    \[
    \lim_{t\searrow 0} \cE(t) = \lim_{t\searrow 0} (\cG(t) + \cH(t) + \cI(t)) = 0.
    \]
    by the dominated convergence theorem.
\end{proof}

\subsection{Vanishing of the energy} The following proposition analogous to \cite[Proposition 7]{kotschwarEnergy} gives a differential inequality for $\cE$.

\begin{prop}\label{proposition diff ineq for E}
There exists $N_0 = N_0(K_0, T^*) >0$ and $T_0=  T_0(\beta) \in (0,T]$ such that $\cE'(t) \le N_0 \cE(t)$ for all $t \in (0,T_0]$. Hence $\cE \equiv 0$ on $(0,T_0]$.
\end{prop}

\begin{proof}
The proof is along the same lines as the proof of \cite[Proposition 7]{kotschwarEnergy}. For any $0<t\le T$, differentiation under the integral sign and integration by parts are valid, by the arguments in the proof of Lemma \ref{lemma well defined integral quantities}. The constants $C >0$ will denote universal constants and the constants $N > 0$ will depend at most on the parameters $\beta, K_0,$ and $T^*$.

When we differentiate the integral quantities $\cG, \cI,$ etc., we differentiate under the integral.  By the product rule, we must differentiate the time-dependent norms and the time-dependent volume element, which, by the evolution equations \eqref{equation evolution metrics} and \eqref{equation evolution volume}, will yield a term of the form $N\cG, N\cI$, and so on. 

It remains to estimate the other summands of the derivative of the integrands. We begin with $\cG$. 
\begin{align*}
    \cG' &\le N\cG + \int_M \bigg( 2  \bigg \langle \frac{\dd S}{\dd t} , S \bigg \rangle + 2  \bigg \langle \frac{\dd V}{\dd t} , V \bigg \rangle - \frac{\dd \eta}{\dd t}(|S|^2 + |V|^2) \bigg) e^{-\eta} d\mu.
\end{align*}
Substituting \eqref{equation ev diff curvatures} and \eqref{equation ev diff deriv tors}, applying the Shi-type estimates in Theorem \ref{theorem local shi estimates}, and using Young's inequality, we obtain the following inequality.
\begin{align*}
    \cG' &\le N\cG + (N + t^{-1+\beta})\cI + N\cH + \eps \cJ \\
    &\qquad + \int_M \bigg( 2  \bigg \langle \De S + \Div \cS , S \bigg \rangle + 2  \bigg \langle \De V + \Div \cV , V \bigg \rangle  - \frac{\dd \eta}{\dd t}(|S|^2 + |V|^2) \bigg) e^{-\eta} d\mu. 
\end{align*}
Next, we integrate by parts to estimate the integral on the right-hand side. 
\begin{align*}
    &\int_M  2  \bigg \langle \De S + \Div \cS , S \bigg \rangle e^{-\eta} d\mu \\
    &\qquad \le  - \int_M 2|\na S|^2 e^{-\eta} d\mu +  \int_M 2\Big( |\na \eta||\na S||S| + |\cS||\na S| + |\na \eta||\cS||S| \Big ) e^{-\eta} d\mu \\
    &\qquad \le - \int_M |\na S|^2 e^{-\eta} d\mu + \int_M \big( 3|\na \eta|^2|S|^2 + 3|\cS|^2 \big ) e^{-\eta} d\mu.
\end{align*}
Similarly,
\begin{align*}
    &\int_M  2  \bigg \langle \De V + \Div \cV , V \bigg \rangle e^{-\eta} d\mu \\
    &\qquad \le - \int_M |\na V|^2 e^{-\eta} d\mu + \int_M \big( 3|\na \eta|^2|V|^2 + 3|\cV|^2 \big ) e^{-\eta} d\mu.
\end{align*}
Summing these two expressions yields
\begin{align*}
    & \int_M \bigg( 2  \bigg \langle \De S + \Div \cS , S \bigg \rangle + 2  \bigg \langle \De V + \Div \cV , V \bigg \rangle  - \frac{\dd \eta}{\dd t}(|S|^2 + |V|^2) \bigg) e^{-\eta} d\mu \\
    & \qquad \le - \cJ  + \int_M \Big(  3|\na \eta| - \frac{\dd \eta}{\dd t}\Big) (|S|^2 + |V|^2) e^{-\eta} d\mu  + \int_M 3\big( |\cV|^2  + |\cS|^2 \big ) e^{-\eta} d\mu.
\end{align*}
By Lemma \ref{lemma decay function} and the estimates on $\cV$ and $\cS$ from Lemma \ref{lemma evolution equations differences}, this further reduces to
\begin{align*}
    &\le -\cJ + Nt^\beta \cI + N\cH.
\end{align*}
Combining all these results gives (for $0 < t \le T$)
\begin{equation}\label{equation derivative of G}
\cG' \le N\cG + (t^{\beta -1} + N) \cI + N\cH - (1-\eps) \cJ,
\end{equation}
where $t \le T^*$ and $t^\beta \le (T^*)^\beta$. 

Next, evaluate $\cH'$. Differentiating the metric quantities and the factor of $t^{-1}$, we obtain the bound
\[
\cH' \le (N-t^{-1})\cH + t^{-1}\int_M \bigg( 2 \bigg\langle \frac{\dd h}{\dd t}, h \bigg \rangle + 2  \bigg\langle \frac{\dd \phi}{\dd t}, \phi \bigg \rangle  - \frac{\dd \eta}{\dd t}(|h|^2 + |\phi|^2) \bigg)e^{-\eta} d\mu.
\]
Observe that Lemma \ref{lemma decay function} implies that $\dd_t \eta \le 0$, and apply the estimates \eqref{equation ev diff g2 structures} and \eqref{equation ev diff metrics} to obtain
\begin{equation}\label{equation derivative of H}
\cH' \le (N - (1/2)t^{-1}) \cH + t^{\beta-1} \cI + N\cG.
\end{equation}

Finally, evaluate $\cI'$. Differentiating the metric quantities and $t^{-\beta}$ gives 
\begin{align*}
    \cI' \le (N-\beta t^{-1})\cI + t^{-\beta} \int_M \bigg( 2 \bigg\langle \frac{\dd A}{\dd t}, A \bigg \rangle + 2  \bigg\langle \frac{\dd U}{\dd t}, U \bigg \rangle  - \frac{\dd \eta}{\dd t}(|h|^2 + |\phi|^2) \bigg)e^{-\eta} d\mu.
\end{align*}
We use \eqref{equation ev diff christoffels} and \eqref{equation ev diff torsions} and Young's inequality to estimate
\begin{align*}
    t^{- \beta}\bigg| \bigg\langle \frac{\dd A}{\dd t}, A \bigg \rangle \bigg| + t^{-2\beta} \bigg| \bigg\langle \frac{\dd U}{\dd t}, U \bigg \rangle \bigg|
    & \le  (N + Ct^{-\beta})(|A|^2 + |U|^2) + Nt^{-1}(|\phi|^2 + |h|^2) \\ 
    &\qquad \qquad+ N(|S|^2 + |V|^2) + \tfrac12(|\na S|^2 + |\na V|^2).
\end{align*}
Thus, 
\begin{equation}\label{equation derivative of I}
    \cI' \le (N-\beta t^{-1} + Ct^{-\beta})\cI + N\cH + N\cG + \tfrac12 \cJ.
\end{equation}
Combining \eqref{equation derivative of G}, \eqref{equation derivative of H}, and \eqref{equation derivative of I} to obtain
\[
\cE'(t) \le N\cE(t) -(1/2)t^{-1} \cH - t^{-1}(\beta - 2t^\beta - Ct^{1-\beta})\cI(t).
\]
For sufficiently small $T_0$, and $N_0(K_0, \beta, T^*)$ sufficiently large, 
\[
\cE'(t) \le N_0 \cE(t),
\]
on $(0,T_0]$. Since $\cE(t) \searrow 0$ as $t \searrow 0$, Gr\"onwall's inequality implies that $\cE \equiv 0$ on $[0, T_0]$.
\end{proof}

\begin{proof}[Proof of Theorem \ref{theorem uniqueness complete bdd curvature}]
The theorem follows by iterating Proposition \ref{proposition diff ineq for E} over the whole interval $[0,T]$.
\end{proof}

\section{Preservation of the AC condition and long-time existence}\label{section long time behavior}

In this section, we prove two facts. The first is that, given an upper bound on the norm of the full curvature tensor along a flow $\vp(t)$, the AC condition is preserved and the short-time existence theorem can be reapplied at any time-slice to show that the flow continues as long as the curvature bound holds. The second is that since the curvature bound implies the preservation of the AC condition, a doubling time estimate for the norm of the curvature can be established, which gives a lower bound for the existence time of an AC flow $\vp(t)$, depending only on the geometric data of the initial condition $\vp(0) = \vp_0$.

We begin by proving the first fact. The key to applying the short-time existence theorem in the bounded curvature regime is to show that, given a bound on curvature, the curvature decays at comparable rates on different time-slices. To do this, we will need to apply Theorem \ref{theorem ecker huisken max}, the Ecker-Husiken maximum principle, along the same lines as in \cite{DaiMa}. The following lemma introduces an auxiliary function necessary for this argument. 

\begin{lem}\cite[Lemma 6]{DaiMa}\label{lemma aux function}
Let $\vp(t)$ be a solution of Laplacian flow \eqref{equation laplacian flow} over $[0,T]$ with AC initial condition $\vp_0$. Assume that the associated flow of metrics $g(t)$ has a uniform curvature bound $|\Rm_{g(t)}|\le K$. Then, for sufficiently large $R$, there is a smooth positive function $f$ on $M$ such that
\[
f(x) = C_0 \gg 1, \qquad \text{for } x \in B_R
\]
\[
cd^t(x) \le f(x) \le C d^t(x). \qquad \text{for } x \in M - B_R,
\]
where $d^t$ is the distance function for the $g(t)$ metric from a fixed point $p$. Moreover, 
\[
f \ge C_0, \; \; |\na^t f| \le C_1, \;\; |\De^t f| \le C_2.
\]
\end{lem}

\begin{proof}
    The proof is almost exactly the same as that of \cite[Lemma 6]{DaiMa}, except that instead of using the Euclidean norm to estimate the time-dependent gradient and Laplacian, we use the AC metric $g(0)$ and compare this metric to $g(t)$ via \eqref{equation compare metrics bdd K}.  Note that the constants $c$, $C$, $C_1$ and $C_2$ depend on $T$. In our case, if $\rho$ represents the standard ``radius" function with respect to $g(0)$ on the AC ends, $f(x)$ can be said to coincide with $\rho(x)$ for $d^0(x)$ sufficiently large.
\end{proof}

Given a flow of AC $G_2$-structures $\vp(t)$ satisfying \eqref{equation laplacian flow}, it will be useful to solve the heat equation on this moving background with initial conditions quadratically decaying along the ends.

\begin{lem}\label{lemma existence of heat solution}
Let $u_0$ be a smooth function on $M$ such that $u_0$ decays at the rate $O(d(x)^{-2})$, where $d$ represents the distance in the $g(0)$-metric from a fixed point $p \in M$. Then, there exists $u(x,t)$ defined for $t \in [0,T]$ such that
\begin{equation}\label{equation moving heat equation}
u_t = \De^t u,
\end{equation}
with initial condition $u(0) = u_0$. Furthermore, if $u_0 \ge 0$, then $u \ge 0$ for all $t \in [0,T]$. 
\end{lem}

\begin{proof}
Let $R_0 >0$ be a large positive constant, and consider $R > R_0$. Consider the family of Cauchy-Dirichlet problems
\begin{equation}\label{equation heat equation with boundary}
    \left\{
    \begin{array}{ll}
      u_t =   \De^t u \;\; &\text{on } M_R \times [0,T]\\
      u(0) = u_0 \;\; &\text{on } M_R \times \{0\}\\ 
      u(x,t) = u_0(x) \;\; &\text{ for }(x,t) \in  \dd M_R \times [0,T].
    \end{array}
  \right.
\end{equation}
Transforming these into Dirichlet problems with homogeneous boundary values in the usual way, we may apply Theorem \ref{theorem existence of linear weak solution} to obtain a weak solution in $H^{1,1}(M)$. We can then use the maximum principle, followed by the interior Schauder estimates on an AC background (see the proof of Proposition \ref{proposition weighted schauder estimates}) to obtain uniform bounds on the derivatives of $u$ on $M_{R-1} \times [0,T]$ to arbitrary order. These local $C^k$-bounds may depend on $k$ but do not depend on $R$. Therefore, we may take a $C^\infty_{loc}$-limit of the solutions to these Dirichlet problems for a sequence of $R_i \rightarrow \infty$, which we call $u_\infty$. It is clear that $u_\infty(t)$ satisfies \eqref{equation moving heat equation}, has initial condition $u(0) = u_0$, and is defined for $t \in [0, T]$. 

Additionally, if $u_0 \ge 0$, the solution $u(x,t)$ of \eqref{equation heat equation with boundary} must be non-negative for all $t \in [0,T]$ by the parabolic maximum principle. Thus, the limit $u_\infty \ge 0$ for all time. 
\end{proof}

If $u_0 \ge 0$, then the next proposition shows that the solution $u(x,t)$ found in Lemma \ref{lemma existence of heat solution} has spatial decay rate comparable to the decay rate of $u_0$, up to a multiplicative constant depending on $T$. 

\begin{prop}\label{proposition comparable decay rates}
Let $u_0$ be a positive initial condition such that 
\[
u_0(x) = O(\rho^{-\sg}), \frac{1}{u_0(x)} = O(\rho^{\sg}).
\]
Let $u(x,t)$ be the solution of the heat equation as in Lemma \ref{lemma existence of heat solution}, let $f$ be the auxiliary function defined in Lemma \ref{lemma aux function}, and let $w := f^\sg u$, for $\sg > 0$. Then, there exist positive constants $c(T), C(T)$ such that 
\[
0<de^{-bt} \le c(T) \le C(T) \le De^{bT},
\]
where $D$,$d$, and $b$ are constants depending only on $u_0$ and $g(0)$, and
\[
c(T) \le \min_{M \times [0,T]} w  \le \max_{M \times [0,T]} w \le C(T).
\]
In particular, 
\[
u(x,t) \le C(T) f(x)^{-\sg}. 
\]
\end{prop}

\begin{proof}
The proof follows that of \cite[Theorem 4]{DaiMa} almost exactly. All that needs to be proved is that $|\na w|$ is bounded, in order to show that $w$ satisfies the hypothesis \eqref{equation growth condition} in the Ecker-Huisken theorem. Since $|\na f| \le C$, it suffices to show that $|\na u|$ is bounded. However, this is an immediate consequence of the uniform Schauder estimates in the proof of Lemma \ref{lemma existence of heat solution}. Thus,
\[
|\na w| \le C(T)f^{\sg + 1},
\]
which clearly satisfies \eqref{equation growth condition}. After this is established, the proof of \cite[Theorem 4]{DaiMa} can be followed word for word. The bound on the constant $C(T)$ is implied by Dai and Ma's proof. 

The lower bound can be proved by following the proof of \cite[Theorem 4]{DaiMa} and applying the Ecker-Huisken \textit{minimum} principle to the differential inequality 
\[
    (\dd_t - \De^t)(w - de^{-bt}) \ge -b (w - de^{-bt}) - 2\na \log (f^\sg) \na (w - d e^{-bt}),
\]
where $w(x,0) - d > 0$. This implies that there exists a uniform constant $c$ such that 
\[
c \le \min_{M \times [0,T]} w,
\]
which yields the lower bound. 
\end{proof}

The next proposition shows that the quadratically weighted norm of the curvature is uniformly bounded along the flow, depending only on the curvature bound $K_0$ and the existence time $T$.

\begin{prop}\label{proposition curvature quadratic decay}
Let $\vp(t)$ satisfy \eqref{equation laplacian flow} on $M \times [0,T]$ with initial condition $\vp_0$, an AC $G_2$-structure. Suppose further that the curvature of the associated metrics $g(t)$ is uniformly bounded by a constant $K_0$. Then, there is a positive constant $0 < C(T,K_0) = De^{KT}$, where $D,K$ are universal constants depending on $g(0)$ and $K_0$, such that 
\[
|\Rm_{g(t)}|_{g(t)} \le C \rho^{-2}. 
\]
\end{prop}

\begin{proof}
Let $u_0$ be a smooth function equal to $C_0\rho(x)^{-4}$ such that it dominates $|\Rm(g_0)|_{g_0}^2$. Let $u$ be the solution to the heat equation with initial condition $u_0$, as in Lemma \ref{lemma existence of heat solution}. We must show that $|\Rm| - e^{Kt}u$ satisfies the conditions of the Ecker-Huisken maximum principle for some $K >0$. The evolution equation \eqref{equation evolution curvature} tells us that
\begin{align*}
\dd_t |\Rm|^2 &= 2 \langle \Rm , \Delta \Rm + \Rm * \Rm + \Rm * T * T + \nabla^2 T * T + \nabla T * \nabla T \rangle \\
&\le \De|\Rm|^2 - 2|\na \Rm|^2 + CK_0 |\Rm|^2 +  \eps |\na \Rm|^2 \\
&\le \De|\Rm|^2  + CK_0 |\Rm|^2.
\end{align*}
Thus, we can see that \eqref{equation ecker huisk diff ineq} holds for $|\Rm|^2 - e^{Kt}u$. It remains to bound $|\na \Rm|^2$ for \eqref{equation growth condition}. Notice that the short-time existence theorem for AC initial conditions (Theorem \ref{theorem solution is exact solution}) gives a solution with uniformly bounded derivatives on a time interval $[0,\eps]$. Thus, the curvature remains uniformly bounded until at least $t = \eps$. Then, we can use the uniform curvature bound $K_0$ and the Shi-type estimates from Theorem \ref{theorem local shi estimates}, to show that the derivatives of curvature are uniformly bounded on the time interval $[\eps, T]$. Thus, $|\na \Rm|^2$ is uniformly bounded on $[0,T]$ and the growth estimate \eqref{equation growth condition} is satisfied.  

Applying the Ecker-Huisken maximum principle (Theorem \ref{theorem ecker huisken max}) to $|\Rm|^2 - e^{Kt}u$ gives us
\[
|\Rm|^2 \le Ce^{CKt}\rho^{-4}, 
\]
which proves the proposition.
\end{proof}

\begin{cor}\label{corollary higher order curvature decay}
Let $\vp(t)$ satisfy \eqref{equation laplacian flow} on $M \times [0,T]$ with initial condition $\vp_0$, an AC $G_2$-structure satisfying \eqref{equation ac condition order k}. Suppose further that the curvature of the associated metrics $g(t)$ is uniformly bounded by a constant $K_0$. Then, there is a positive constant $0 < C(T,K_0,k) = De^{KT}$, where $D,K$ are universal constants depending on $k$, $K_0$, and the higher order data of $g(0)$ such that 
\[
|\na^k \Rm_{g(t)}|_{g(t)} \le C \rho^{-2-2k}. 
\]
\end{cor}

\begin{proof} The proof is by induction. The base case, $k= 0$, is proved in Proposition \ref{proposition curvature quadratic decay}.
    We can repeat the proof of Proposition \ref{proposition curvature quadratic decay} almost verbatim for $|\na^k \Rm|^2$, once the evolution inequality for this quantity is established. Consider
    \begin{align*}
        \tfrac12 \dd_t |\na^k \Rm|^2 &= \langle  \na^k \Rm, \dd_t \na^k \Rm \rangle \\
        &= \bigg\langle  \na^k \Rm,  \na^k \dd_t \Rm + \sum_{i=1}^k C_i\na^{i-1}(\na \Rm) * \na^{k-i}\Rm \bigg\rangle \\
        &= \bigg\langle  \na^k \Rm,  \na^k \Big( \De \Rm + \Rm * \Rm + \Rm * T * T \\
        &\qquad \qquad + \na^2 T * T + \na T * \na T \Big) +  \sum_{i=1}^k C_i \na^{i-1}(\na \Rm) * \na^{k-i}\Rm \bigg\rangle \\
        &= \bigg\langle  \na^k \Rm,  \De \na^k \Rm +  \sum_{j=0}^k C_j \na^j \Rm * \na^{k-j} \Rm \\
        &\qquad + \na^k \Big(  \Rm * \Rm + \Rm * T * T + \na^2 T * T + \na T * \na T \Big) \\
        &\qquad \qquad+ \sum_{i=1}^k C_i \na^{i-1}(\na \Rm) * \na^{k-i}\Rm \bigg\rangle.
    \end{align*}
    Distributing the derivatives, taking norms, and subtracting the $|\na^{k+1}\Rm|^2$ term yields the following inequality. 
    \begin{align*}
        \dd_t |\na^k \Rm|^2 &\le \De |\na^k \Rm|^2 + C(K_0,k)\bigg(\sum_{i=1}^{k-1} |\na^i \Rm|^2|\na^{k-i} \Rm|^2 + |\nabla^k \Rm|^2\bigg).
    \end{align*}
    Now, take a positive function $u_0 := C\rho^{-4 -4k}$ such that $u_0$ dominates $|\na^k \Rm|^2$ at $t=0$, and consider the heat solution $u$ (as in Lemma \ref{lemma existence of heat solution}) starting at $u_0$.  By Proposition \ref{proposition comparable decay rates}, this solution $u(x,t)$ is bounded below by $c(T)f^{-4-4k}$. Thus, by the inductive hypothesis, there exists a large $C>0$ depending on $T$ and $k$, and the initial conditions such that $Cu$ dominates $|\na^{k-i} \Rm||\na^i \Rm|$ for all $t \in [0,T]$. 
    
    We redefine $u$ to be equal to $Cu$ and take
    \begin{align*}
        \dd_t \big(|\na^k \Rm|^2 - e^{Kt}u\big) &\le \De 
        \big(|\na^k \Rm|^2 - e^{Kt}u|) + C(K_0,k)|\na^k \Rm|^2 - Ke^{Kt}u \\
        &\qquad + C(K_0,k)\sum_{i=1}^{k-1} |\na^i \Rm||\na^{k-i} \Rm|\\
        &\le \De \big(|\na^k \Rm|^2 - e^{Kt}u) + C(K_0,k)\big(|\na^k \Rm|^2 - e^{Kt}u\big),
    \end{align*}
    for $K$ sufficiently large depending on $K_0$ and $k$. The same argument involving the Shi-type estimates from Proposition \ref{proposition curvature quadratic decay} shows that $\na (|\na^k \Rm|^2-e^{Kt}u)$ satisfies \eqref{equation growth condition}.  Applying the Ecker-Huisken maximum principle to this function and invoking Proposition \ref{proposition comparable decay rates} to obtain the decay rate of $u$ on $[0,T]$ yields the corollary.
\end{proof}

The next lemma shows that if the curvature tensor is uniformly bounded, then if a flow $\vp(t) - \vp_0$ is in $H^{r,s}_\la$ for $t \in [0, T)$, then an approximate extension with good properties (see Lemma \ref{lemma approx extension}) can be found on the time interval $t \in [0, T+\ga]$ for some $\ga > 0$.

\begin{lem}\label{lemma approximate solution for AC preservation}
Suppose that $\vp_0$ is an initial closed AC $G_2$-structure of sufficiently high regularity, which satisfies \eqref{equation ac condition order k} for sufficiently large $k$. Let $\vp(t)$ be a flow of closed AC $G_2$-structures defined for $t \in [0, T)$ such that $\vp(0) = \vp_0$, $|\Rm_{g(t)}|_{g(t)} \le K_0$ and 
\[
\theta(t) = \vp(t) - \vp_0 \in H^{r+1, s+2}_{\la} \Om^3(M \times [0,T)),
\]
for all $t \in [0,T)$. There is a constant $\ga >0$ (depending only on $T$, $K_0$ and $\vp_0$) such that there is a family parametrized by $\tau$ (with $0 < T-\tau \ll 1$) of approximate extensions of $\vp(t)$, defined on $[0,T+\ga]$ and which coincide with $\vp(t)$ on $[0,\tau] \subset [0,T]$, which satisfies the conditions of Lemma \ref{lemma approx extension}. Furthermore, given $0 < \eps \ll 1$, there are constants $\ka, \tau >0$ which can be chosen depending only on $\eps$, $T$, $K_0$ and $\vp_0$ such that the form $\eta_\ka$ defined in Lemma \ref{lemma approx extension} vanishes on $[0, T + \ka/2] \subset [0,T+\ga]$ and is $\eps$-close to $\eta_0$ (defined in Lemma \ref{lemma approx extension}). 
\end{lem}

\begin{proof}
    Fix $0 < \eps \ll 1$. Given $\tau < T$, we may apply Lemma \ref{lemma approx extension} to obtain an approximate extension to $[0, \tau']$, for $\tau' > \tau$ and a $\ka(\tau) > 0$ such that $\eta_{\ka}$ vanishes on $[0, \tau + \ka(\tau)] \subset [0, \tau']$. We must prove that $|\tau'-\tau|$ and $|\ka (\tau)|$ are bounded below independently of the choice of $\tau$. This problem reduces to showing that constants $C_k$ in the curvature bounds \eqref{equation long time curvature bounds} are independent of $\tau$, i.e. that they depend only on $T$, $K_0$, and $\vp_0$.

    Assume that $\vp_0$ satisfies \eqref{equation ac condition order k} up to order $k \gg r+ s$. Then, by Proposition \ref{proposition curvature quadratic decay} and Corollary \ref{corollary higher order curvature decay} combined with \eqref{equation compare metrics bdd K}
    \[
    |\na^k \Rm_{g(t)}|_{g(0)} \le D_ke^{K_k\tau} \rho^{-2-2k} \le D_ke^{K_kT} \rho^{-2-2k},
    \]
    for all $t \in [0,\tau]$ where $D,K$ depend on $K_0$, $k$, $\vp_0$, and the covariant derivatives of of $\vp_0$. Since we may let 
    \begin{equation}\label{equation extension curvature norms}
    C_k := D_ke^{K_kT},
    \end{equation}
    for all $k$ under consideration, for each $\tau < T$, the proof of Lemma \ref{lemma approx extension} makes it clear that $|\tau'-\tau|$ and $|\ka (\tau)|$ are bounded below independently of the choice of $\tau$. Thus, we obtain the desired approximate solution by selecting $\tau$ sufficiently close to $T$. 
\end{proof}

In the next proposition, we prove that with the correct choice of $\eps > 0$, there is a constant $\ka > 0$ and solution to Laplacian flow \ref{equation laplacian flow} on $M \times [0, T + \ka]$ that extends $\vp(t)$ and satisfies the correct decay properties. 

\begin{prop}\label{proposition preservation of AC}
Suppose that $\vp_0$ is an initial closed AC $G_2$-structure of sufficiently high regularity, which satisfies \eqref{equation ac condition order k} for sufficiently large $k$. Let $\vp(t)$ be a flow of closed AC $G_2$-structures defined for $t \in [0, T)$ such that $\vp(0) = \vp_0$, $|\Rm_{g(t)}|_{g(t)} \le K_0$ and 
\[
\theta(t) = \vp(t) - \vp_0 \in H^{r+1, s+2}_{\la} \Om^3(M \times [0,T)),
\]
for all $t \in [0,T)$. Then, there exists a constant $\ka > 0$ such that the flow $\vp(t)$ continues to exist for $t \in [0, T + \ka]$ and 
\[
\theta(t) = \vp(t) - \vp_0 \in H^{r+1, s+2}_{\la} \Om^3(M \times [0,T + \ka]).
\]
In particular, the flow exists and remains AC as long as the curvature is uniformly bounded above.
\end{prop}

\begin{proof}
The extension is obtained by linearizing around the approximate solution obtained in Lemma \ref{lemma approximate solution for AC preservation} and applying the proof of Theorem \ref{theorem solution is exact solution} to obtain a solution on $[0, T + \ka/2]$ (in the notation of Lemma \ref{lemma approximate solution for AC preservation}). In particular, one aims to invert the linearized operator up to the approximate kernel on a sequence of space-time domains $M_{R_i} \times [0,T+ \ga]$, where $R_i \rightarrow \infty$. If $\sL_{R_i}$ is the linearization of the nonlinear map $\cP$ defined in \eqref{equation fix notation mapping} on this domain, it suffices to show that the norm of its right inverse $\cT_{R_i}$ is uniformly bounded, independent of $R_i$.

This question further reduces to bounding the coefficients of the Schauder estimates in Proposition \ref{proposition weighted schauder estimates} applied to the potential $\beta$ and its time derivatives $\dd_t^k\beta$. In the proof of Proposition \ref{proposition uniform boundedness of right inverse}, the uniform boundedness of the Schauder constants is established via Lemma \ref{lemma coefficients for weighted space} and Corollary \ref{corollary coefficients for weighted space time derivs}. These results use the uniform boundedness in the weighted norms of the geometric data the approximate solution $\vp_0 + \theta$ around which the nonlinear operator is linearized.  Since we are linearizing around the approximate extension defined in Lemma \ref{lemma approximate solution for AC preservation} on $M \times [0, T + \ga]$, the relevant norms on the higher order data of the approximate extension are bounded up to order $k$ by
\[
\max_{i \le k+r+s+2} 2|C_i|,
\]
where the $C_i$ are those defined in \eqref{equation extension curvature norms}, and we assume that $\ga < 1$. Applying these uniform bounds to Lemma \ref{lemma coefficients for weighted space} and Corollary \ref{corollary coefficients for weighted space time derivs} gives us uniform control of the Schauder constants for the relevant weighted parabolic equations. The a priori estimates obtained in Corollary \ref{corollary apriori estimate on L2 norm of solution} and Proposition \ref{proposition boundedness of restricted forms} can be obtained exactly as before. Therefore, it is possible to run the proof of Proposition \ref{proposition uniform boundedness of right inverse} to obtain the uniform boundedness of the right inverse $\cT_{R}$ in the appropriate weighted Sobolev norms. 

By the quantitative inverse function theorem (Theorem \ref{theorem quantitative inverse function theorem}), the operator norm bound $\|\cT_R\| < C$ for all $R \gg 0$ yields a radius $\eps > 0$ in which we may solve the inhomogeneous non-linear equation. Set this constant $\eps$ to be the parameter $\eps$ in the statement of Proposition \ref{lemma approximate solution for AC preservation}, and choose $\ka > 0$ so that the form $\eta_\ka$ is $\eps$-close to $\eta_0$ in $H^{r, s}_\la \Om^3 (M_R \times [0,T+\ga])$. 

Run the proof of short-time existence (Theorem \ref{theorem main theorem}) to obtain an exact solution $\hat \vp(t)$ to \eqref{equation laplacian deturck boundary} on $[0, T + \ka /2]$ with initial condition $\vp_0$. The inverse function theorem implies that \[
\hat \vp(t) - \vp_0 \in H^{r+1, s+2}_{\la} \Om^3(M \times [0,T + \ka])
\]
and the uniqueness theorem (Theorem \ref{theorem uniqueness complete bdd curvature}) implies that $\hat \vp(t)$ coincides with $\vp(t)$ on $[0,T)$ and is really an extension. Relabeling $\hat \vp(t)$ by $\vp(t)$ and $\ka/2$ by $\ka$ yields the statement of the proposition. 
\end{proof}

The preservation of the AC condition (for initial conditions satisfying \eqref{equation ac condition order k} for sufficiently large $k$) can be used to prove the following doubling time estimate for the curvature.

\begin{cor}\label{corollary doubling time estimate}
Let $\vp(t)$ be a solution to the Laplacian flow \eqref{equation laplacian flow} with an AC initial condition $\vp_0$ satisfying \eqref{equation ac condition order k} for sufficiently large $k$. Following \cite{LotayWeiLaplacianFlow}, define 
\[
\La(x,t) = \Big( |\na T(x,t)|^2_{g(t)} + |\Rm(x,t)|^2_{g(t)} \Big)^\frac12.
\]
There exists a constant $C > 0$ such that the flow $\vp(t)$ can be extended to $\big[0, \tfrac{1}{C\La(0)} \big]$, remains AC, and $\La(t) \le 2\La(0)$.
\end{cor}

\begin{proof}
    The proof of the doubling time is identical to that of Lotay-Wei in \cite[Proposition 4.1]{LotayWeiLaplacianFlow}. The same methods can be used to establish the differential inequality 
    \[
    \frac{\dd}{\dd t} \La(x,t)^2 \le \De \La (x,t)^2 - (|\na \Rm|^2 + |\na^2 T|^2) + C\La(x,t)^3. 
    \]
    Remove the squares by differentiating and absorbing. Then consider the Lipschitz function $\La(t) = \max_M \La(x,t)$. Since $\vp(t)$ is AC for as long as the flow is defined on $[0,T)$, for each $t \in (0,T)$ there is a $\de >0$ and a compact subset $K$ such that the maximum of $\La(x,t')$ is achieved inside $K$ for all $t' \in (t-\de, t+\de)$. Thus, we may apply the non-compact maximum principle \cite[Theorem 2.1.1]{Mantegazza} to $\La(t)$ to deduce that
    \begin{equation}\label{equation curvature max bound}
    \La(t) \le \frac{\La(0)}{1- \tfrac12 C \La(0)t},
    \end{equation}
    for as long as the flow exists. However, if $T < \tfrac{1}{C\La(0)}$, then the uniform curvature bound $\La(t) \le 2\La(0)$ for $t \in [0,T)$ ensures that the flow may be continued and remains AC for a short time beyond $T$, by Proposition \ref{proposition preservation of AC}. 
\end{proof}

Proposition \ref{proposition preservation of AC} and Corollary \ref{corollary doubling time estimate} give the following long-time existence result. 

\begin{cor}\label{corollary existence time lower bound}
    The flow $\vp(t)$ satisfying \eqref{equation laplacian flow} with AC initial condition $\vp_0$ exists as long as $\La(t) < \infty$. Furthermore, it can be uniquely extended to the time interval $\big[0, \tfrac{2}{C\La(0)}\big)$. 
\end{cor}

\begin{proof}
    If $\La(t) < C$ for all $t \in [0,T')$, then by Proposition \ref{proposition preservation of AC}, the flow can be extended to $[0,T'+\eps)$, for $\eps >0$. The uniqueness of this extension follows from Theorem \ref{theorem uniqueness complete bdd curvature}. The lower bound on the existence time follows from the inequality \eqref{equation curvature max bound} used to prove the doubling time estimate in Corollary \ref{corollary doubling time estimate}.
\end{proof}

\subsection{Proof of Theorem \ref{theorem main theorem}} If $\vp_0$ is a smooth, asymptotically conical $G_2$-structure that converges to its asymptotic cone to infinite order, apply Theorem \ref{theorem solution is exact solution} to obtain short-time existence. Then, apply Corollary \ref{corollary existence time lower bound} to extend the existence time up to some $T := C(\max_M |\Rm_{\vp_0}|)^{-1}$. If $\vp_0$ is assumed only to be asymptotically conical in the sense of Definition \ref{definition asymptotically conical} and to have no further regularity, applying Corollary \ref{corollary existence for lower regularity} yields Theorem \ref{theorem main theorem} in its entirety. 

\appendix

\section{Linear theory for elliptic and parabolic equations on vector bundles}\label{section linear parabolic appdx}

Useful results from the general theory of linear elliptic and parabolic equations are collected in this appendix.

\subsection{Existence theory for linear equations on vector bundles}\label{subsection existence parabolic pde}

In this section, we collect the relevant results from the exposition in \cite{Baker} regarding the existence of $H^{1,1}$-regular solutions with homogeneous boundary conditions to linear parabolic equations on vector bundles over manifolds with boundary. 

Let $E_{\omega}$ be a vector bundle over $M \times [0,\omega]$, where $(M, \dd M)$ is a manifold with boundary. We consider the following boundary value problem in which $U$ and $V$ are sections of $E_{\omega}$ with a fixed connection $\na$. 
\begin{equation}\label{equation linear parabolic pde}
\left\{
\begin{array}{lr}
\dd_t U^a - LU^a := \dd_t U^a - \na_i(A^{aij}_b\na_j U^b) + B^{ai}_b(X) \na_i U^b + C^a_b(X) U^b = F^a(X),   \\
U|_{t= 0} \equiv 0 , 
\end{array}
\right.
\end{equation}
where $X \in M_\om $, and $A^{aij}_b$ satisfies the following \textit{Legendre-Hadamard condition}: there is $\la > 0$ such that
\[
A^{aij}_b\xi_i \xi_j\eta_a{\eta^*}^b > \la |\xi|^2|\eta|^2,
\]
where $\eta \in \Ga(E)$, $\eta^* \in \Ga(E^*)$, and $\xi \in \Om^1(M)$. This condition implies parabolicity.

Let $\overset{\bullet}{H^{1,1}}(E_\omega)$ be the closure in $H^{1,1}$ of $\overset{\bullet}{C^\infty}(\overline{M},E_\omega)$, the smooth (up to the boundary $\dd M$) sections of $E_\omega$ that vanish at the parabolic boundary of $M \times [0,\omega]$, i.e. $(\dd M \times [0,\omega]) \cup (M \times \{0\})$. The boundedness and coercivity of the bilinear form associated to the operator on the left-hand side of the equation implies (via the standard Lax-Milgram argument) the following result.

\begin{thm}\cite[Theorem 3.8]{Baker}\label{theorem existence of linear weak solution}
    If $F \in L^2(E_\omega)$, then the initial value problem \eqref{equation linear parabolic pde} admits a unique weak solution $U \in \overset{\bullet}{H^{1,1}}(E_\omega)$. 
\end{thm}

\subsection{Compatibility conditions and higher regularity}\label{subsection compatibility conditions}

In order to improve the regularity of the above weak solution up to the boundary $\dd M$, it is necessary to check the following compatibility conditions. The following statement is adapted from \cite[Theorem 6, p. 386]{EvansPDE} for the initial value problem \eqref{equation linear parabolic pde}. 

\begin{thm}\cite{EvansPDE}\label{theorem higher regularity linear par pde}
Assume that 
\[
\frac{d^k}{dt^k}F \in L^2([0,\om]; H^{2m-2k}(E_\omega))\qquad (k = 0, \ldots m).
\]
Suppose also that the following $m$-th-order \textit{compatibility conditions} hold:
\[
\left\{
\begin{array}{lr}
g_1 := F|_{t=0} \in H^1_0(E_\omega), g_2 := \frac{d}{dt}F\big|_{t=0} - Lg_1 \in H^1_0(E_\omega), \\
\ldots, g_m := \frac{d^{m-1}}{dt^{m-1}}F\big|_{t=0} - Lg_{m-1} \in H^1_0(E_\omega).
\end{array}
\right.
\]
Then
\[
\frac{d^k}{dt^k} U \in L^2([0,\om],H^{2m+2 -2k}(E_\om)) \qquad (k = 0, \ldots m + 1).
\]
\end{thm}

\subsection{Weighted Schauder estimates}\label{subsection weighted schauder}

We now prove Schauder estimates in weighted Sobolev spaces for solutions to \eqref{equation linear parabolic pde}, which we assume to be globally defined on a time-dependent vector bundle $E_{\om}$ over $M \times [0,\om]$, where $M$ is an $AC$ Riemannian manifold with metric $g$. We assume that the coefficients $A$, $B$, and $C$ have the property that
\begin{equation}\label{equation conditions on coeffs}
|g^{ij}G^a_b - A^{aij}_b| = o(\rho^{-2}), |\na (g^{ij}G^a_b - A^{aij}_b)|, |B|, \text{ and } |C| = O(\rho^{-2}),
\end{equation}
as the radial coordinate on $M$, $\rho$, approaches infinity. Here, $G_{ab}$ is the metric on the fibers of the vector bundle $E_\om$, $g$ is the metric on $M$, and all norms are measured with respect to $g$. If $R < \infty$ and $M_R := \rho^{-1}((0,R]) \subset M$, we denote the restriction vector bundle $E_{\om}|_{M_R \times [0,\om]}$ by $E_{R,\om}$.

We largely follow the statement and the proof of \cite[Propositions 7.3-7.4]{LMCFwithsings}, with modifications along the lines of \cite[\S 7]{SchlagSchauder} to prove estimates up to the boundary.

\begin{prop}\label{proposition weighted schauder estimates}
    Let $(M, g)$ be a complete manifold equipped with an AC metric $g$. Let $T>0$, $k \in \bN$ with $k \ge 2$. Let $R_0 > 0$ be a large constant depending on $g$. If $\rho: M \rightarrow \R$ is the radius function on $M$, we consider domains $M_R$, with $0 < R_0 < R_0 + 2 < R$ such that 
    \[
    M_R := \rho^{-1}((0,R]). 
    \]
   Let $E_{R,\om}$ be the restriction of the vector bundle $E_\om$ over $M_R \times [0,\om]$. Let $F \in H^{0,k-2}_{\la -2}(E_{R,\om}) \cap C^{0,0}(E_{R,\om})$ such that $F \equiv 0$ on the ``corner" of the parabolic boundary, $\dd M_R  \times \{t= 0\}$.  Assume that $U \in H^{1,2}(E_{R,\om}) \cap  L^2_{\la}(E_{R,\om})$ and $U$ solves \eqref{equation linear parabolic pde}, whose coefficients are assumed to satisfy \eqref{equation conditions on coeffs}. Then there exists a constant $c >0$ independent of $U, F,$ and $R$, such that
    \begin{equation}\label{equation schauder estimates vb with boundary}
        \|U\|_{H^{1,k}_{\la}(E_{R,\om})} \le c\Big(\|F\|_{H^{0,k-2}_{\la-2}(E_{R,\om})} + \|U\|_{L^2_{\la}(E_{R,\om})} \Big).
    \end{equation}
\end{prop}

\begin{proof}
    Given a solution to \eqref{equation linear parabolic pde} on $E_{R,\om}$, the standard interior Schauder estimates for parabolic equations in unweighted H\"older spaces on $M_{R_0} \subset \subset M_R$ yield
    \begin{equation}\label{equation base schauder estimates}
    \|U\|_{H^{1,k}(E_{R_0,\om})} \le c\Big(\|F\|_{H^{0,k-2}(E_{(R_0+2),\om})} + \|U\|_{L^2(E_{(R_0+2),\om})} \Big).
    \end{equation}
    Note that our choice of $R_0$ will depend on the AC structure of $(M,g)$ and nothing else, and therefore is a universal constant. Next we will assume that we have modified each end by the diffeomorphism $\Psi$ described in Definition \ref{definition asymptotically conical} and that all properties in Lemma \ref{lemma decay estimates on AC G2 forms} hold with $\Psi = \mathrm{Id}$. We define the following parametrization map on the asymptotically conical end $V \cap M_R \cong \Sigma \times (R_0, R)$.
    \[
    \delta^s : (0,s^{-2}T) \times \Sigma \times \big(\tfrac12, 1\big) \subset \big(0,\tfrac12\big) \times \Sigma \times \big(\tfrac12, 1\big) \rightarrow (0,T) \times \Sg \times (R_0,R)
    \]
    \[
    \delta^s(t,x,\rho) = (s^2t, x, s\rho)
    \]
    where $s$ is chosen in $\big(2R_0,R \big)$. For $U$ and $F$ (restricted to the appropriate domains) and valid $s$, we define
    \begin{align*}
        U^s : (0,s^2T) \times \Sigma \times \big( \tfrac12, 1 \big) \rightarrow E_\om, &\qquad U^s = s^{-\la}(\delta^s)^*U \\
        F^s : (0,s^2T) \times \Sigma \times \big( \tfrac12, 1 \big) \rightarrow E_\om, &\qquad F^s = s^{-\la + 2}(\delta^s)^*F. 
    \end{align*}
    Because $U \in L^2_{\la}$ and $F \in H^{0,k-2}_{\la-2}$, the change of variables formula implies that there is a constant $c > 0$ independent of $r$ such that 
    \[
    \|\U^s\|_{L^2}, \|F^s\|_{H^{0,k-2}} < c,
    \]
    for all $s \in \big(2R_0, R \big)$. We rescale the linear equation \eqref{equation linear parabolic pde} to obtain that $U^s$ and $F^s$ satisfy
    \begin{align*}
    \dd_t U^s &= s^{-\la + 2}(\delta^s)^*\dd_t U \\ 
    &= s^{-\la + 2}(\delta^s)^*\big(\na_i(A^{aij}_b\na_j U^b) - B^{ai}_b(X) \na U^b - C^a_b(X) U^b + F^a(X)\big) \\
    &= s^{-\la + 2}(\delta^s)^*\big(\De U^a + \na_i((A^{aij}_b - g^{ij}G^a_b)\na_j U^b) - B^{ai}_b(X) \na U^b - C^a_b(X) U^b + F^a(X)\big) \\
    &= s^{2}(\delta^s)^* (\Delta_g \big((\delta^s)^{-1}\big)^*U^s)^a + s^{2}(\delta^s)^*(A^{aij}_b - g^{ij}G^a_b)\na^2_{ij} (U^s)^b \\ 
    &\qquad +  s^{2}(\delta^s)^*(\na_i(A^{aij}_b - g^{ij}G^a_b))\na_j (U^s)^b - s^2 (\delta_i^s)^* B^{ai}_b(X) \na (U^s)^b \\
    &\qquad \qquad - s^2 (\delta_i^s)^*C^a_b(X) (U^s)^b + F^s \\
    &= \Delta_{g_C} (U^s)^a + (LU^s)^a + s^{2}(\delta^s)^*(A^{aij}_b - g^{ij}G^a_b)\na^2_{ij} (U^s)^b \\ &\qquad +  s^{2}(\delta^s)^*(\na_i(A^{aij}_b - g^{ij}G^a_b))\na_j (U^s)^b  - s^2 (\delta_i^s)^* B^{ai}_b(X) \na (U^s)^b \\
    &\qquad \qquad - s^2 (\delta_i^s)^*C^a_b(X) (U^s)^b + F^s
    \end{align*}
    where $g_{C}$ is the asymptotic cone metric on the end $V$, and $L$ is a differential operator given by
    \begin{align*}
    L U^s &= s^{2}(\delta^s)^* \Big(\Delta_{g}\big((\delta^s)^{-1}\big)^*U^s - \Delta_{g_{C}}\big((\delta^s)^{-1}\big)^*U^s \Big).
    \end{align*}
    The asymptotic conicality of $g$ contributes an $O(\rho^{-1})$ term for each pulled-back derivative, and the conditions \eqref{equation conditions on coeffs} contribute enough decay to counteract the largeness of the $s^2$ term.  The condition \eqref{equation ac condition order k} and Lemma \ref{lemma decay estimates on AC G2 forms} imply that the coefficients of $L$ are $o(1)$ in $s$. The rescaled equation may be written as
    \begin{align}\label{equation rescaled parabolic system}
    (\dd_t U^s)^a
    &= \Delta_{g_C} (U^s)^a + \hat A^{aij}_b \na^2_{ij} (U^s)^b + \hat B^{ai}_b(X) \na (U^s)^b  + \hat C^a_b(X) (U^s)^b + (F^s)^a,
    \end{align}
    where $\hat A = o(1)$ as $s\rightarrow \infty$ (independently of $R$), and $\hat B, \hat C$ are uniformly bounded independently of $R$, $s$, $U$ and $F$. Thus, standard interior parabolic Schauder estimates allow us to obtain the following bound for a constant $c > 0$ independent of $R$, $s$, $U$ and $F$.
    \begin{equation}\label{equation interior schauder estimates}
       \|U^s\|_{H^{1,k}(\Om_{s,\eps})} \le c\Big(\|F^s\|_{H^{0,k-2}(\Om_s)} + \|U^s\|_{L^2(\Om_s)} \Big),
    \end{equation}
    where the domain $\Om_{s,\eps}$ on the left hand side is $E_{R,\om}$ pulled back by $(\de^s)^*$ to $(0,s^{-2}T) \times \Sg \times \big(\tfrac12 + \eps, 1 - \eps\big)$ and the domain $\Om_s$ on the right hand side is $E_{R,\om}$ pulled back to $(0,s^{-2}T) \times \Sg \times \big(\tfrac12, 1\big)$. Distances on the base of both the domains $\Om_{s,\eps}$ and $\Om_{s}$ are measured with respect to the cone metric $g_C$. In particular, the constant $c >0$ depends on the choice of $\eps > 0$.

    By rescaling and choosing $\eps > 0$ so that a finite selection of non-overlapping strips $\Omega_{s_i,\eps}$ covers $\Sg \times (R_0, R(1-\eps))$, we obtain the interior estimate
    \[
    \|U\|_{H^{1,k}_{\la}(E_{R(1-\eps),\om})} \le c\Big(\|F\|_{H^{0,k-2}_{\la-2}(E_{R,\om})} + \|U\|_{L^2_{\la}(E_{R,\om})} \Big),
    \]
    where $c > 0$ does not depend on $R$, but does depend on the choice of $\eps > 0$.

    It remains to prove the estimate up to the boundary $\partial M_R = \rho^{-1}(R)$. Due to the asymptotically conical structure of $g$, the geometry of the exterior boundary $(\de^{r})^{-1}(\dd M_R)$ of the rescaled domain $(\de^{r})^{-1}(M_R)$ approaches the geometry of the link $\Sg$ of the cone $C$, for large values of $R$. 

    Consider a neighborhood $(1-\eps,1] \times \Sg$ equipped with the cone metric $g_C$. Then, given $\de \in (0,1)$, the proof of \cite[Lemma 4]{SchlagSchauder} can be modified to show that the space-time neighborhood $[0,R^{-2}T] \times \Sg \times (1-\eps,1]$ can be covered by the union of finitely many neighborhoods $P_i$, contained in larger neighborhoods $\hat P_i$ and associated to diffeomorphisms $\Psi_i$ such that
    \begin{enumerate}
        \item $\mathrm{diam}(\hat P_i) < \de$ in the anisotropic parabolic metric, for all $i$.
        \item Let $B_s^{+} \equiv \{x = (x_1, \ldots, x_n) \in \R^n| \max |x^i| < s, x^n  \ge 0$\}, and let $Q_s = B_s^{+} \times [0,R^{-2}T]$. For all $i$, $\Psi_i(P_i) \subset Q_1$ and $\Psi_i(\hat P_i) = Q_3$. 
        \item for all $i$, bounds on the derivatives of $\Psi_i$, scaled by appropriate powers of $\de$ hold (see (iv) in the statement of \cite[Lemma 4]{SchlagSchauder}).
        \item At every point $x$ in the domain, the number of $\hat P_i$'s that overlap at $x$ is bounded above independently of $x$ and $\de$.
    \end{enumerate}
    Note that the shrinking of the time-interval $[0,R^{-2}T]$ does not affect the geometry of the decomposition. In fact, the $P_i$'s can be constructed for each $\delta$ with respect to the time interval $[0,1]$, and then restricted to $[0,R^{-2}T]$ for a given $R>0$.

    Recall that $U^R = R^{-\la}(\de^R)^*U$. For all $i$, the local section $U^R$ restricted to $\hat P_i$ satisfies \eqref{equation rescaled parabolic system} on $\hat P_i$. Notice that for large $R$, the coefficients of the second derivatives have uniformly bounded modulus of continuity. Thus, for any rescaling $R$, the following version of inequality \cite[(65)]{SchlagSchauder} holds with constants independent of $R$. 
    \begin{align}
        \int_{P_i} |\na^2 U^R|^2 &\le C_0 \sup_{(x,t) \in \hat P_i} |(g_C + A)(x,t)-(g_C + A)(x_i,t_i)|^2 \int_{\hat P_i} |\na^2 U^R|^2 \label{equation boundary schauder estimates}\\
        &\qquad + C_\de \int_{\hat P_i} (|F^R|^2 + |\na U^R|^2 + |U^R|^2),\nonumber
    \end{align}
    where $(x_i,t_i) \in P_i$ is arbitrary but fixed. We choose $\de >0$ small enough that \[
    C_0 \sup_i\sup_{(x,t) \in \hat P_i} |(g_C + A)(x,t)-(g_C + A)(x_i,t_i)|^2 < \frac12,\]
    and note that this $\de$ may be chosen once and for all due to the uniform bounds on the coefficients of the second derivatives in \eqref{equation rescaled parabolic system}. Then combining the estimates \eqref{equation base schauder estimates}, \eqref{equation interior schauder estimates}, and \eqref{equation boundary schauder estimates}, and applying standard absorption and interpolation results, we obtain the full Schauder estimates \eqref{equation schauder estimates vb with boundary}, with coefficient $c >0$ independent of the rescaling parameter $R > R_0$.
    \end{proof}

We will also need Schauder estimates for boundary value problems involving elliptic operators of high order. The following proposition establishes a priori estimates for the weighted Hodge bilaplacian, an operator of order 4. 

\begin{prop}\label{proposition lopatinski shapiro condition}
Let $L$ be the elliptic operator of order $4$ defined on $k$-forms $\Omega^k(M)$ on a complete AC manifold $(M,g)$ by
\begin{align*}
    L\sg &= \De^*\De \sg \\
        &= \rho^{2\la + 4 + 7}\De (\rho^{-2\la -7} \De \sg) \\
        &= \rho^{4}\De \De \sg + C_3 \rho^{3} \na \rho * \na \De \sg + \rho^{2} (\De \rho )(\De \sg).
\end{align*}
Notice that the $j$th order coefficients $B_j$ have the property that
\begin{equation}\label{equation growth of coefficients}
|B_j| = O(1),
\end{equation}
for $j = 0, \ldots, 4$ (compare to the equivalent condition to \eqref{equation conditions on coeffs} in the parabolic case). Let $L^*$ denote the adjoint of $L$, defined by
\begin{align*}
    L^*\sg &= \De\De^* \sg \\
        &= \De \rho^{2\la + 4 + 7}\De (\rho^{-2\la -7} \sg) \\
        &= \rho^{4}\De \De \sg + \sum_{i=0}^3 \sum_{j=1}^{4-i} C_{ij} \rho^i \na^j \rho * \na^{4-i-j} \rho * \na^i \sg
\end{align*}
As before, the $j$th order coefficients $B_j$ have $O(1)$ decay. Let $M_R = \rho^{-1}((0,R]) \subset M$ and suppose that $\sg \in L^2_{\la}\Om^k(M_R) \cap H^4_{loc} \Om^k(M_R)$ solves either of the following adjoint boundary value problems: 
\begin{equation}\label{equation fourth order bvp with LS}
    \left\{
\begin{array}{lr}
L\sg = \eta   & \;\; \text{ on } M_R\\ 
\bt \sg = 0 \;\text{ and }\; \bt \rho \de \sg = 0 & \;\; \text{ on } \dd M_R,
\end{array}
\right.
\end{equation}
or its adjoint,
\begin{equation}\label{equation fourth order bvp with LS adjoint}
    \left\{
\begin{array}{lr}
L^*\sg = \eta   & \;\; \text{ on } M_R\\ 
\bt \sg = 0 \;\text{ and }\; \bt \rho^{2\la + 7 + 1} \de (\rho^{-2\la -7} \sg) = 0 & \;\; \text{ on } \dd M_R,
\end{array}
\right.
\end{equation}
where $\eta \in H_{\la}^{\ell-4}(M_R)$ for $\ell \ge 4$.
Then there exists a constant $c >0$ independent of $\sg, \eta,$ and $R$, such that
    \begin{equation}\label{equation schauder estimates forms with lopatinski shapiro}
        \|\sg\|_{H^{\ell}_{\la}(M_R)} \le c\Big(\|\eta\|_{H^{\ell-4}_{\la}(M_R)} + \|\sg\|_{L^2_{\la}(M_R)} \Big).
    \end{equation}
\end{prop}

\begin{rmk}
    Note that the boundary conditions $\bt \rho \de \sg = 0$ and $\bt \rho^{2\la + 7 + 1} \de (\rho^{-2\la -7} \sg) = 0$ are equivalent to $\bt \de \sg = 0$ and $\bt \de (\rho^{-2\la -7} \sg) = 0$.
\end{rmk}

\begin{proof}
    The proposition will be first proved for the boundary value problem \eqref{equation fourth order bvp with LS}. Since the proof for \eqref{equation fourth order bvp with LS adjoint} is almost identical, we will simply make a note at the end of the necessary modifications needed to prove the estimates for the adjoint problem. 
    
    The proof uses the same method of scaling as the proof of Proposition \ref{proposition weighted schauder estimates}. The only substantial difference is that the estimates at the boundary are obtained by using elliptic estimates that are not equivalent to the parabolic estimates used before. 

    The scale-invariant interior estimates are proved first. As before, define a parametrization map on the asymptotically conical end $V \cap M_R \cong \Sigma \times (R_0, R)$.
    \[
    \delta^s : \Sigma \times \big(\tfrac12, 1\big) \rightarrow \Sg \times (R_0,R)
    \]
    \[
    \delta^s(x,\rho) = (x, s\rho)
    \]
    where $s$ is chosen in $\big(2R_0,R \big)$. For $\sg$ and $\eta$ (restricted to the appropriate domains) and valid $s$, we define
    \begin{align*}
        \sg^s : \Sigma \times \big(\tfrac12, 1 \big) \rightarrow \La^k(M_R), &\qquad \sg^s = s^{-\la}(\delta^s)^*\sg \\
        \eta^s :  \Sigma \times \big(\tfrac12, 1\big) \rightarrow \La^k(M_R), &\qquad \eta^s = s^{-\la}(\delta^s)^*\eta. 
    \end{align*}
    Because $\sg \in L^2_{\la}$ and $\eta \in H^{\ell-4}_{\la}$, the change of variables formula implies that there is a constant $c > 0$ independent of $R$ such that 
    \[
    \|\sg^s\|_{L^2}, \|\eta^s\|_{H^{\ell-4}} < c,
    \]
    for all $s \in \big(2R_0, R \big)$. We rescale the linear equation \eqref{equation fourth order bvp with LS} to obtain that $\sg^s$ and $\eta^s$ satisfy
    \begin{align*}
    \eta^s &= s^{-\la}(\delta^s)^*\eta \\ 
    &= s^{-\la}(\delta^s)^*(L\sg) \\
    &= s^{-\la}(\delta^s)^*\big(\rho^{4}\De \De \sg + \rho^{3} \na \rho * \na \De \sg + \rho^{2} (\De \rho )(\De \sg)\big) \\
    &= s^{-\la}(\delta^s)^*\big(\rho^{4}\De_g \De_g \sg + \rho^{3} \na \rho * \na \De_g \sg + \rho^4 \cR * \De_g \sg + \rho^4 \De_g (\cR * \sg) \\
    &\qquad + \rho^{2} (\De \rho )(\De \sg) + \rho^{3} \na \rho * \na (\cR * \sg) + \rho^4 \cR * \cR * \sg  \big),
    \end{align*}
    where the last line follows by applying the Weitzenb\"ock identity. Observe that the lower order coefficients satisfy the boundedness condition \eqref{equation growth of coefficients} and we simply collect these terms in the bounded tensors $B_0,\ldots,B_3$. Continuing the computation,
    \begin{align*}
    \eta^s &= (\delta^s)^* (\rho^4 \De_g \Delta_g \big((\delta^s)^{-1}\big)^*\sg^s) + \sum_{j=0}^3 \big((\delta^s)^*B_j\big) * (\na^j_{g_C} \sg^s) \\ 
    &= \bigg(\frac{(\de^s)^* \rho}{s}\bigg)^4\De_{g_C} \Delta_{g_C} \sg^s + \sR \sg^s + \sum_{j=0}^3 \big((\delta^s)^*B_j\big) * (\na^j_{g_C} \sg^s), 
    \end{align*}
    where $g_{C}$ is the asymptotic cone metric on the end $V$, and $\sR$ is a differential operator given by
    \begin{align*}
    \sR \sg^s &= (\delta^s)^* \Big(\rho^{4}\De_{g}\Delta_{g}\big((\delta^s)^{-1}\big)^*\sg^s - \rho^{4}\De_{g_{C}}\Delta_{g_{C}}\big((\delta^s)^{-1}\big)^*\sg^s \Big).
    \end{align*}
    The asymptotic conicality of $g$ contributes an $O(\rho^{-1})$ term for each pulled-back derivative.  The condition \eqref{equation ac condition order k} and Lemma \ref{lemma decay estimates on AC G2 forms} imply that the coefficients of $\sR$ are $o(1)$ in $s$. Putting everything together
    \begin{align}\label{equation rescaled elliptic system}
    \eta^s
    &= \bigg(\frac{(\de^s)^* \rho}{s}\bigg)^4 \De_{g_C} \Delta_{g_C} \sg^s + \sR \sg^s + \sum_{j=0}^3 \big((\delta^s)^*B_j\big) * (\na^j \sg^s),
    \end{align}
    where the coefficients of $\sR$ are $o(1)$ as $s\rightarrow \infty$ (independently of $R$), and the lower order coefficients are uniformly bounded independently of $R$, $s$, $\sg$ and $\eta$. Furthermore, the term 
    \[
    \bigg(\frac{(\de^s)^* \rho}{s}\bigg)^4
    \]
    is a linear function interpolating between $\tfrac{1}{16}$ and $1$ over the interval $\big(\tfrac12, 1 \big)$. In particular, the equation in uniformly elliptic. Thus, standard interior elliptic Schauder estimates (see \cite[Corollary 10.3.9]{NicolaescuManifoldLectures}) allow us to obtain the following bound for a constant $c > 0$ independent of $R$, $s$, $\sg$ and $\eta$.
    \begin{equation}
      \|\sg^s\|_{H^{\ell}\Om^k(\sD_{s,\eps})} \le c\bigg(\bigg\|\eta^s + \sR \sg^s + \sum_{j=0}^3 \big((\delta^s)^*B_j\big) * (\na^j \sg^s)\bigg\|_{H^{\ell-4}\Om^k(\sD_s)} + \|\sg^s\|_{L^2 \Om^k(\sD_s)} \bigg)
    \end{equation}
    The smallness of the coefficients of $\sR$ and standard interpolation results (see \cite[Theorem 7.28]{GT}) give the bound 
    \[
    \bigg\| \sR \sg^s  + \sum_{j=0}^3 \big((\delta^s)^*B_j\big) * (\na^j \sg^s)\bigg\|_{H^{\ell-4}\Om^k(\sD_s)}  \le \eps_1 \|\sg^s\|_{H^{\ell}\Om^k(\sD_{s})} + C(\eps_1) \|\sg^s\|_{L^2 \Om^k(\sD_s)}.
    \]
    Putting these inequalities together yields the following interior estimate.
    \begin{equation}\label{equation interior elliptic schauder estimates}
      \|\sg^s\|_{H^{\ell}\Om^k(\sD_{s,\eps})} \le c\Big(\|\eta^s\|_{H^{\ell-4}\Om^k(\sD_s)} + \|\sg^s\|_{L^2_{\la}\Om^k(\sD_s)} \Big) + \eps_1 \|\sg^s\|_{H^{\ell} \Om^k(\sD_{s})},
    \end{equation}
    where the space $\Om^k(\sD_{s,\eps})$ on the left hand side is $\Om^k(M_R)$ pulled back by $(\de^s)^*$ to $\Sg \times \big(\tfrac12 + \eps, 1 - \eps\big]$ and the space $\Om^k(\sD_{s})$on the right hand side is $\Om^k(M_R)$ pulled back by $(\de^s)^*$ to $\Sg \times \big(\tfrac12, 1\big)$. In particular, the constant $c >0$ depends on the choice of $\eps > 0$.

    By rescaling and choosing $\eps > 0$ so that a finite selection of non-overlapping strips $\de^{s_i}\big( \sD_{s_i,\eps} \big)$ covers $\Sg \times (R_0, R(1-\eps))$, we obtain the interior estimate
    \begin{equation}\label{equation interior only estimates}
        \|\sg\|_{H^{\ell}_{\la}\Om^k(M_{R(1-\eps)})} \le c\Big(\|\eta\|_{H^{\ell-4}_{\la}\Om^k(M_R)} + \|\sg\|_{L^2_{\la}\Om^k(M_R)} \Big) + \eps_1 \|\sg\|_{H^{\ell}_{\la}\Om^k(M_{R})},
    \end{equation}
    where $c > 0$ does not depend on $R$, but does depend on the choice of $\eps > 0$.

    It remains to prove the estimate up to the boundary. Assume that $\eps \ll \tfrac14$. Consider a cutoff function $\zeta \in C^\infty(\Sg \times [\tfrac12,1])$ such that $\zeta \equiv 1$ on $\Sg \times (\tfrac34 , 1)$ and $\zeta \equiv 0$ on $\Sg \times (\tfrac12, \frac58)$, that $\na^j \zeta$ are bounded, possibly depending on $j$, and (by an abuse of notation) $\zeta(x) = \zeta(|x|)$. Let $\zeta_R(x) = \zeta (R^{-1}|x|)$, where $x \in \Sg \times [\tfrac12, 1]$. 

    Consider the expression
    \begin{align*}
        L(\zeta_R \sg) &= \zeta_R L\sg  + \sum_{j=1}^3 \sum_{i= 1}^j B_j * \na^{j-i}\sg * \na^i \zeta_R\\
        &= \zeta_R \eta + \sum_{j=1}^3 \sum_{i= 1}^j  (\rho^{j-i} R^{i-j} B_j) * \na^{j-i}\sg * \na^{i} \zeta \\
        &:= \hat \eta.
    \end{align*}
    We repeat the rescaling process used for the interior estimates, except we set $s := R$ (e.g. $s^{-\la}(\de^s)^*= R^{-\la}(\de^R)^*$, etc.). Note that $(\de^R)^* \zeta_R = \zeta$. We obtain the equation
    \begin{align*}
    \bigg(\frac{(\de^R)^* \rho}{R}\bigg)^4\De_{g_C} \Delta_{g_C} \zeta \sg^R + \sR\zeta \sg^R 
    & = \zeta \eta^R + \sum_{j=1}^3 \sum_{i= 1}^j  \big((\delta^R)^*B_j\big) * \na^{j-i}\sg^R * \na^{i} \zeta \\
    & = \hat \eta^R.
    \end{align*}
    Since $\na^i \zeta \equiv 0$ on $[\tfrac34 , 1]$, the appropriate norms of $\hat \eta^R$ can be bounded in terms of the norms of $\eta^R$ and the interior estimates on $\sg^R$. 

    Now note that after rescaling, the boundary conditions $\bt \sg = 0$ and $\bt \rho \de \sg = 0$ imply that $\bt \zeta \sg^R$ and $\bt \de_{R^{-2}(\de^R)^*g} \zeta \sg^R$ vanish on the boundary $\dd (\Sg \times [\tfrac12,1]) := (\Sg \times \{\frac12\} \cup (\Sg \times \{1\})$. Observe that $\zeta \sg^R$ solves the following elliptic boundary value problem 
    \begin{equation}
        \left\{
            \begin{array}{lr}
                \big(\tfrac{(\de^R)^* \rho}{R}\big)^4\De_{g_C} \Delta_{g_C} \zeta \sg^R = \hat \eta^R - \sR \zeta \sg^R  & \;\; \text{ on } \Sg \times [\tfrac12,1]\\ 
                \bt \zeta \sg^R = 0 \;\text{ and }\; \bt \de_{g_C} \sg^R = \bt(\de_{g_C} - \de_{R^{-2}(\de^R)^*g})\sg^R & \;\; \text{ on } \dd (\Sg \times [\tfrac12,1]),
            \end{array}
        \right.
    \end{equation}
    where the pair $(\bt, \bt \de_{g_C})$ satisfies the Lopatinski-Shapiro condition defined in \cite[Definition 1.6.1]{HodgeDecompSchwarz}. Thus, we may apply the elliptic apriori estimate in \cite[(ii), Theorem 1.6.2]{HodgeDecompSchwarz} to find that
    \begin{multline}\label{equation initial boundary estimates}
        \|\sg^R\|_{H^{\ell}\Om^k(\Sg \times [\frac34,1])} \le c\Big(\|\hat \eta^R - \sR \zeta \sg^R\|_{H^{\ell-4}\Om^k(\Sg \times [\frac12,1])} \\ + \|\bt(\de_{g_C} - \de_{R^{-2}(\de^R)^*g})\sg^R\|_{H^{\ell-\frac32}\Om^k(\dd (\Sg \times [\frac12,1]))} + \|\sg^R\|_{L^2\Om^k(\Sg \times [\frac12,1])} \Big),
    \end{multline}
    where $c>0$ does not depend on $R$. The trace operator is a continuous operator by \cite[Theorem 1.3.7(c)]{HodgeDecompSchwarz}, so there is a constant $c>0$ such that
    \[
    \|\bt(\de_{g_C} - \de_{R^{-2}(\de^s)^*g})\sg^s\|_{H^{\ell-\frac32}\Om^k(\dd (\Sg \times [\frac12,1]))} \le c \|(\de_{g_C} - \de_{R^{-\la}(\de^s)^*g})\sg^R\|_{H^{\ell-1} \Om^k(\Sg \times [\frac12,1])}
    \]
    The smallness of the coefficients of the operators $\sR$ and $\de_{g_C} - \de_{R^{-2}(\de^R)^*g}$ allow us to absorb their corresponding terms into the left-hand side of the inequality \eqref{equation initial boundary estimates}. Undoing the scaling and using the interior estimates to bound $\hat \eta$, we find that
    \begin{align*}
        \|\sg\|_{H^{\ell}_{\la}\Om^k(M_{R}\setminus M_{3R/4})} & \le c\Big(\|\eta\|_{H^{\ell-4}_{\la}\Om^k(M_R)} + \|\sg\|_{L^2_{\la}\Om^k(M_R)} \Big).
    \end{align*}
    Adding this estimate up to the boundary to the interior estimates \eqref{equation interior only estimates} finishes the proof of the proposition for the problem \eqref{equation fourth order bvp with LS}. To prove the estimates for solutions to \eqref{equation fourth order bvp with LS}, simply note that a) the principal symbol of $\De^*\De$ and that of $\De \De^*$ are the same, and b) the principal symbol of the boundary operator $\big(\bt \sg, \bt \rho^{2\la + 7 + 1} \de (\rho^{-2\la -7} \sg) \big)$ is the same as the principal symbol of $\big(\bt \sg, \bt \rho \de \sg \big)$. This is enough to show that the estimates \eqref{equation schauder estimates forms with lopatinski shapiro} can be obtained for solutions to \eqref{equation fourth order bvp with LS adjoint}.
\end{proof}

\begin{cor}\label{corollary lopatinski shapiro higher order}
Take the hypotheses of Proposition \ref{proposition lopatinski shapiro condition} and consider a form $\sg \in L^2_{\la}\Om^k(M_R) \cap H^{4m}_{loc} \Om^k(M_R)$ which solves either of the following boundary value problems.
\begin{equation}\label{equation higher order arb elliptic bvp with LS}
    \left\{
\begin{array}{lr}
L^m\sg = \eta   & \;\; \text{ on } N\\ 
\bt \sg = 0 \;\text{ and }\; \bt \de \sg = 0 & \;\; \text{ on } \dd N,
\end{array}
\right.
\end{equation}
or,
\begin{equation}\label{equation higher order arb elliptic bvp with LS adjoint}
    \left\{
\begin{array}{lr}
(L^*)^m\sg = \eta   & \;\; \text{ on } N\\ 
\bt \sg = 0 \;\text{ and }\; \bt \rho^{2\la + 7 + 1} \de (\rho^{-2\la -7} \sg) = 0 & \;\; \text{ on } \dd N,
\end{array}
\right.
\end{equation}
where $\eta \in H_{\la}^{k-4m}(M_R)$. Then there exists a constant $c >0$ independent of $\sg, \eta,$ and $R$, such that
    \begin{equation}\label{equation high order elliptic schauder estimates on forms}
        \|\sg\|_{H^{k}_{\la}(M_R)} \le c\Big(\|\eta\|_{H^{k-4m}_{\la}(M_R)} + \|\sg\|_{L^2_{\la}(M_R)} \Big).
    \end{equation}
\end{cor}

\begin{proof}
    The proof proceeds along the same lines as that of Proposition \ref{proposition lopatinski shapiro condition}, but this time the principal symbol of $L^k$ is
    \[
    \mathrm{sym}_P(L^k) = \rho^{4m}g^{i_1j_1}g^{i_2j_2}\cdots g^{i_{2m}j_{2m}}\xi_{i_1}\xi_{j_1}\xi_{i_2}\xi{j_2}\cdots \xi_{i_{2m}}\xi{j_{2m}},
    \]
    which is elliptic of order $m$. The lower order terms scale appropriately and the boundary conditions are still elliptic in the sense of Lopatinski-Shapiro. Observing that estimates in \cite[Corollary 10.3.9]{NicolaescuManifoldLectures} and \cite[Theorem 1.6.2]{HodgeDecompSchwarz} apply to elliptic operators of arbitrary order and following the proof of Proposition \ref{proposition lopatinski shapiro condition} gives the corollary.
\end{proof}

\subsection{Extension of Sobolev functions}

In this section, we define an operator $E_R$ extending forms in $W^{k,p}_{\la}(M_R)$ to forms in $W^{k,p}_{\la}(M)$, for $1 < p < \infty$ and any $R > R_0$, such that the operator norms $\|E_R\|$ are uniformly bounded independent of $R$.  

\begin{constr}\label{construction extension to collar neighborhood}
Let $\phi$ be a form in $W^{k,p}_{\la}\Omega^k(M_R)$. Define the map
\begin{align*}
    \de^R : \Sg \times \big(\tfrac12, 1\big] &\rightarrow \Sg \times \big(\tfrac{R}{2}, R \big]\\
    \de^R(x, t) &= (x,Rt).
\end{align*}
Consider $\phi$ restricted to $\rho^{-1}\big(\frac{R}{2}, R \big)$, and take 
\[
\phi^R := R^{-\la} (\de^R)^* \phi.
\]
Multiply $\phi^R$ by a cutoff function in a collar neighborhood of $t^{-1}(\{1\})$ and define $\phi^R$ on the exterior of $t^{-1}(\{1\})$ in this neighborhood by a higher order reflection of the sort employed in \cite[Corollary 1.3.7]{HamiltonNashMoser}. This defines an extension of $\phi^R$ as an $W^{k,p}$-regular form over the boundary $t^{-1}(\{1\})$. This extension is bounded as an operator on $W^{k,p}$-regular forms over the appropriate domains. We restrict our view to a band of width $\tfrac{1}{4}$ on either side of the boundary $t^{-1}(\{1\})$. Undoing the scaling transformations, we see that the extension on this $1/4$-sized region defines a bounded extension of $\phi$ in $W^{k,p}_{\la}$ to $M_{5R/4}$. Denote this extension by $\hat \phi$. Let $E_R$ be defined by
\[
E_R(\phi) = g_R \hat \phi,
\]
where $g_R: M \rightarrow \R$ is a family of functions parametrized by $R > R_0$ and depending only on the real variable $\rho$ which are uniformly equal to 1 on $M_R$, and equal to zero on $M\setminus M_{\frac{9R}{8}}$. These $g_R$'s are defined so that
\begin{align*}
    g_R &= O(1)\\
    g_R'&=O(R^{-1})\\
    &\;\; \vdots \\
    g_R^{(k)} &= O(R^{-k}).
\end{align*}
with the constants in the implied upper bounds not depending on $R$. 
\end{constr}

\begin{lem}\label{lemma extension of functions}
The operator $E_R$ defined in Construction \ref{construction extension to collar neighborhood} is bounded as a map from $W^{k,p}_\la \Om^k(M_R)$ to $W^{k,p}_{\la} \Om^k (M)$, with an upper bound independent of $R$.
\end{lem}

\begin{proof}
    If $\phi \in W^{k,p}_\la \Om^k(M_R)$, then $\|\hat \phi g_R\|_{L^p_\la(M)} \le C \|\hat \phi\|_{W^{k,p}_\la (M_R)}$. Now take the first derivatives: since $g_R$ only changes in the $\rho$-direction. 
    \[
    \na (g_R \hat \phi) = g_R \na \hat \phi + g_R' \na \rho \otimes \hat \phi. 
    \]
    The first summand is bounded by $\|\hat \phi \|_{W^{1,p}_{\la}(M_R)}$, and the second summand is bounded by $|g'|\| \phi \|_{L^p_{\la}(M_R)}$. The decay conditions on $g_R$ and, for higher derivatives, the fact that $\nabla^k \rho = O(\rho^{-k+1})$ gives the correct decay to the derivatives. Thus, we find that 
    \[
    \|E_R(\phi)\|_{W^{k,p}_{\la}(M)} \le C\|\phi\|_{W^{k,p}_{\la}(M_R)},
    \]
    where $C>0$ is independent of $R$.
\end{proof}

\subsection{The Ecker-Huisken maximum principle}

In order to run the barrier arguments in \S \ref{section long time behavior}, we appeal to the Ecker-Huisken maximum principle for heat subsolutions with respect to a moving background metric. For the reader's convenience, we reproduce the full statement of the result here, as written in \cite[Theorem 5]{DaiMa}. We adopt the convention that the superscript `$t$' indicates the $t$-dependence of a geometrical quantity associated to the moving metric $g(t)$.

\begin{thm}\label{theorem ecker huisken max}
    Suppose that the complete non-compact manifold $M^n$ with Riemannian metric $g(t)$ satisfies the uniform volume growth condition
    \begin{equation}\label{equation uniform volume growth}
        \vol^t(B^t_R(p)) \le \exp \big( k (1 + r^2) \big)
    \end{equation}
    for some point $p \in M$ and a uniform constant $k > 0$ for all $t \in [0,T]$. Let $w$ be a function on $M \times [0,T]$ which is smooth on $M \times (0,T]$ and continuous on $M \times [0,T]$. Asssume that $w$ and $g(t)$ satisfy
    \begin{enumerate}
        \item the differential inequality 
            \begin{equation}\label{equation ecker huisk diff ineq}
                \dd_t w - \De^t w \le \mathbf{a} \cdot \na w + b w
            \end{equation}
            where the vector field $\mathbf{a}$ and the function $b$ are uniformly bounded
            \begin{equation}\label{equation ecker huisk coeff bdd}
               \sup_{M \times [0,T]}|\mathbf{a}| \le \al_1, \qquad \sup_{M \times [0,T]}|b| \le \al_2
            \end{equation}
            for some constants $\al_1, \al_2 < \infty$;
        \item the initial data
            \begin{equation}\label{equation initial data negative}
                w(p,0) \le 0
            \end{equation}
            for all $p \in M$;
        \item the growth condition
            \begin{equation}\label{equation growth condition}
                \int_0^T \bigg( \int_M \exp \big[ -\al_3 d^t(p,y)^2\big]|\na w|^2(y)d\mu_t \bigg)dt< \infty
            \end{equation}
            for some constant $\al_3 >0$;
        \item the bounded variation conditions for metrics
            \begin{equation}\label{equation bounded metric variation condition}
                \sup_{M \times [0,T]} |\dd_t g(t)| \le \al_4
            \end{equation}
            for some constant $\al_4 < \infty$.
    \end{enumerate}
    Then, we have
    \[
    w \le 0
    \]
    on $M \times [0,T]$.
\end{thm}

In the case of a Laplacian flow $\vp(t)$ on a manifold $M$ such that the associated metrics $g(t)$ have uniformly bounded curvature, two of the hypotheses of Theorem \ref{theorem ecker huisken max} can be confirmed immediately.

\begin{lem}\label{lemma confirm volume and variation bound}
If the flow $\vp(t)$ satisfies \eqref{equation laplacian flow} and its associated metrics $g(t)$ have the uniform curvature bound 
\[
\sup_{M \times [0,T]} |\Rm| \le K,
\]
then $g(t)$ satisfies \eqref{equation uniform volume growth} and \eqref{equation bounded metric variation condition}. 
\end{lem}

\begin{proof}
    The evolution equation \eqref{equation evolution metrics} combined with the curvature bound immediately implies the bounded variation condition \eqref{equation bounded metric variation condition} for $g(t)$. The curvature bound and the Bishop-Gromov volume comparison theorem ensure that a constant $k$ may be selected so that \eqref{equation uniform volume growth} holds for all $t \in [0,T]$.
\end{proof}

\section{Estimates on the potentials of exact forms} \label{section estimates potential}

Since we are solving the Laplacian flow equation on the level of the potential, the ability to find a reasonable choice of potential function for a given exact form (and likewise for families) will be crucial. We will use the tools of the Hodge-Morrey decomposition on domains to this end. Namely, given a compact manifold with boundary $(N, \dd N)$, recall from \eqref{equation hodge morrey further decomp} the \emph{Dirichlet harmonic fields}:
\[
\sH^k_D(N) := \sH^k(N) \cap H^1\Omega^k_D(N) = \{\nu \in H^1 \Omega^k(N)\; |\; d\nu = \de \nu = 0, \; \bt \nu = 0\}.
\]
By \cite[Theorem 2.2.2]{HodgeDecompSchwarz}, the space $\sH^k_D(N)$ is finite-dimensional. Therefore, there is an $L^2$-orthogonal decomposition 
\[
L^2\Omega^k(N) = \sH^k_D(N) \oplus \sH^k_D(N)^\perp. 
\]
If $\eta \in \sH^k_D(N)^\perp$, by \cite[Theorems 2.2.4-6]{HodgeDecompSchwarz}, there is a strong solution $\phi_D$ called the \emph{Dirichlet potential} to the boundary value problem 
\begin{align}\label{equation dirichlet potential bvp}
    \begin{aligned}
    \Delta \phi_D &= \eta \\
    \bt \phi_D = 0 \;\;\;&\text{ and }\;\;\; \bt \de \phi_D = 0
    \end{aligned}
    &&
    \begin{aligned}
        &\text{ on }N\\
     &\text{ on }\dd N. 
    \end{aligned}
\end{align}
The latter two boundary conditions are called the \emph{homogeneous boundary conditions}. 

\begin{rmk}
    In this Appendix, $\De$ will always refer to the Hodge Laplacian. In the case of the domains $M_R$, this operator will be taken with respect to the restriction of a fixed background metric on $M$. 
\end{rmk}

Suppose that $\eta \in L^2\Omega^k(N)$, $\phi_D$ is its Dirichlet potential, and let $d\al_\eta$ be its component in the exact part $\sE^k(N) \subset L^2\Omega^k(N)$ of the Hodge-Morrey decomposition. Then, one may choose
\[
\al_\eta = \de \phi_D,
\]
where the potential $\al_\eta$ satisfies the Dirichlet boundary condition $\bt \al_\eta$ by the homogeneous boundary condition for $\phi_D$. Clearly, any desired estimates on $\al_\eta$ will follow from estimates on the Dirichlet potential $\phi_D$.

As discussed above, there is a Green's function for the boundary value problem \eqref{equation dirichlet potential bvp} on the $L^2$-orthogonal complement of the Dirichlet harmonic fields, $\sH_D^k(N)^\perp$.  However, as we need estimates on not only one potential $\phi_D$, but a whole sequence of potentials associated to the domains $M_{R_i}$, we need to make sure that the norm of the Green's operator does not become unbounded as $R_i \rightarrow \infty$. 

Unfortunately, this introduces two complications. The first is that the uniformity of the estimates is equivalent to boundedness as a map between the weighted Sobolev spaces, $H^s_\la$, in which many important operators are not self-adjoint. The second is that regardless of which norms we take, the Green's operator's norm may explode as $i \rightarrow \infty$: This roughly corresponds to the so-called ``problem of small eigenvalues" in a way which we will make precise soon. 

Recall our notation for the domains $M_R := \rho^{-1}([0, R])$ where $\rho: M \rightarrow \R$ is the radius function associated to the asymptotically conical end of $M$ with respect to the asymptotically conical metric $g_{\vp_0}$.

Now, consider the Hilbert space $H^s_\la \Omega^k(M_R) \cap \sH^k_D(M_R)^\perp$. By the equivalence of the weighted and unweighted spaces on domains (c.f. Remark \ref{remark weighted unweighted equivalence}) this is merely the space of forms in $\sH^k_D(M_R)^\perp$ that have $s$ weak $L^2$-derivatives, only measured with respect to the weighted $H^s_\la$-inner product,
\begin{align*}
    \langle \phi, \psi \rangle_{H^s_\la} &= \sum_{j = 0}^s \int_{M_R} \langle \rho^{-\la + j}\nabla^j \phi,\rho^{-\la + j}\nabla^j \psi \rangle \rho^{-7}dV_g.
\end{align*}

The following lemma about the $L^2_{\la}$-orthogonal projection of $H^s_\la \Omega^k(M_R)$ to $\sH^k_D(M_R)^\perp$ will be useful.

\begin{lem}\label{lemma projection to harmonic dirichlet fields}
    The $L^2_{\la}$-orthogonal projection of $H^s_\la \Omega^k(M_R)$ to $H^s_\la \Omega^k(M_R) \cap \sH^k_D(M_R)^\perp$ is bounded with respect to the $H^s_{\la}$-norm, with constant independent of $R$. The kernel of this map is spanned by $\{\rho^{2\la + 7} \sg_i\}_i$, where the $\{\sg_i\}_i$ comprise a basis of $\sH^k_D(M_R)$.
\end{lem}

\begin{proof}
    Let $\ve = -2\la - 7$. If the $\{\sg_i\}_{i=1}^m$ comprise a basis of $\sH^k_D(M_R)$, then the vectors $\rho^{-\ve} \sg_i$ are in the $L^2_{\la}$-orthogonal complement of $\sH^k_D(M_R)^\perp$. If the $\sg_i$'s are $L^2$-orthogonal, then
    \[
    \langle \rho^{-\ve} \sg_i, \rho^{-\ve} \sg_j \rangle_{L^2_{\la - \ve}} = \de_{ij}.
    \]
    So these vectors are linearly independent, have rank equal to $\dim \sH^k_D(M_R)$, and thus comprise a basis of the $L^2_{\la}$-orthogonal complement of $\sH^k_D(M_R)^\perp$. Let $\Pi$ denote the 
     $L^2_{\la}$-orthogonal projection of $H^s_\la \Omega^k(M_R)$ to $\sH^k_D(M_R)^\perp \cap H^s_\la \Omega^k(M_R)$. Then if $\zeta \in H^s_\la \Omega^k(M_R)$,
     \[
     \eta := (\id - \Pi)(\zeta) = \sum_{i=1}^m a_i \rho^{-\ve} \sg_i.
     \]
     \begin{clm}\label{claim adjoint harmonic field}
    Any kernel element $\eta$ of the $L^2_\la$-orthogonal projection onto $\sH^k_D(M_R)^\perp$ has the property that
    \[
    \rho^{-\ve + 2} d (\rho^{\ve} \eta), \; \rho^{-\ve+2} \de (\rho^{\ve} \eta )= 0.
    \]
    \end{clm}
    \begin{proof}[Proof of Claim]
    Any kernel element $\eta$ is given by a linear combination
    \[
    \eta = \sum_{i=1}^m a_i \rho^{-\ve} \sg_i.
    \]
    Then, 
    \begin{align*}
    \rho^{-\ve + 2 } d (\rho^{\ve} \eta) &= \rho^{-\ve + 2} d \bigg( \sum_{i=1}^m a_i \sg_i \bigg) \\
    &= \sum_{i=1}^m \rho^{-\ve + 2}  a_i d \sg_i \\
    &= 0.
    \end{align*}
    The vanishing of $\rho^{-\ve + 2} \de (\rho^{\ve} \eta )$ is proved exactly analogously.
    \end{proof}
    Thus, $\eta = (\id - \Pi)(\zeta)$ satisfies 
    \begin{equation}\label{equation dirichlet harmonic bounds}
    \left\{
    \begin{array}{lr}
    \rho^{-\ve + 4}\De \rho^{\ve} \eta = 0  & \;\; \text{ on } M_R\\ 
    \bt \eta = 0 \;\text{ and }\; \bt \rho^{-\ve + 2}\de (\rho^{\ve} \eta) = 0 & \;\; \text{ on } \dd M_R.
    \end{array}
    \right.
    \end{equation}
    This elliptic boundary-value problem satisfies the Lopatinski-Shapiro condition (see the proof of \cite[Lemma 1.6.5]{HodgeDecompSchwarz}). Thus by the proof of Proposition \ref{proposition lopatinski shapiro condition}, $\eta$ satisfies an apriori estimate of the form 
    \[
    \|\eta\|_{H^s_\la(M_R)} \le C_M \| \eta \|_{L^2_\la(M_R)},
    \]
    where $C_M$ does not depend on $R$. Since $\Pi$ is $L^2_\la$-orthogonal projection, 
    \[
    \|\eta\|_{L^2_\la(M_R)} \le \|\zeta\|_{L^2_\la(M_R)}. 
    \]
    Thus,
    \begin{align*}
        \|\Pi(\zeta)\|_{H^s_\la(M_R)} &\le C_M\|\zeta\|_{L^2_\la(M_R)} + \|\zeta\|_{H^s_\la(M_R)}\\
        &\le (1 + C_M)\|\zeta\|_{H^s_\la(M_R)}.
    \end{align*}
    This concludes the proof of the lemma.
\end{proof}

We now reframe the Green's operator as a map between weighted spaces. Define the space of $2$-times weakly differentiable $k$-forms obeying the homogeneous boundary conditions by
\[
H^{2}\Omega^k_{\mathrm{hom}} (M_R) = \{\om \in H^{2}\Omega_D^k(M_R) \;|\; \bt \de \om = 0\}.
\]
Restricting the domain of the Hodge Laplacian $\De$ to these homogeneous forms, the existence and regularity results of \cite[Theorems 2.2.4-6]{HodgeDecompSchwarz} imply that (c.f. \cite[p. 78]{HodgeDecompSchwarz})
\[
\ker(\De| H^{2}\Omega^k_{\mathrm{hom}} ) = \sH^k_D(M_R).
\]
We then consider $\Delta$ as an operator on the weighted space $H_{\la+2}^{2}\Omega^k_{\mathrm{hom}} (M_R)$, i.e. $H^{2}\Omega^k_{\mathrm{hom}} (M_R)$ seen as a subspace of $H^{2}_{\la+2}\Omega^k(M_R)$. The equivalence of the weighted and unweighed norms implies that
\[
\ker(\De| H_{\la+2}^{2}\Omega^k_{\mathrm{hom}} ) = \sH_D^k(M_R) \subset H_{\la+2}^{2}\Omega^k (M_R).
\]
The strategy will be to compare the kernel of the Hodge Laplacian on the complete manifold $M$ in the weighted space,
\[
\De: H^{2}_{\la + 2}\Omega^k(M) \rightarrow L^{2}_{\la}\Omega^k(M),
\]
to a space $V_R$ consisting of \emph{both} the kernel elements \emph{and} the ``small eigenforms" associated to the Hodge Laplacian on the space of homogeneous $k$-forms on $M_R$, 
\[
\De : H^{2}_{\la + 2}\Omega^k_{\mathrm{hom}}(M_R) \rightarrow L^{2}_{\la}\Omega^k(M_R).
\]

\subsection{Estimating the kernel of the Hodge Laplacian}

\subsubsection{Facts from the noncompact case}

We recall the following result for $AC$ manifolds from Karigiannis and Lotay (originally proved in \cite{Lockhart1987}, see Remark \ref{remark lockhart vs KL}). 

\begin{thm}\cite[Theorem 4.11]{KarigiannisLotay}\label{theorem fredholmness of laplacian} Let $(M,g)$ be an AC Riemannian manifold with asymptotic cone $C$. The map
\[
\De_M: H^{2}_{\la + 2}\Omega^k(M) \rightarrow L^{2}_{\la}\Omega^k(M)
\]
is Fredholm if and only if $\la +2$ is not in the set of \emph{critical rates} of the asymptotic cone $C$, which is a discrete, countable subset of $\R$ that has finite intersection with any closed, bounded interval of $\R$.
\end{thm}

When the rate $\la + 2$ is non-critical, the finite dimensionality of $\ker(\De_M)_{\la +2}$ gives the following orthogonal decomposition with respect to the $H^{2}_{\la + 2}$-inner product. 
\[
H^{2}_{\la + 2}\Omega^k(M) = \ker(\De_M)_{\la +2} \oplus \ker(\De_M)_{\la +2}^\perp.
\]
Observe that $\De_M$ restricted to $\ker(\De_M)_{\la +2}^\perp$ will be an isomorphism onto its image, with operator norm of the inverse bounded by a constant $\ka^{-1} > 0$.  

\subsubsection{A singular value decomposition for the Laplacian on domains} \label{section SVD}
 
Let $G$ be the solution (Green's) operator for the boundary value problem \eqref{equation dirichlet potential bvp}, such that 
\[
G: \sH^k_D(M_R)^\perp \subset L^2_{\la}\Omega^k(M_R) \rightarrow H^{2}_{\la +2} \Omega^k_{\mathrm{hom}} (M_R). 
\]
Note that the $\perp$ is with respect to the standard $L^2$-norm, since we've invoked \cite[Theorems 2.2.4-6]{HodgeDecompSchwarz} for the existence of $G$. In the strategy to find domains where the Green's function $G$ is uniformly bounded for all $R$, the key tool is a singular value decomposition for $G$ (and ultimately for $\De$).  

\begin{lem}\label{lemma compactness of Greens}
    The Green's operator is a compact mapping 
    \[
    G: \sH^k_D(M_R)^\perp \subset L^2_{\la}\Omega^k(M_R) \rightarrow L^{2}_{\la +2} \Omega^k (M_R).
    \]
\end{lem}

\begin{proof}
    Recall that for the compact domains $M_R$, the weighted Sobolev norms are equivalent to the unweighted ones, albeit via constants depending on $R$. Therefore, we may use the usual Sobolev embedding theorem to compactly embed $ H^{2}_{\la +2} \Omega^k_{\mathrm{hom}} (M_R)$ into the space $L^2_{\la +2} \Omega^k (M_R)$. Composing $G$ with this embedding completes the proof.
\end{proof}

\begin{rmk}
    Note that the embedding $H^{s+2}_{\beta} \Omega^k(M) \hookrightarrow H^{s}_{\ga} \Omega^k(M)$ is only compact for noncompact $M$ if $\beta < \la$. Since we are dealing with compact $M_R$, we can use the usual Sobolev embedding theorem and let $\beta = \ga = \la + 2$. However, the failure of this embedding to be compact for $M$ accounts for the fact that the spectrum of $\De$ in the weighted spaces on $M$ may be continuous, but approximated by discrete singular values on each subdomain $M_R$.
\end{rmk}

The compactness of $G$ allows us to establish the following singular value decomposition.

\begin{prop}\label{proposition singular value decomp for weighted}
    For any $\la$, there is 
    \begin{enumerate}
        \item an $L^2_{\la}$-orthonormal basis $\{\phi_i\}_{i\in \bN}$ of $\sH^k_D(M_R)^\perp \subset L^2_{\la}\Omega^k(M_R)$
        \item an $L^{2}_{\la + 2}$-orthonormal basis $\{\psi_i\}_{i\in \bN} \subset H^2\Om^k_{hom}(M_R)$ of $L^{2}_{\la + 2}\Omega^k(M_R)$.
        \item singular values $\{\sg_i\}_{i \in \bN}$ with $\sg_i > 0$, $\sg_i \rightarrow 0$ as $i \rightarrow \infty$
    \end{enumerate}  
    such that the Green's operator $G$ for the boundary value problem \eqref{equation dirichlet potential bvp} has the following form for $\eta \in \sH^k_D(M_R)^\perp$.
    \begin{equation}\label{equation SVD for greens function}
        G(\eta) := \sum_{i=1}^\infty \sg_i \langle \phi_i, \eta \rangle_\la \psi_i.
    \end{equation}
\end{prop}

\begin{proof}
    By Lemma \ref{lemma compactness of Greens}, $G$ is a compact operator from $\sH^k_D(M_R)^\perp$ to $L^{2}_{\la +2} \Omega^k (M_R)$. These are Hilbert spaces, so the adjoint map $G^*$ is a compact operator from $L^{2}_{\la +2} \Omega^k (M_R)$ to $\sH^k_D(M_R)^\perp$. The map 
    \[
    GG^* : L^2_{\la+2}\Omega^k(M_R) \rightarrow L^2_{\la+2}\Omega^k(M_R)
    \]
    is a compact, self-adjoint linear operator. There is, therefore, an orthonormal eigenbasis $\{\psi_i\}_{i = 1}^\infty$ of $L^2_{\la +2} \Om^k(M_R)$ with corresponding positive, real eigenvalues $\{\ka_i\}_{i=1}^\infty$ accumulating only at $0$. Consider the set $\{G^*(\psi_i)\}_i$. 
    \begin{align*}
        \langle G^*(\psi_i), G^*(\psi_j) \rangle_{\la}
        &= \langle GG^*(\psi_i), \psi_j \rangle_{\la + 2} \\
        &= \ka_i \langle \psi_i, \psi_j \rangle_{\la + 2} \\
        &= \ka_i \de_{ij}.
    \end{align*}
    Thus, the $G^*(\psi_i)$'s are an orthogonal set in $\sH^k_D(M_R)^\perp$. Note that $\ka_i = \|G^*(\psi_i)\|_\la^2 >0$. Now take $\phi_i := {\ka_i}^{-1/2}G^*(\psi_i)$ and $\sg_i = \ka_i^{1/2}$. The $\psi_i$'s are an orthonormal basis and $G^*$ is linear and surjective (since $G$ is injective), so the orthonormal set $\{\phi_i\}_i$ is actually a basis of $\sH^k_D(M_R)^\perp$. Thus, we can decompose $\eta$ as
    \[
    \eta = \sum_i \langle \phi_i, \eta \rangle_{\la} \phi_i.
    \]
    Then
    \begin{align*}
    G(\eta) &= \sum_i \langle \phi_i, \eta \rangle_{\la} G(\phi_i) \\
    &= \sum_i \sg_i \langle \phi_i, \eta \rangle_{\la} \psi_i.
    \end{align*}
    Note finally that $G(\phi_i) = \psi_i$, the definition of $G$ implies that $\psi_i \in H^2\Om^k_{hom}(M_R)$. This completes the proof.
\end{proof}

We wish to use the singular value decomposition to represent the action of the Hodge Laplacian and its adjoint on certain forms. However, in order to do this, the correct domains of the Hodge Laplacian and its adjoint need to be defined. The first step is to calculate the formal adjoint of $\De$.

\begin{prop}\label{proposition formal adjoint of laplacian}
    The Hodge Laplacian restricted to $H^2_{\la +2} \Omega^k_{hom}(M_R)$ maps into $ L^2_{\la} \Om^k(M_R)$ and has formal adjoint given by the formula
    \[
    \De^*\eta = \rho^{2\la + 4 + 7}\De (\rho^{-2\la -7} \eta),
    \]
    where $\eta \in L^2_{\la} \Om^k(M_R)$ is assumed to be smooth.
\end{prop}

\begin{proof} Let $\eta $ be a smooth form in $L^2_{\la} \Om^k(M_R)$, and let $\nu$ have two weak derivatives and satisfy the homogeneous boundary conditions. Apply Green's formula \cite[Proposition 2.1.2]{HodgeDecompSchwarz} and compute
    \begin{align*}
        \int_{M_R} \langle \eta, \De \nu \rangle \rho^{-2\la}\rho^{-7} dV & = \int_{M_R} \langle \eta, (d\de + \de d) \nu \rangle \rho^{-2\la -7} dV\\
        &= \int_{M_R} \langle \de \nu, \de (\rho^{-2\la- 7} \eta) \rangle dV + \int_{\dd M_R} \bt \de \nu \wedge \star \bn \rho^{-2\la- 7} \eta \\
        &\qquad + \int_{M_R} \langle d \nu, d (\rho^{-2\la- 7} \eta) \rangle dV + \int_{\dd M_R} \bt (\rho^{-2\la- 7} \eta) \wedge \star \bn d \nu \\
        &= \int_{M_R} \langle \de \nu, \de (\rho^{-2\la- 7} \eta) \rangle dV + \int_{M_R} \langle d \nu, d (\rho^{-2\la- 7} \eta) \rangle dV \\
        &\qquad \qquad + \int_{\dd M_R} \bt (\rho^{-2\la- 7} \eta) \wedge \star \bn d \nu,
    \end{align*}
    where the last line follows from the homogeneous boundary conditions on $\nu$ and the fact that $\bt d \nu = d \bt \nu$. Now we continue to apply Green's theorem to obtain the adjoint. 
    \begin{align}\label{equation adjoint via greens}
       &\int_{M_R} \langle \de \nu, \de (\rho^{-2\la- 7} \eta) \rangle dV + \int_{M_R} \langle d \nu, d (\rho^{-2\la- 7} \eta) \rangle dV + \int_{\dd M_R} \bt (\rho^{-2\la- 7} \eta) \wedge \star \bn d \nu \nonumber  \\ &\qquad = \int_{M_R} \langle  \nu, d\de (\rho^{-2\la- 7} \eta) \rangle dV - \int_{\dd M_R} \bt \de (\rho^{-2\la- 7} \eta) \wedge \star \bn \nu  \nonumber \\
       &\qquad \qquad + \int_{M_R} \langle \nu, \de d (\rho^{-2\la- 7} \eta) \rangle dV - \int_{\dd M_R} \bt \nu \wedge \star \bn d (\rho^{-2\la- 7} \eta) \nonumber\\
       &\qquad \qquad \qquad + \int_{\dd M_R} \bt (\rho^{-2\la- 7} \eta) \wedge \star \bn d \nu \nonumber \\
       &\qquad = \int_{M_R} \langle \nu, \rho^{2\la + 4 + 7}\De(\rho^{-2\la - 7} \eta) \rangle \rho^{-2\la -4 - 7} dV - \int_{\dd M_R} \bt \de (\rho^{-2\la- 7} \eta) \wedge \star \bn \nu \nonumber \\
       &\qquad \qquad + \int_{\dd M_R} \bt (\rho^{-2\la- 7} \eta) \wedge \star \bn d \nu 
    \end{align}
    To obtain the formal adjoint, we restrict our test forms $\nu$ to the smooth forms vanishing at the boundary $\Om^k_0(M_R)$. For such test forms $\nu$, the boundary terms in \eqref{equation adjoint via greens} vanish. Thus, the formal adjoint of $\De$ is given by the formula
    \[
    \De^*\eta = \rho^{2\la + 4 + 7}\De (\rho^{-2\la -7} \eta),
    \]
    where all necessary derivatives of $\eta$ are assumed to be defined. 
\end{proof}

The domain of the adjoint operator is defined by
\[
\sD(\De^*) = \{ \om \in L^2_{\la}\Om^k(M_R) \;|\; \exists \eta \in L^2, \; \forall \nu \in \sD(\De), \; \langle \om, \De \nu \rangle_{\la} = \langle \eta, \nu \rangle_{\la+2} \}.
\]
Since the smooth test forms vanishing on the boundary, $\Om^k_0(M_R)$, are contained in $H^2_{\la+2}\Om^k_{hom}(M_R)$, the integration by parts calculation in Proposition \ref{proposition formal adjoint of laplacian} means that 
\[
\De^* \om = \rho^{2\la + 4 + 7}\De (\rho^{-2\la -7} \om ) = \eta 
\]
in the distributional sense. However, since we are testing against a larger set of functions, $H^2_{\la+2}\Om^k_{hom}(M_R) \supset \Om^k_0(M_R)$, forms $\om$ in $\sD(\De^*)$ must also satisfy (in the distributional sense) dual boundary conditions to the homogeneous boundary conditions (see Corollary \ref{corollary vanishing of boundary terms}). 

The following lemma provides expressions for the Laplacian and its adjoint on domains of interest.  

\begin{lem}\label{lemma SVD for hodge laplacian}
    The Hodge Laplacian on the domain $\sD(\De)$ taken as a subset of $L^{2}_{\la + 2}\Om^k(M_R)$ has the form
    \begin{equation}\label{equation SVD for hodge laplacian}
        \De \nu := \sum_{i=1}^\infty \sg_i^{-1} \langle \psi_i, \nu \rangle_{\la + 2} \phi_i.
    \end{equation}
    The adjoint $\De^*$ has the following form when applied to forms $\eta \in \sH^k_D(M_R)^\perp \cap \sD(\De^*)$.
    \begin{equation}\label{equation SVD for exact adjoint laplacian}
        \De^* \eta := \sum_{i=1}^\infty \sg_i^{-1} \langle \phi_i, \eta \rangle_{\la} \psi_i.
    \end{equation}
\end{lem}

\begin{proof}
    It follows from the injectivity of the Green's function $G$ that the formula \eqref{equation SVD for greens function} in Proposition \ref{proposition singular value decomp for weighted} implies the formula \eqref{equation SVD for hodge laplacian} for $\nu \in G(\sH^k_D(M_R)^\perp)$. However, we can represent \[
    H^2\Om^k_{hom}(M_R) \cong G(\sH^k_D(M_R)^\perp) \oplus \ker \De|_{H^2\Om^k_{hom}} \cong G(\sH^k_D(M_R)^\perp) \oplus \sH^k_D(M_R),
    \]
    where $\sH^k_D(M_R)$ is a finite-dimensional vector space. Any sequence in $\sH^k_D(M_R)^\perp$ approaching zero is mapped by $G$ to a sequence approaching zero in the $H^2_{\la+2}$-norm (and consequently the $L^2_{\la+2}$-norm). Thus, the kernel of $\De$, $\sH^k_D(M_R)$, can be considered to be in the $L^2_{\la+2}$-orthogonal complement of $G(\sH^k_D(M_R)^\perp)$, and thus the formula \eqref{equation SVD for hodge laplacian} may be extended to all forms in $\sD(\De)$.
    
    To obtain the adjoint formula, let $\eta \in \sH^k_D(M_R)^\perp$. Then, for all $\nu \in H^2 \Om^k_{hom} (M_R)$,
    \begin{align*}
        \langle \eta, \De \nu \rangle_{\la} &= \sum_{i=1}^\infty \sg_i^{-1} \langle \psi_i, \nu \rangle_{\la + 2} \langle \phi_i, \eta \rangle_{\la} \\
        &= \bigg \langle \sum_{i=1}^\infty \sg_i^{-1} \langle \phi_i, \eta \rangle_{\la} \psi_i, \nu \bigg\rangle_{\la+2}.
    \end{align*}
    This establishes formula \eqref{equation SVD for exact adjoint laplacian}, when the weak derivative 
    \[
    \sum_{i=1}^\infty \sg_i^{-1} \langle \phi_i, \eta \rangle_{\la} \psi_i
    \]
    is in $L^2$, i.e. when $\eta \in \sD(\De^*)$.
\end{proof}

\begin{cor}\label{corollary smoothness of eigenforms}
The orthogonal basis forms $\{\phi_i\}_i$ in Proposition \ref{proposition singular value decomp for weighted} are smooth on the interior of $M_R$. 
\end{cor}

\begin{proof}
First, we show that the basis vectors $\phi_i$ are in the domain of the adjoint, $\sD(\De^*)$. For any test form $\nu \in H^2 \Om^k_{hom}(M_R)$, observe that by Lemma \ref{lemma SVD for hodge laplacian}, 
\begin{align*}
    \langle \phi_i , \De \nu \rangle_{\la} &= \bigg \langle \sum_{j=1}^\infty \sg_j^{-1} \langle \phi_j, \phi_i \rangle_{\la} \psi_j, \nu \bigg\rangle_{\la+2}\\
    &= \langle \sg_i^{-1} \psi_i, \nu \rangle_{\la+2}
\end{align*}
Thus, in the distributional sense $\rho^{2\la + 4 + 7}\De(\rho^{-2\la -7}\phi_i) = \sg_i^{-1} \psi_i$, which is in $H^2 \subset L^2$. Thus the $\phi_i$'s are in $\sD(\De^*)$. 

By Lemma \ref{lemma SVD for hodge laplacian} and the previous calculation, $\De \De^* \phi_j = \sg_j^{-2} \phi_i$ in the distributional sense, and in particular is in $L^2_{\la}\Om^k(M_R)$. Then, $\phi_j \in L^2_{\la}$ weakly solves the elliptic equation 
\[
\De\De^* \phi_j - \sg_j^{-2} \phi_j = 0. 
\]
The interior elliptic regularity proved in \cite[Corollary 10.3.9]{NicolaescuManifoldLectures} then imply that $\phi_j$ is in $H^k_{loc}$ for all $k > 0$. 
\end{proof}

\begin{cor}\label{corollary equality of adjoints}
    For each $i \in \bN$, 
    \begin{equation}\label{equation adjoints and eigenvalues}
    \rho^{2\la + 4 +7} \De (\rho^{-2\la - 7} \phi_i) = \De^* \phi_i = \sg_i^{-1} \psi_i
    \end{equation}
    as smooth forms. 
\end{cor}

\begin{proof}
Equation \eqref{equation SVD for exact adjoint laplacian} implies that $\De^* \phi_i = \sg_i^{-1} \psi_i$ as $L^2$ forms.  The smoothness of $\phi_i$ implies that $\De^*$ is in this case equal to its formal adjoint, which was derived in Proposition \ref{proposition formal adjoint of laplacian}. That is,
\[
\rho^{2\la + 4 +7} \De (\rho^{-2\la - 7} \phi_i) = \De^* \phi_i.
\]
Finally, note that the pointwise almost-everywhere equivalence in \eqref{equation adjoints and eigenvalues} and the smoothness of $\phi_i$ implies that $\psi_i$ is smooth on the interior of $M_R$. 
\end{proof}

\begin{prop}\label{proposition smoothness up to boundary of eigenfunctions}
    For any $k > 0$ and all $i \in \bN$, $\phi_i$ and $\psi_i$ are in the Sobolev space $H^k(M_R)$.
\end{prop}

\begin{proof}
    The $\psi_i$'s are in $L^2$ and satisfy the boundary value problem
    \begin{align*}
        \De^* \De \psi_i - \sg_i^{-2}\psi_i &= 0\\
        \bt \psi_i = 0 \; \textrm{ and } \; \bt \de \psi_i &= 0 
    \end{align*}
    pointwise. The principal symbol of the operator $\De^*\De$ satisfies the condition for ellipticity, and the boundary operator $(\bt, \bt \de)$ satisfies ellipticity conditions in the sense of the Lopatinski-Shapiro condition (see \cite[Lemma 1.6.5]{HodgeDecompSchwarz}). Since the boundary data and the right hand side of the equation vanish and are therefore smooth, the apriori estimate \cite[(ii), Theorem 1.6.2]{HodgeDecompSchwarz} for elliptic operators in the sense of Lopatinski-Shapiro implies that $\psi_i$ is in $H^k(M_R)$ for any given $k$. Since $\De \psi_i = \sg_i^{-1} \phi_i$, the same is true for the $\phi_i$'s.
\end{proof}

\begin{rmk}
Note that on the domain $M_R$, the weighted Sobolev norms are equivalent to the unweighted ones, and the spaces $H^s\Om^k(M_R) = H^s_{\la}\Om^k(M_R) = H^s_{\la + 2}\Om^k(M_R)$.
\end{rmk}

The following corollary will be useful for estimating the ``approximate kernel" later on in the paper.

\begin{cor}\label{corollary vanishing of boundary terms}
The orthogonal basis forms $\{\phi_i\}_i$ in Proposition \ref{proposition singular value decomp for weighted} have the property that
\[
\bt \de (\rho^{-2\la -7} \phi_i) = 0,
\]
and
\[
\bt (\rho^{-2\la -7} \phi_i) = 0,
\]
for all $i \in \bN$.
\end{cor}

\begin{proof}
    Let $\nu$ be a k-form with sufficient regularity that the integration by parts calculation in Proposition \ref{proposition formal adjoint of laplacian} holds. Then, given a $\phi_i$ in the orthonormal basis, 
    \begin{align}\label{equation integration by parts with bdry}
        & \int_{M_R} \langle \phi_i, \De \nu \rangle \rho^{-2\la} \rho^{-7} dV -  \int_{M_R} \langle \De^* \phi_i, \nu \rangle \rho^{-2\la-4}\rho^{-7} dV \\
        &\qquad = \int_{\dd M_R} \bt \de \nu \wedge \star \bn \rho^{-2\la- 7} \phi_i + \int_{\dd M_R} \bt (\rho^{-2\la- 7} \phi_i) \wedge \star \bn d \nu \nonumber\\
        &\qquad \qquad   -\int_{\dd M_R} \bt \de (\rho^{-2\la- 7} \phi_i) \wedge \star \bn \nu - \int_{\dd M_R} \bt \nu \wedge \star \bn d (\rho^{-2\la- 7} \phi_i). \nonumber
    \end{align}
    We will focus first on the term 
    \[
    \int_{\dd M_R} \bt \de (\rho^{-2\la- 7} \phi_i) \wedge \star \bn \nu = \int_{\dd M_R} \bt \de (\rho^{-2\la- 7} \phi_i) \wedge \bt (\star \nu).
    \]
    If $\bt (\star \nu)$ could be taken to be any smooth form on $\dd M_R$, and this integral were known to vanish for all such choices, the corollary would be proved. To this end, let $\eta$ be an arbitrary smooth form in $\Om^{k+1}(\dd M_R)$, and $\tilde \eta$ be its smooth extension to $\Om^{k+1} (M_R)$. We then seek a form $\nu \in \Om^{k+1} (M_R)$ such that
    \begin{equation}
        \nu|_{\dd M_R} = \star \tilde \eta|_{\dd M_R}, \; \bn d\nu = 0, \; \bt \de \nu = 0
    \end{equation}
    where $\star$ is taken with respect to the metric on $M_R$. This extension problem can be made homogeneous by rearranging terms
    \begin{equation}\label{equation homogeneous extension problem}
        \hat \nu|_{\dd M_R} = 0, \; \bn d\hat \nu = - \bn d (\star \tilde \eta), \; \bt \de \hat \nu = - \bt (\de \star \tilde \eta)
    \end{equation}
    By \cite[Lemma 3.3.2]{HodgeDecompSchwarz}, there is a smooth solution $\hat \nu$ to the homogeneous extension problem \eqref{equation homogeneous extension problem}. This gives us a form $\nu \in \Om^{k+1} (M_R)$ such that
    \[
    \bt (\star \nu) = \eta, \; \bn (\star \nu) = 0, \; \bn d\nu = 0, \; \bt \de \nu = 0
    \]
    Duality gives 
    \begin{align*}
        \bn(\star \nu) &= \star \bt \nu \\
        &= 0\\
        \implies \bt \nu & = 0
    \end{align*}
    Similarly, we deduce that $\star \bn d \nu = 0$. Thus, for such a choice of $\nu$, the right-hand side of \eqref{equation integration by parts with bdry} can be simplified to be
    \[
    \int_{\dd M_R} \bt \de (\rho^{-2\la- 7} \phi_i) \wedge \eta. 
    \]
    Notice that our choice of $\nu$ is smooth and satisfies the homogeneous boundary conditions $\bt \nu, \bt \de \nu = 0$, i.e. $\nu \in H^2_{\la+2}\Om^k_{hom}(M_R)$. Appealing to \eqref{equation SVD for hodge laplacian}, 
    \begin{align*}
    \int_{M_R} \langle \phi_i, \De \nu \rangle \rho^{-2\la} \rho^{-7} dV &= \int_{M_R} \bigg\langle \phi_i, \sum_{j=1}^\infty \sg_j^{-1} \langle \psi_j, \nu \rangle_{\la + 2} \phi_j \bigg\rangle \rho^{-2\la} \rho^{-7} dV  \\
    &=\sg_i^{-1} \langle \psi_i, \nu \rangle_{\la + 2}.
    \end{align*}
    Furthermore, by \eqref{equation adjoints and eigenvalues}, 
    \begin{align*}
        \int_{M_R} \langle \De^* \phi_i, \nu \rangle \rho^{-2\la-4}\rho^{-7} dV  &= \int_{M_R} \langle \sg_i^{-1} \psi_i, \nu \rangle \rho^{-2\la-4}\rho^{-7} dV \\
        &=\sg_i^{-1} \langle \psi_i, \nu \rangle_{\la + 2}. 
    \end{align*}
    Thus, the left-hand side of \eqref{equation integration by parts with bdry} vanishes for any such choice of $\nu$.  Putting everything together, for any smooth $\eta \in \Om^{k+1}(\dd M_R)$, 
    \[
    0 = \int_{\dd M_R} \bt \de (\rho^{-2\la- 7} \phi_i) \wedge \eta. 
    \]
    Therefore, $\bt \de (\rho^{-2\la- 7} \phi_i) = 0$.

    Next, consider the term
    \begin{equation}\label{equation adjoint dirichlet boundary integral}
    \int_{\dd M_R} \bt (\rho^{-2\la- 7} \phi_i) \wedge \star \bn d \nu.
    \end{equation}
    Take an arbitrary 3-form $\eta$ defined on the boundary $\dd M_R$, and consider an arbitrary extension $\hat \eta$ into the interior of $M_R$. Consider $\star \hat \eta$ and invoke \cite[Lemma 3.3.2]{HodgeDecompSchwarz} to obtain a form $\nu$ such that 
    \[
    \nu \big|_{\dd M} = 0, \; \bn (d \nu) = \bn (\star \hat \eta), \; \text{and}\; \bt(\de \nu) = 0.  
    \]
    The first and last of these identities ensure that all boundary terms except \eqref{equation adjoint dirichlet boundary integral} vanish, and that $\nu$ satisfies the homogeneous boundary conditions. Thus, by the same logic as in the first part, we obtain the identity
    \begin{align*}
        0 &= \int_{\dd M_R} \bt (\rho^{-2\la- 7} \phi_i) \wedge \star \bn d \nu \\
        &= \int_{\dd M_R} \bt (\rho^{-2\la- 7} \phi_i) \wedge \star \bn (\star \hat \eta) \\
        &=\int_{\dd M_R} \bt (\rho^{-2\la- 7} \phi_i) \wedge \bt (\hat \eta) \\
        &=\int_{\dd M_R} \bt (\rho^{-2\la- 7} \phi_i) \wedge \eta.
    \end{align*}
    As $\eta$ is an arbitrarily chosen form on $\dd M_R$, it follows that $\bt (\rho^{-2\la- 7} \phi_i) = 0$.
\end{proof}

\begin{prop}\label{proposition uniform boundedness of projection} 
    Let the inverse $(\De^*\De)^{-1}$ indicate the inverse of the operator $\De^*\De$ on the noncompact manifold $M$ with domain restricted to $\ker(\De_M)^\perp_{\la+2}$. Let $U_R$ be the  span of $\{\phi_i\}_{i= N}^\infty$, where $N$ is the smallest number such that for $i \ge N$,  $\sg_i^2 < C \|(\De^*\De)^{-1}\|$, where $C>0$ is a fixed positive constant. 
    
    Let the projection $\mathrm{Proj}_{U}: (\sH^k_D(M_R))^\perp \rightarrow U_R$ be the $L^2_{\la}$-orthogonal projection of $(\sH^3_D(M_R))^\perp$ onto $U_R$. Then, $\mathrm{Proj}_{U}$ is uniformly bounded in the $H^{s-3}_{\la}$-norm, i.e.
    \[
    \|\mathrm{Proj}_{U}(\eta)\|_{H^{s-3}_\la} \le C_M \|\eta\|_{H^{s-3}_\la},
    \]
    for all $\eta \in H^{s-3}_{\la}\Om^k(M_R) \cap (\sH^k_D(M_R))^\perp$, and where $C_M > 0$ does not depend on $R$.
\end{prop}

\begin{proof}
    Each $\eta \in H^{s-3}_{\la}\Om^k(M_R) \cap (\sH^k_D(M_R))^\perp$ can be uniquely decomposed into
    \[
    \eta = \mathrm{Proj}_{U}(\eta) + \big(\eta -\mathrm{Proj}_{U}(\eta) \big).
    \]
    The last term is the projection of $\eta$ onto the finite dimensional subspace spanned by $\{\phi_i\}_{i= 1}^N$. That is,
    \begin{align*}
        \nu := \eta -\mathrm{Proj}_{U}(\eta)  &= \sum_{i=1}^N \langle \eta, \phi_i \rangle_{\la} \phi_i.
    \end{align*}
    For any $k > 0$,
    \begin{equation*}
        \left\{
            \begin{aligned}
                (\De \De^*)^k \nu &= \sum_{i=1}^N \langle \eta, \phi_i \rangle \sg_i^{-2k} \phi_i,\\
                \bt (\rho^{-2\la - 7} \nu) = 0 \;&\textrm{ and }\; \bt \de (\rho^{-2\la - 7} \nu) = 0,
            \end{aligned}
        \right.
    \end{equation*}
    where the boundary conditions are an immediate consequence of the boundary conditions established in Corollary \ref{corollary vanishing of boundary terms}. The adjoint boundary operators can be seen to satisfy the Lopatinski-Shapiro condition by the same argument as employed in \cite[Lemma 1.6.5]{HodgeDecompSchwarz}. Thus, the elliptic Schauder estimates in Corollary \ref{corollary lopatinski shapiro higher order} imply that there exists a constant $c > 0$ which does not depend on $R$ such that
    \begin{align*}
    \|\nu\|_{H^{2k}_{\la}} &\le c\big(\sg_{N}^{-2k}\|\nu\|_{L^2_{\la}} + \|\nu\|_{L^2_{\la}}\big)  \\
    &\le c(\sg_N^{-2k} + 1)\|\eta\|_{L^2_\la} \\
    &\le c(\|(\De^*\De)^{-1}\|^{-2k} + 1 ) \|\eta\|_{L^{2}_{\la}}.
    \end{align*}
    Note that $\|(\De^*\De)^{-1}\|$ depends only on the \textit{fixed} background metric on $M$ (and, in particular, in the case of the Laplacian flow, the metric on the initial time-slice, $g_{\vp_0}$). Now evaluate the norm 
    \begin{align*}
         \|\mathrm{Proj}_{U}(\eta)\|_{H^{s-3}_{\la}} &= \big\|\eta - \big(\eta -\mathrm{Proj}_{U}(\eta) \big)\big\|_{H^{s-3}_{\la}}\\
         &\le \|\eta\|_{H^{s-3}_{\la}} + \big\|\eta -\mathrm{Proj}_{U}(\eta)\big\|_{H^{s-3}_{\la}} \\
         &\le \|\eta\|_{H^{s-3}_{\la}}  + c(\|(\De^{-1})(\De^{-1})^*\|^{-k} + 1 ) \|\eta\|_{L^{2}_{\la}}\\
         &\le c(\|(\De^*\De)^{-1}\|^{-k} + 2) \|\eta\|_{H^{s-3}_{\la}}.
    \end{align*}
    This completes the proof of the proposition.
\end{proof}

\begin{prop}\label{proposition boundedness of restricted forms}
As before, let $U_R$ be the span of $\{\phi_i\}_{i= N}^\infty$, where $N$ is the smallest number such that for $i \ge N$,  $\sg_i^2 < C \|(\De^*\De)^{-1}\|$ for $C>0$ a fixed positive constant. Let $s \ge 3$ and let
\[
H_{\la}^{r-1,s-3}\sU^3(M_R \times I)
\]
be the image of $\mathrm{Proj}_U$ along time-slices of $H_{0,\la}^{r-1,s-3}\Om^3(M_R \times I)$ to $U_R$. Then, let 
\[
H_{0,\la}^{r-1,s-3}\sU^3(M_R \times I)
\]
be the image of the projection $\mathrm{Proj}_U$ along time-slices of the subspace $H_{0,\la}^{r-1,s-3}\Om^3(M_R \times I)$ of forms which vanish at the initial time to order $r-2$ in $t$. Then, if $\eta$ is in $H_{0,\la}^{r-1,s-3}\sU^3(M_R \times I)$, there is a potential $\xi \in H_{0,\la+1}^{r-1,s-2}\Omega_D^2(M_R \times I)$ such that
\[
d\xi = \mathrm{Proj}_{\sE} (\eta) \in H_{0,\la}^{r-1,s-3}\sE^3(M_R \times I),
\]
where $\mathrm{Proj}_{\sE}$ is the projection to the exact part of the Hodge-Morrey decomposition. Furthermore, there is a potential $\om \in H_{0,\la+1}^{r-1,s-2}\Omega^4(M_R \times I)$ such that 
\[
\de \om = \mathrm{Proj}_{\sC \oplus \sH_{co}} (\eta) \in H_{0,\la}^{r-1,s-3}(\sC^3 \oplus \sH_{co}^3) (M_R \times I),
\]
where $\mathrm{Proj}_{\sC \oplus \sH_{co}}$ is the projection to the co-exact piece of the Helmholtz decomposition given in \eqref{equation helmholtz decomp}. These two potentials satisfy the following estimates for all $t_0 \in I$. 
\begin{align}\label{equation timeslice potential estimates}
\|\dd_t^j \om (t_0)\|_{H_{\la+1 - 2j}^{s-2-2j}\Omega^4(M_R)},\; \|\dd_t^j \xi (t_0)\|_{H_{\la+1 - 2j}^{s-2-2j}\Omega^2(M_R)} \le C_M \|\dd_t^j \eta (t_0)\|_{H_{\la-2j}^{s-3-2j}\Omega^3(M_R)},
\end{align}
where $j \in \{0, \ldots, r-1\}$ and $C_M > 0$ does not depend on $R$. Note in particular that this implies that the maps 
\[
\mathrm{Proj}_{\sE} \text{ and } \mathrm{Proj}_{\sC \oplus \sH_{co}}
\]
are uniformly bounded maps on $H_{0,\la}^{r-1,s-3}\sU^3(M_R \times I)$.
\end{prop}

\begin{proof}
Let $\eta(t)$ be a family of 3-forms parametrized by $t$ in $H_{0,\la}^{r-1,s-3}\sU^3(M_R \times I)$. Fix $t \in I$. Then
\begin{align*}
    G(\eta(t)) &= \sum_i \sg_i \langle \phi_i, \eta(t) \rangle_{L^2_\la} \psi_i.
\end{align*}
This formula implies that the $L^2_{\la+2}$-norm of $G(\eta(t))$ is bounded by the $L^2_{\la}$-norm of $\eta(t)$ times a constant depending only on the operator norm $\|\De^{-1}(\De^{-1})^*\|$ on the whole space $M$ with respect to the initial metric $g_0$. This observation paired with the elliptic Schauder estimates in Proposition \ref{proposition lopatinski shapiro condition} gives higher order estimates in the spatial derivatives. Setting $\xi(t) := \de G( \eta(t))$ yields \eqref{equation timeslice potential estimates} when $j = 0$.

Since $\phi_i$ and $\psi_i$ are time-independent, differentiating the $\xi$ in time yields
\begin{align}\label{equation commute time and greens}
    \dd_t G(\eta) &= \dd_t \bigg(\sum_i \sg_i \langle \phi_i, \eta \rangle_{L^2_\la} \psi_i\bigg)\\
    &= \sum_i \sg_i \langle \phi_i, \dd_t \eta \rangle_{L^2_\la} \psi_i \nonumber\\
    &= G(\dd_t \eta) \nonumber
\end{align}
Identical calculations to those above yield the estimates in terms of the Sobolev norms of the time derivative $\dd_t \eta$. Performing these calculations for $j = 2,\ldots,r-1$ establishes that $G(\eta) \in H^{r-1, s-1}_{\la+2}\Omega^3_{hom} (M \times I)$. Since the metric is fixed at the initial time, the time derivative $\dd_t$ and the co-differential $\de$ commute. Set $\xi := \de G(\eta)$. Note that the homogeneous boundary condition on $G(\eta)$ implies that $\xi$ satisfies the Dirichlet boundary condition. Observe that for all $t_0 \in I$,
\begin{align*}
    \big\|\dd_t^j \xi \big|_{t_0} \big\|_{H_{\la+1 - 2j}^{s-2-2j}} &= \big\|\dd_t^j \de G(\eta) \big|_{t_0} \big\|_{H_{\la+1 - 2j}^{s-2-2j}} \\
    &= \big\|\de \dd_t^j G( \eta) \big|_{t_0} \big\|_{H_{\la+1 - 2j}^{s-2-2j}} \\
    &\le \big\|\dd_t^j G( \eta )\big|_{t_0} \big\|_{H_{\la+2 - 2j}^{s-1-2j}} \\
    &\le C_M \big\|\dd_t^j \eta \big|_{t_0} \big\|_{H_{\la-2j}^{s-3-2j}}.
\end{align*}
The last line follows from the commutativity of $G$ and $\dd_t$ shown in \eqref{equation commute time and greens}. In particular, this estimate implies that $\xi \in H^{r-1, s-2}_{\la+1}\Omega^2_{D} (M \times I)$. Furthermore, for $j = 0,\ldots, r-2$, this further implies that if $\eta \in H_{0,\la}^{r-1,s-3}\sU^3(M_R \times I)$
\[
\|\dd_t^j \xi (t_0)\|_{H_{\la+1 - 2j}^{s-2-2j}} \rightarrow 0, \; \text{ as } t_0 \rightarrow 0.
\]
Thus, $\xi \in H_{0,\la+1}^{r-1,s-2}\Omega_D^2(M_R \times I)$ and satisfies \eqref{equation timeslice potential estimates}. Setting $\om :=  dG(\eta)$ and repeating the proof given above with the natural modifications for the coexact case completes the proof of the theorem.
\end{proof}

\subsection{Estimating the approximate kernel} \label{section est approx ker}
We only solve the linearized equation up to a projection onto the exact part of the Hodge-Morrey decomposition, modulo an approximate kernel of ``small eigenvectors" and a co-exact component. Thus, the sequence of flows on the $M_R$'s that we obtain from the inverse function theorem will only solve \eqref{equation laplacian deturck boundary} up to these extra terms. In this section, we show that the contributions of the approximate kernel to the $H^k_{loc}$-limit of the inhomogeneous term must lie perpendicular to the image of the Dirac operator $d + \de$.

However, the singular values represented in the approximate kernel do not necessarily concentrate at zero. Thus, we must consider separately: the space of ``intermediate" kernel elements, $W_{\ga,R}$, with small singular values that are nonetheless greater than a fixed small parameter $\ga$; and the space of ``small" kernel elements $V_{\ga, R}$, which consists of the remaining small singular values. 

\subsubsection{The space of ``intermediate eigenforms"}

In this section, we assume that $\la + 2 \in (-2,-1)$ is a non-critical rate. Let $\{\psi_{R,i}\}$ and $\{\phi_{R,i}\}$ be the orthonormal bases of $L^2_{\la+2}\Om^k(M_R)$ and $L^2_{\la}\Om^k(M_R)$, respectively, which were defined in Proposition \ref{proposition singular value decomp for weighted}. Let $C_{\la} > 0$ be an upper bound on the norm of the operator $(\De^*\De)^{-1}$ on the noncompact manifold $(M,g_{\vp_0})$, where the inverse is of the operator $\De^*\De$ restricted to $\ker(\De_M)^\perp_{\la+2}$. For each $R$ and a positive parameter $\ga$, consider the set of eigenforms with intermediate singular values.
\[
W_{\ga,R} := \mathrm{span}\bigg\{\phi_{R,i}\;|\; 0 < \ga^2 \le \sg_{R,i}^{-2} \le \frac{C_{\la}^{-1}}{2}\bigg\}.
\]
For each $R$, this is a finite-dimensional subspace.

\begin{lem}\label{lemma remainder vanishes on intermediate spec}
     Let $\{\phi_j \in W_{\ga, R_j} \}_j$ be an $L^2_{\la}$-bounded sequence of vectors in $W_{\ga, R_j}$. For any given $\ga > 0$ and $k \ge 2$, the $H^k_{loc}$-limit of the $\phi_j$'s vanishes.
\end{lem}

\begin{proof}
    Let $\{\phi_j \in W_{\ga, R_j} \}_j$ be an $L^2_{\la}$-bounded sequence of vectors in $W_{\ga, R_j}$. Then, write
    \[\phi_j = \sum_{i=1}^k a_i \phi_{R_j, i}, \; \text{ where }\sum_{i=1}^k a_i^2 = \|\phi_j\|_{L^2_{\la}}^2.\]
    Take the associated forms
    \[\psi_j = \sum_{i=1}^k a_i \sg_{R_j,i}\psi_{R_j, i}, \; \text{ such that }\De \psi_j = \phi_j.\]
    The sequence $\{\psi_j\}_j$, is uniformly bounded in any Sobolev space $H^k\Omega^3(K)$, when the $\psi_j$ are restricted to a compact neighborhood $K \subset M$. This follows from the weighted Schauder estimates (Proposition \ref{proposition lopatinski shapiro condition}).
    \begin{align*}
        \|\psi_j\|_{H^{4k}_{\la+2}} &\le c\big(\|(\De^*\De)^k \psi_j\|_{L^2_{\la+2}} + \|\psi_j\|_{L^2_{\la+2}}\big) \\
        &\le c\big(\|\psi_j\|_{L^2_{\la+2}} + \|\psi_j\|_{L^2_{\la+2}}\big) \\
        &\le c \ga^{-1} \|\phi_j\|_{L^2_\la}\\
        &\le c
    \end{align*}
    Note that the formula for $\psi_j$ implies that $c$ depends on the choice of $\ga$ and $k$, as well as the fixed constant $C_\la$ and the upper bound on norms of the sequence $\phi_j$. Thus, restricted to any compact set $K \subset M$, the sequence of functions $\{\psi_j|_K\}_j$ is uniformly bounded in $H^k\Omega^3(K)$. 

    Furthermore, if we take the extension operators $E_{R_j}$ from Construction \ref{construction extension to collar neighborhood} and apply them to each of the $\psi_j$'s, we find that by Lemma \ref{lemma extension of functions} the sequence $\{E_{R_j}(\psi_j)\}_j$ is uniformly bounded in $L^2_{\la+2}(M)$. Choose a small $\ve > 0$ such that there are no critical rates of $\De$ in $[\la+2-\ve, \la+2+\ve]$. Upon taking a diagonal subsequence, $\{\psi_{j_\ell}\}_\ell$, we can assume that $\psi_{j_\ell} \xrightarrow{H^k_{loc}} \psi$, where $\psi \in H^k_{loc}(M)$, by the uniform boundedness of the $H^k$-norms on compact sets $K$. Furthermore, this implies that $E_{R_{j_\ell}}(\psi_{j_\ell}) \xrightarrow{H^k_{loc}} \psi$, and since $\max_{j_\ell} \|E_{R_{j_\ell}}(\psi_{j_\ell})\|_{L^2_{\la +2}} < C$, the weak lower semi-continuity of the weighted norms $\| \cdot \|_{L^2_{\la+2}(M)}$ implies that $\|\psi\|_{L^2_{\la + 2}} < \infty$.

    \begin{clm}\label{claim compact vanishing}
    Fix a compactly-supported test function $\mu \in \Om^3_0(M)$. For any uniformly $L^2_\la$-bounded sequence of unit vectors $\{\phi_j \in W_{\ga, R_j} \}_j$,
    \[
    \limsup_j |\langle \mu, \phi_j \rangle_{L^2_\la}| = 0.
    \]
    \end{clm}

    \begin{proof}[Proof of Claim \ref{claim compact vanishing}]
    Note that the compactly-supported test function $\mu \in \Om^3_0(M) \subset L^2_{\la'}\Omega^3(M)$, where $\la' < \la$, and suppose for the sake of contradiction that there exists a small number $\ka >0$ and $L^2_\la$-bounded sequence of unit vectors $\{\phi_j \in W_{\ga, R_j} \}_j$  such that 
    \[
    |\langle \mu, \phi_j \rangle_{L^2_\la}| > \ka >0,
    \]
    
    Assume first that $\psi$ vanishes. Then, given a domain $M_{R_0}$ and $\ve >0$, there exists $\psi_{j_\ell}$ such that $\|\psi_{j_\ell}\|_{H^2_{\la+2}(M_{R_0})} < \ve$. Since $\mu$ has order $\la' < \la$ decay, the following estimate must hold.
    \begin{align*}
        |\langle \mu, \phi_{j_\ell} \rangle_{L^2_\la} | &= |\langle \mu, \De \psi_{j_\ell} \rangle_{L^2_\la} | \\
        &\le \bigg| \int_{M_{R_0} } \langle \mu, \De \psi_{j_\ell} \rangle \rho^{-2\la - 7} dV \bigg| + \bigg| \int_{ M_{r_{j_\ell}} \setminus M_{R_0} } \langle \mu, \De \psi_{j_\ell} \rangle \rho^{-2\la - 7} dV \bigg| \\
        &\le \|\mu\|_{L^2_{\la}(M)} \|\psi_{j_\ell}\|_{H^2_{\la+2}(M_{R_0})} + R_0^{\la' -\la}\|\mu\|_{L^2_{\la'}(M)}\|\psi_{j_\ell}\|_{H^2_{\la + 2}(M_{R_{j_\ell}})}\\
        &\le C_{\mu}\ve + C_{\mu}R_0^{\la'-\la}\\
        &\le \ka,
    \end{align*}
    for $R_0 >0$ sufficiently large, and $\ve >0$ sufficiently small. This contradicts the assumption of a non-vanishing $H^2_{loc}$ limit.  
    
    If $\psi$ does not vanish, then we can prove that $\psi$ must be perpendicular to $\ker(\De)_{\la+2}$.  Since there are no critical rates between $\la + 2$ and $\la + 2 - \ve$, $\ker(\De)_{\la+2} = \ker(\De)_{\la+2 - \ve}$ by \cite[Theorem 4.10]{KarigiannisLotay} (see also \cite{Lockhart1987}, \cite{LockhartMcOwen1985}). Thus, if $\eta \in \ker (\De)_{\la + 2}$, then $\|\eta\|_{L^2_{\la + 2 -\ve}(M)} < \infty$. Then, given $R_0 > 0$, compute
    \begin{align*}
        &\langle \eta, \psi_{j_\ell} \rangle_{L^2_{\la + 2}(M_{R_{j_\ell}})}\\
        &\qquad = \int_{M_{R_{j_\ell}}} \langle \rho^{-\la - 2} \eta, \rho^{-\la - 2} \psi_{j_\ell} \rangle \rho^{-7} dV \\
        &\qquad = \int_{M_{R_0}}\langle \rho^{-\la - 2} \eta, \rho^{-\la - 2} \psi_{j_\ell} \rangle \rho^{-7} dV \\
        &\qquad \qquad+ \int_{M_{R_{j_\ell}} \setminus M_{R_0}}\langle \rho^{-\la - 2} \eta, \rho^{-\la - 2} \psi_{j_\ell} \rangle \rho^{-7} d \\
        \implies  & | \langle \eta, \psi_{j_\ell} \rangle_{L^2_{\la + 2}(M_{R_{j_\ell}})} - \langle \eta, \psi_{j_\ell} \rangle_{L^2_{\la + 2}(M_{R_0})}| \\
        &\qquad \le R_0^{-\ve}\|\eta\|_{L^2_{\la + 2 -\ve}(M)}\|\psi_{j_\ell}\|_{L^2_{\la + 2}(M_{R_{j_\ell}})}\\ 
        &\qquad = O(R_0^{-\ve}). 
    \end{align*}
    Compute $\langle \eta, \psi_{j_\ell} \rangle_{L^2_{\la + 2}(M_{R_{j_\ell}})}$
    \begin{align*}
        \langle \eta, \psi_{j_\ell} \rangle_{L^2_{\la + 2}(M_{R_{j_\ell}})} &= \bigg\langle \eta, \sum_{i=1}^k a_i \sg_{R_{j_\ell},i}\psi_{R_{j_\ell}, i} \bigg\rangle_{L^2_{\la + 2}(M_{R_{j_\ell}})}  \\
        &=  \bigg\langle \eta, \sum_{i=1}^k a_i \sg_{R_{j_\ell},i}^{-1} \De^*\De \psi_{R_{j_\ell}, i} \bigg\rangle_{L^2_{\la + 2}(M_{R_{j_\ell}})}  \\
         &=  \sum_{i=1}^k a_i \sg_{R_{j_\ell}, i}^{-1} \bigg\langle \eta, \De^*\De \psi_{R_{j_\ell}, i} \bigg\rangle_{L^2_{\la + 2}(M_{R_{j_\ell}})} 
    \end{align*}
    To compute the above inner products, notice that $\eta$ and the $\psi_{R,i}$'s have sufficient regularity to do integration by parts as in Proposition \ref{proposition formal adjoint of laplacian}.
    \begin{align*}
        & \int_{M_R} \langle \eta, \rho^{2\la + 4 + 7}\De(\rho^{-2\la - 7} \De \psi_{R_{j_\ell}, i}) \rangle \rho^{-2\la -4 - 7} dV \\
        &\qquad = \int_{M_R} \langle \De \eta, \De \psi_{R_{j_\ell}, i}\rangle \rho^{-2\la}\rho^{-7} dV\\
        &\qquad \qquad + \int_{\dd M_R} \bt \de (\rho^{-2\la- 7} \De \psi_{R_{j_\ell}, i}) \wedge \star \bn \eta + \int_{\dd M_R} \bt \eta \wedge \star \bn d (\rho^{-2\la- 7} \De \psi_{R_{j_\ell}, i}) \\
        &\qquad \qquad \qquad - \int_{\dd M_R} \bt \de \eta \wedge \star \bn \rho^{-2\la- 7} \De \psi_{R_{j_\ell}, i} - \int_{\dd M_R} \bt \rho^{-2\la- 7} \De \psi_{R_{j_\ell}, i} \wedge \star \bn d \eta 
    \end{align*}
    Since $\De \eta = 0$, the integral over $M_R$ vanishes, leaving only the boundary terms. Furthermore, as $\De \psi_{R,i} = \sg_{R,i}^{-1} \phi_{R,i}$, we obtain
    \begin{align*}
        &\qquad = \sg_{R_{j_\ell},i}^{-1}\bigg(\int_{\dd M_R} \bt \de (\rho^{-2\la- 7} \phi_{R_{j_\ell}, i}) \wedge \star \bn \eta + \int_{\dd M_R} \bt \eta \wedge \star \bn d (\rho^{-2\la- 7} \phi_{R_{j_\ell}, i}) \\
        &\qquad \qquad \qquad - \int_{\dd M_R} \bt \de \eta \wedge \star \bn \rho^{-2\la- 7} \phi_{R_{j_\ell}, i} - \int_{\dd M_R} \bt \rho^{-2\la- 7} \phi_{R_{j_\ell}, i} \wedge \star \bn d \eta \bigg)
    \end{align*}
    Plug back into the original expression and obtain the sum
    \begin{align*}
        &\langle \eta, \psi_{j_\ell} \rangle_{L^2_{\la + 2}(M_{R_{j_\ell}})}\\
        &\qquad = \sum_{i=1}^k a_i\sg_{R_{j_\ell},i}^{-2}\bigg(\int_{\dd M_R} \bt \de (\rho^{-2\la- 7} \phi_{R_{j_\ell}, i}) \wedge \star \bn \eta + \int_{\dd M_R} \bt \eta \wedge \star \bn d (\rho^{-2\la- 7} \phi_{R_{j_\ell}, i}) \\
        &\qquad \qquad \qquad - \int_{\dd M_R} \bt \de \eta \wedge \star \bn \rho^{-2\la- 7} \phi_{R_{j_\ell}, i} - \int_{\dd M_R} \bt \rho^{-2\la- 7} \phi_{R_{j_\ell}, i} \wedge \star \bn d \eta \bigg)\\
        &\qquad = \int_{\dd M_R} \bt \de (\rho^{-2\la- 7} (\De \De^* \phi_{j_\ell})) \wedge \star \bn \eta + \int_{\dd M_R} \bt \eta \wedge \star \bn d (\rho^{-2\la- 7} (\De \De^* \phi_{j_\ell})) \\
        &\qquad \qquad \qquad - \int_{\dd M_R} \bt \de \eta \wedge \star \bn \rho^{-2\la- 7} (\De \De^* \phi_{j_\ell}) - \int_{\dd M_R} \bt \rho^{-2\la- 7} (\De \De^* \phi_{j_\ell}) \wedge \star \bn d \eta 
    \end{align*}
    Next, compute the following $L^2_\la$-norm estimate for $k \in \bN$.  
    \begin{align*}
    \bigg\| \sum_{i=1}^k a_i\sg_{R_{j_\ell},i}^{-2k} \phi_{R_{j_\ell}, i} \bigg\|_{L^2_\la (M_{R_{j_\ell}})}^2 &= \sum_{i=1}^k a_i^2 \sg_{R_{j_\ell},i}^{-4k} \\
    &\le \big( \sup_i \sg_{R_{j_\ell},i}^{-4k} \big) \|\phi_{R_{j_\ell}}\|^2 \\
    &\le C,
    \end{align*}
    where $C>0$ depends on $k$, $C_\la$, and the uniform bound on the $\phi_j$'s, but does not depend on $R_{j_\ell}$. This establishes uniform $L^2_\la$-bounds for $(\De \De^*)^k \phi_{j_\ell}$, and we can appeal to the higher order Schauder estimates in Corollary \ref{corollary lopatinski shapiro higher order} (see Proposition \ref{proposition uniform boundedness of projection} for details of this argument) to obtain uniform weighted bounds up to the boundary in higher order Sobolev spaces. 
    
    The $H^k_{\la}(M_R)$-norms of the $\phi_{j_\ell}$ and the $H^k_{\la +2 - \ve}(M)$-norm of $\eta$ are uniformly bounded, independent of $R$. Thus, the weighted Sobolev embedding theorem (Theorem \ref{theorem weighted sobolev embedding theorem}) implies the following pointwise decay conditions.
    \[
    \eta = O(\rho^{\la + 2 - \ve }), \;\;d\eta, \de \eta = O(\rho^{\la + 1 - \ve }),
    \]
    \[
    \De \De^* \phi_{j_\ell} = O(\rho^{\la + \ve/2}), \;\;d(\De \De^* \phi_{j_\ell}), \de(\De \De^* \phi_{j_\ell}) = O(\rho^{\la -1 +\ve/2 }).
    \]
    Adding the orders of decay, the integrand of each of the boundary terms will be $O(R^{-6-\ve/2})$. Since $\mathrm{Vol}(\dd M_R)$ approaches $R^6\mathrm{Vol}(\Sg)$, where $\Sg$ is the link of the asymptotic cone, the boundary terms are $O(R^{-\ve})$. This means that
    \begin{align*}
        \langle \eta, \psi_{j_\ell} \rangle_{L^2_{\la + 2}(M_{R_{j_\ell}})} &= O(R_\ell^{-\ve/2}).
    \end{align*}
    
    Combining this with the estimate $| \langle \eta, \psi_{j_\ell} \rangle_{L^2_{\la + 2}(M_{R_{j_\ell}})} - \langle \eta, \psi_{j_\ell} \rangle_{L^2_{\la + 2}(M_{R_0})}| = O(R_0^{-\ve})$ implies that 
    \[
    \langle \eta, \psi \rangle_{L^2_{\la + 2}(M_{R_0})} \leftarrow \langle \eta, \psi_{R_\ell} \rangle_{L^2_{\la + 2}(M_{R_0})} = o(1) \text{ as } R_0 \rightarrow \infty. 
    \]
    Since $\|\psi\|_{L^2_{\la + 2}}, \|\eta\|_{L^2_{\la + 2 - \ve}} < \infty$,
    \[
    \langle \eta, \psi \rangle_{L^2_{\la + 2}(M\setminus M_{R_0})} \le \|\psi\|_{L^2_{\la + 2} (M \setminus M_{R_0})} \|\eta\|_{L^2_{\la + 2}(M \setminus M_{R_0})} = o(1) \text{ as } R_0 \rightarrow \infty
    \]
    Since $R_0$ may be taken to be arbitrarily large, conclude that $\langle \eta, \psi \rangle_{L^2_{\la + 2}(M)} = 0$, i.e. that a non-zero limit $\psi$ must be perpendicular to the kernel of $\De$. 

    Observe that, given a compact set $K$ and sufficiently large $R_{j_\ell}$,
    \[
    \ga^2 \|\psi_{j_\ell}\|_{L^2_{\la+2}(K)} \le  \|\De^*\De \psi_{j_\ell}\|_{L^2_{\la+2}(K)},
    \]
    since the singular values $\sg_i^{-2}$ are bounded below by $\ga$. Taking the loc-limit,
    \[
    \ga^2 \| \psi \|_{L^2_{\la+2}(K)} \le  \|\De^*\De \psi \|_{L^2_{\la+2}(K)}.
    \]
    The monotone convergence theorem implies that this inequality holds for the norms on the whole space.
    \[
    \ga^2 \| \psi \|_{L^2_{\la+2}(M)} \le  \|\De^*\De \psi \|_{L^2_{\la+2}(M)}.
    \]
    Since $\ga^2 \le \frac{C_\la^{-1}}{2}$ and $\psi$ is perpendicular to $\ker(\De)_{\la+2}$ (and thus $\ker(\De^*\De)_{\la+2}$), this implies that implies that the operator norm $\|(\De^*\De)^{-1}\| \ge  \ga^{-2} > C_{\la}$,  which is a contradiction. 

    This concludes the proof of the claim.    
    \end{proof}
    
    Thus, if $\phi$ is the $H^k_{loc}$-limit of the $\phi_j$'s, then for any test function $\mu$
    \[
   |\langle \mu, \phi \rangle_{L^2_{\la}}| \le \limsup_j |\langle \mu, \phi_j \rangle_{L^2_{\la}}| = 0,
    \]
    as $j \rightarrow \infty$. This implies that $\phi \equiv 0$.
\end{proof}

\subsubsection{The space of ``small eigenforms"}
In this section, we show that the forms $\phi$ which are linear combinations of ``eigenforms" with very small singular values have \textit{adjoint} exterior derivative $\rho^{-2\la + 2 -7}d\rho^{-2\la -7}\phi$ and \textit{adjoint} co-differential $\rho^{-2\la + 2 -7}\de\rho^{-2\la -7} \phi$ bounded by the max of these small singular values. Thus, the loc-limit of any sequence of these forms will vanish under the weighted $L^2$-adjoint of the Dirac operator, $\rho^{-2\la + 2 -7}d\rho^{-2\la -7} + \rho^{-2\la + 2 -7}\de\rho^{-2\la -7} $.

\begin{lem}\label{lemma smallness of dirac}
    Let $\ve := -2\la -7 \ge 0$ be sufficiently small. Let
    \[
    V_{\ga,R} := \mathrm{span}\bigg\{\phi_{R,i}\;|\; 0 <  \sg_{R,i}^{-1}  < \ga  \bigg\}.
    \]
    Each element $\phi \in V_{\ga,R}$ has 
    \[
    \|\de (\rho^\ve \phi)\|_{L^2(M_R)}^2 + \|d (\rho^\ve \phi)\|_{L^2(M_R)}^2 \le \ga \|\phi\|^2_{L^2_{\la}(M_R)}.
    \]
\end{lem}

\begin{proof}
Let $\phi_i$ be one of the orthonormal basis vectors in $V_{\ga,R}$. Then
    \begin{align*}
        \langle \psi_j, \De^* \phi_i \rangle_{\la+2}  &= \sg_{R,j}^{-1} \de_{ij}. 
    \end{align*}
    By the integration-by-parts calculation in Proposition \ref{proposition formal adjoint of laplacian}, we have
    \begin{align*}
    \langle \psi_j, \De^* \phi_i \rangle_{\la} &= \langle \de \psi_j , \de (\rho^{-2\la - 7} \phi_i) \rangle_{L^2(M_R)} + \langle d \psi_j , d (\rho^{-2\la - 7} \phi_i) \rangle_{L^2(M_R)} \\
    &\qquad + \int_{\dd M_R} \bt \de(\rho^{-2\la - 7} \phi_i) \wedge \star \bn \psi_j.
    \end{align*}
    Corollary \ref{corollary vanishing of boundary terms} implies that $\bt \de (\rho^{-2\la - 7} \phi_i) = 0$. Therefore,
    \[
    \langle \de \psi_j , \de (\rho^{-2\la - 7} \phi_i) \rangle_{L^2(M_R)} + \langle d \psi_j , d (\rho^{-2\la - 7} \phi_i) \rangle_{L^2(M_R)} = \sg^{-1}_{R,j}\de_{ij}
    \]
    
    Let $\ve = -2\la - 7$ be such that $0 \le \ve \ll 1$. Then, the $\psi_j$'s span a larger space which includes each form $\rho^{\ve} \phi_i$. Thus, there is a linear combination 
    \[
    \rho^\ve \phi_i = \sum_{j=1}^\infty b_{i,j} \psi_j.
    \]
    Since $\rho^{\ve} \phi_i$ has unit norm in a space with harsher decay conditions, i.e. $L^2_{\la + \ve}$, it has a smaller norm with respect to the more lenient decay conditions in $L^2_{\la + 2}$.
    \begin{align*}
        1 = \|\rho^{\ve} \phi_i\|_{\la + \ve}^2 \ge \|\rho^\ve \phi_i\|_{\la + 2}^2
        &= \bigg\|\sum_{j=1}^\infty b_{i,j} \psi_j\bigg\|_{\la+2}^2 \\
        &= \sum_{j=1}^\infty b_{i,j}^2 \|\psi_j\|_{\la+2}^2 \\
        &= \sum_{j=1}^\infty b_{i,j}^2.
    \end{align*}
    Thus the coefficients $b_{i,j}$ of the decomposition cannot be larger than 1. Now consider a unit vector $\phi$ in $V_{\ga,R}$ with the $L^2_\la$-norm such that
    \[
    \phi = \sum_{i=1}^n a_i \phi_i, \;\text{ and } \sum_{i=1}^n a_i^2 = 1
    \]
    Now consider
    \[
    \rho^\ve \phi = \sum_{i=1}^n a_i \rho^\ve \phi_i.
    \]
    We calculate
    \begin{align*}
        \|\de (\rho^\ve \phi)\|_{L^2}^2 + \|d (\rho^\ve \phi)\|_{L^2}^2 &= \langle \de (a_i \rho^\ve \phi_i) , \de (a_i \rho^\ve \phi_i) \rangle_{L^2} + \langle d (a_i \rho^\ve \phi_i) , d (a_i \rho^\ve \phi_i) \rangle_{L^2} \\
        &= \langle a_i \de (b_{i,j} \psi_j) , \de (a_i \rho^\ve \phi_i) \rangle_{L^2} + \langle a_i d (b_{i,j} \psi_j) , d (a_i \rho^\ve \phi_i) \rangle_{L^2} \\
        &= \sum_{i,j} a_i^2 b_{i,j} \big( \langle  \de \psi_j , \de (\rho^\ve \phi_i) \rangle_{L^2}  + \langle  d \psi_j , d (\rho^\ve \phi_i) \rangle_{L^2}\big) \\
        &= \sum_{i,j} a_i^2 b_{i,j} \sg_{R,j}^{-1}\de_{ij} \\
        &= \sum_{i} a_i^2 b_{i,i} \sg_{R,i}^{-1} \\
        &\le \ga \sum_{i} a_i^2  \\
        &\le \ga
    \end{align*}
    Multiplying the coefficients $a_i$ by $\|\phi\|_{L^2_{\la}}$ for a non-unit $\phi$ completes the proof of the lemma.
\end{proof}

\subsubsection{Limits of sequences in the approximate kernels}

\begin{cor}\label{corollary small eigenvalues become harmonic}
    Let $\phi_j$ be an $H^{k+1}_{\la}$-bounded sequence of of vectors in the $L^2_{\la}$-orthogonal complements of $U_{R_j}$, the ``upper spectrum" of $\De^*\De$ on the domains $M_{R_j}$ (defined in Proposition \ref{proposition uniform boundedness of projection}). Let $\phi \in L^2_{\la}(M)$ be the $H^k_{loc}$-limit of the sequence $\{\phi_j\}$. Then, $\phi$ can be written as the sum of an exact and a co-exact form if and only if $\phi \equiv 0$.
\end{cor}

\begin{proof}
    The $H^{k+1}_{\la}$-boundedness of the $\phi_j$'s and Rellich-Kondrashov implies that they have a limit $\phi \in H^k_{loc}(M)$ (possibly up to a subsequence). As before, the extension lemma (Lemma \ref{lemma extension of functions}) and the weak lower semi-continuity of the $L^2_{\la}$-norm implies that $\phi \in L^2_{\la} \Om^3 (M)$. 

    Given $\ga$, each $\phi_j$ decomposes ($L^2_{\la}$-orthogonally) into 
    \[
    \phi_j = \phi_{j,\ga}^- + \phi_{j, \ga}^+ + \phi_{j,\ga}^D \in V_{\ga,R_j} \oplus W_{\ga, R_j} \oplus (\sH^3_D(M_R)^\perp)^{\perp_\la},
    \]
    where $(\sH^3_D(M_R)^\perp)^{\perp_\la}$ be the $L^2_\la$-orthogonal complement of $\sH^3_D(M_R)^\perp$.
    
    On a compact domain $M_{R}$, for any fixed $R > R_0$, the restriction $\phi |_{M_R}$ will be equal to the sum of the $H^k$-limits of $\phi_{j,\ga}^D$, $\phi_{j,\ga}^-$, and $\phi_{j, \ga}^+$ restricted to $M_R$. By Lemma \ref{lemma remainder vanishes on intermediate spec}, $\phi_{j, \ga}^+|_{M_R} \xrightarrow{H^k} 0$. By Claim \ref{claim adjoint harmonic field}, 
    \[
    \rho^{-\ve + 2} \de \rho^\ve \phi_{j,\ga}^D,\; \rho^{-\ve + 2} d \rho^\ve \phi_{j,\ga}^D = 0,
    \]
    Thus, by Lemma \ref{lemma smallness of dirac}, 
    \begin{align*}
        &\|\rho^{-\ve + 2} \de \rho^\ve \phi\|_{L^2(M_R)}^2 + \|\rho^{-\ve + 2} d \rho^\ve \phi\|_{L^2(M_R)}^2 \\ &\qquad \le \sup_j \big(\|\rho^{-\ve + 2} \de \rho^\ve \phi_{j,\ga}^-\|_{L^2(M_R)}^2 + \|\rho^{-\ve + 2} d \rho^\ve \phi_{j,\ga}^-\|_{L^2(M_R)}^2\big) \\
        &\qquad \le \ga \big(\sup_j \|\phi_{j,\ga}^-\|_{L^2_\la(M_R)}^2 \big)\\
        &\qquad \le \ga \big(\sup_j \|\phi_j \|_{L^2_\la(M_R)}^2 \big)\\
        &\qquad \le C \ga,
    \end{align*}
    since the sequence $\{ \phi_j\}_j$ is uniformly bounded in the $L^2_\la$-norm. Fix a radius $R>R_0$: this inequality is true for any choice of $\ga$ small, so
    \[
    \rho^{-\ve + 2} \de \rho^\ve \phi + \rho^{-\ve + 2} d \rho^\ve \phi = 0,
    \]
    We claim that $\rho^{-\ve + 2} \de \rho^\ve \phi + \rho^{-\ve + 2} d \rho^\ve \phi$ is the adjoint of the Dirac operator 
    \[
    d + \de : H^1_{\la + 1}\Om^2(M) \oplus H^1_{\la + 1}\Om^4(M) \rightarrow L^2_{\la + 1}\Om^3(M).
    \]
    The various properties of this operator (ellipticity, Fredholm, and so on) may be found in the comprehensive description in \cite[\S 4.2]{KarigiannisLotay}. The fact that the adjoint of $d + \de$ is $\rho^{-\ve + 2} \de \rho^\ve + \rho^{-\ve + 2} d \rho^\ve$ can be seen by a straightforward application of Green's theorem (c.f. Proposition \ref{proposition greens theorem}). Thus,
    \[
    \phi \in \ker (d + \de)^* = \big(\mathrm{ran}(d + \de)\big)^\perp.
    \]
    The range of $d + \de$ is precisely the closed subspace in $L^2_{\la}\Om^3(M)$ of forms with only exact and co-exact parts (see \cite[Lemma 4.30]{KarigiannisLotay}). Therefore, if $\phi$ has only exact and co-exact parts, it must vanish.
\end{proof}

\bibliographystyle{amsplain}
\bibliography{bibliography2}

\end{document}